\documentclass[oneside,english]{amsart}
\usepackage[T1]{fontenc}
\usepackage[latin9]{inputenc}
\usepackage{geometry}
\usepackage{xcolor}
\usepackage{babel}
\usepackage{array}
\usepackage{multirow}
\usepackage{amstext}
\usepackage{amsthm}
\usepackage{amssymb}
\usepackage[unicode=true,pdfusetitle,
 bookmarks=true,bookmarksnumbered=false,bookmarksopen=false,
 breaklinks=false,pdfborder={0 0 1},backref=false,colorlinks=false]
 {hyperref}
\usepackage{breakurl}

\makeatletter

\providecommand{\tabularnewline}{\\}

\numberwithin{equation}{section}
\numberwithin{figure}{section}
\newenvironment{lyxlist}[1]
{\begin{list}{}
{\settowidth{\labelwidth}{#1}
 \setlength{\leftmargin}{\labelwidth}
 \addtolength{\leftmargin}{\labelsep}
 }}
{\end{list}}
  \theoremstyle{plain}
  \newtheorem*{lem*}{\protect\lemmaname}
\theoremstyle{plain}
\newtheorem{thm}{\protect\theoremname}[section]
  \theoremstyle{plain}
  \newtheorem{prop}[thm]{\protect\propositionname}
  \theoremstyle{plain}
  \newtheorem{lem}[thm]{\protect\lemmaname}
  \theoremstyle{remark}
  \newtheorem{rem}[thm]{\protect\remarkname}
  \theoremstyle{plain}
  \newtheorem{cor}[thm]{\protect\corollaryname}

\makeatother

  \providecommand{\corollaryname}{Corollary}
  \providecommand{\lemmaname}{Lemma}
  \providecommand{\propositionname}{Proposition}
  \providecommand{\remarkname}{Remark}
\providecommand{\theoremname}{Theorem}

\begin{document}

\title{Fundamental Solution to 1D Degenerate Diffusion Equation with Locally
Bounded Coefficients}

\author{Linan Chen\textsuperscript{{*}} and Ian Weih-Wadman\textsuperscript{\dag{}}}
\begin{abstract}
In this work we study the degenerate diffusion equation $\partial_{t}=x^{\alpha}a\left(x\right)\partial_{x}^{2}+b\left(x\right)\partial_{x}$
for $\left(x,t\right)\in\left(0,\infty\right)^{2}$, equipped with
a Cauchy initial data and the Dirichlet boundary condition at $0$.
We assume that the order of degeneracy at 0 of the diffusion operator
is $\alpha\in\left(0,2\right)$, and both $a\left(x\right)$ and $b\left(x\right)$
are only locally bounded. We adopt a combination of probabilistic
approach and analytic method: by analyzing the behaviors of the underlying
diffusion process, we give an explicit construction to the fundamental
solution $p\left(x,y,t\right)$ and prove several properties for $p\left(x,y,t\right)$;
by conducting a localization procedure, we obtain an approximation
for $p\left(x,y,t\right)$ for $x,y$ in a neighborhood of 0 and $t$
sufficiently small, where the error estimates only rely on the local
bounds of $a\left(x\right)$ and $b\left(x\right)$ (and their derivatives).
There is a rich literature on such a degenerate diffusion in the case
of $\alpha=1$. Our work extends part of the existing results to cases
with more general order of degeneracy, both in the analysis context
(e.g., heat kernel estimates on fundamental solutions) and in the
probability view (e.g., wellposedness of stochastic differential equations). 
\end{abstract}


\address{{*}Department of Mathematics and Statistics, McGill University, 805
Sherbrooke St. West, Montreal, Quebec, H3A 0B9, Canada. }

\address{\dag Department of Mathematics and Statistics, McMaster University,
1280 Main St. West, Hamilton, Ontario, L8S 4K1, Canada. }

\email{{*}linan.chen@mcgill.ca }

\email{\dag weihwadi@mcmaster.ca}

\thanks{The first author was partially supported by the NSERC Discovery Grant
(No. 241023).}

\maketitle

\section{Introduction}

In this article we consider the following Cauchy initial value problem
with the Dirichlet boundary condition: 
\begin{equation}
\begin{array}{c}
\partial_{t}u_{f}\left(x,t\right)=x^{\alpha}a\left(x\right)\partial_{x}^{2}u_{f}\left(x,t\right)+b\left(x\right)\partial_{x}u_{f}\left(x,t\right)\text{ for }\left(x,t\right)\in\left(0,\infty\right)^{2},\\
\lim_{t\searrow0}u_{f}\left(x,t\right)=f\left(x\right)\text{ for }x\in\left(0,\infty\right)\text{ and }\lim_{x\searrow0}u_{f}\left(x,t\right)=0\text{ for }t\in\left(0,\infty\right),
\end{array}\label{eq:IVP general equation}
\end{equation}
where $f\in C_{c}\left(\left(0,\infty\right)\right)$ and $\alpha\in\left(0,2\right)$.
Set $L:=x^{\alpha}a\left(x\right)\partial_{x}^{2}+b\left(x\right)\partial_{x}$.
We further impose the following assumptions on $a\left(x\right)$
and $b\left(x\right)$:\\
\begin{lyxlist}{00.00.0000}
\item [{\textbf{(H1):}}] $a\in C\left([0,\infty)\right)\cap C^{2}\left(\left(0,\infty\right)\right)$,
$a\left(x\right)>0$ for every $x\in[0,\infty)$ and $a\left(0\right)=1$.
\\
$b\in C\left([0,\infty)\right)\cap C^{1}\left(\left(0,\infty\right)\right)$,
$b\left(0\right)\in[0,1)$ when $\alpha\leq1$, and $b\left(0\right)=0$
when $\alpha>1$. \\
$a\left(x\right)$, $a^{\prime}\left(x\right)$, $a^{\prime\prime}\left(x\right)$,
$b\left(x\right)$ and $b^{\prime}\left(x\right)$ are all bounded
on $(0,I]$ for every $I>0$. \\
\item [{\textbf{(H2):}}] There exists $C>0$ such that for every $x\in[0,\infty)$,
\[
a\left(x\right)\leq C\left(1+x^{2-\alpha}\right)\text{ and }\left|b\left(x\right)\right|\leq C\left(1+x\right).
\]
\end{lyxlist}
Our goal is to construct and to study the fundamental solution $p\left(x,y,t\right)$
to (\ref{eq:IVP general equation}), with a particular emphasis on
the behaviors of $p\left(x,y,t\right)$ when $x,y$ are near the boundary
and when $t$ is small. Since $L$ is degenerate at $0$, standard
methods on strictly parabolic equation no longer apply in this case,
and the degeneracy of $L$ does have an impact on the regularity of
$p\left(x,y,t\right)$. Moreover, the assumptions \textbf{(H1)} and
\textbf{(H2)} only guarantee local boundedness of $a\left(x\right)$,
$b\left(x\right)$ and their derivatives, so we need to conduct our
analysis of $p\left(x,y,t\right)$ only relying on the local bounds
of the coefficients. 

\subsection{Background and motivation}

Our work is primarily motivated by two previous works \cite{deg_diff_global}
and \cite{siamwfeq} on related problems. \cite{deg_diff_global}
treats  the initial/boundary value problem for a degenerate diffusion
equation similar to the one in (\ref{eq:IVP general equation}), but
under stronger conditions on the coefficients. To interpret the hypotheses
adopted in \cite{deg_diff_global} in terms of $\alpha$, $a\left(x\right)$
and $b\left(x\right)$ in our setting, we define the following two
functions for $x>0$:
\begin{equation}
\phi\left(x\right):=\frac{1}{4}\left(\int_{0}^{x}\frac{ds}{s^{\frac{\alpha}{2}}\sqrt{a\left(s\right)}}\right)^{2}\text{ and }\theta\left(x\right):=\frac{1}{2}-\nu+\frac{2b\left(x\right)-\left(x^{\alpha}a\left(x\right)\right)^{\prime}}{2x^{\frac{\alpha}{2}}\sqrt{a\left(x\right)}}\sqrt{\phi\left(x\right)},\label{eq:def of phi =000026 theta}
\end{equation}
where $\nu$ is the constant such that $\lim_{x\searrow0}\theta\left(x\right)=0$.
It is assumed in \cite{deg_diff_global} that $\nu<1$, as well as
\begin{equation}
\sup_{x\in(0,\infty)}\frac{\left|\theta\left(x\right)\right|}{\sqrt{\phi\left(x\right)}}<\infty\text{ and }\sup_{x\in\left(0,\infty\right)}\frac{\left|\theta^{\prime}\left(x\right)\right|}{\phi^{\prime}\left(x\right)}<\infty.\label{eq:assumptions in the global case}
\end{equation}
Under the assumption (\ref{eq:assumptions in the global case}), through
a series of transformations and perturbations, \cite{deg_diff_global}
completes a construction of the fundamental solution $p\left(x,y,t\right)$
to (\ref{eq:IVP general equation}), conducts a careful analysis of
the regularity properties of $p\left(x,y,t\right)$ near the boundary,
and derive an approximations for $p\left(x,y,t\right)$ in terms of
explicitly formulated functions. In particular, if $p^{approx.}\left(x,y,t\right)$
denotes the approximation for $p\left(x,y,t\right)$, then it is proven
in \cite{deg_diff_global} that there exists a constant $C>0$, universal
in all $x,y$ in a neighborhood of $0$ and all $t$ sufficiently
small, such that
\begin{equation}
\left|\frac{p\left(x,y,t\right)}{p^{approx.}\left(x,y,t\right)}-1\right|\leq Ct.\label{eq:estimate p/p^approx. global case}
\end{equation}

Such an estimate is useful in multiple ways. First, while one expects
$p\left(x,y,t\right)$ to resemble the fundamental solution to a strictly
parabolic equation for $x,y$ away from the boundary, (\ref{eq:estimate p/p^approx. global case})
captures accurately the asymptotics of $p\left(x,y,t\right)$ when
$x,y$ are close to the boundary, and demonstrates the influence of
the degeneracy of $L$ on $p\left(x,y,t\right)$. Second, if one could
apply the general heat kernel estimates (see, e.g., $\mathsection4$
of \cite{PDEStroock}) to $p\left(x,y,t\right)$, then one would get
that for every $\delta$ and $t$ sufficiently small, there is constant
$C_{\delta,t}>1$ such that 
\begin{equation}
C_{t,\delta}^{-1}\exp\left(-\frac{d\left(x,y\right)^{2}}{2\left(1-\delta\right)t}\right)\leq p\left(x,y,t\right)\leq C_{t,\delta}\exp\left(-\frac{d\left(x,y\right)^{2}}{2\left(1+\delta\right)t}\right),\label{eq:general heat kernel estimate}
\end{equation}
where $d\left(x,y\right)$ is the distance between $x$ and $y$ under
the Riemannian metric corresponding to $L$; it is clear that (\ref{eq:estimate p/p^approx. global case})
is a sharper estimate than (\ref{eq:general heat kernel estimate})
for small $t$, and hence $p^{approx.}\left(x,y,t\right)$ is a more
accurate short-term approximation for $p\left(x,y,t\right)$ compared
with the general heat kernel approximation. In addition, in \cite{deg_diff_global},
$p^{approx.}\left(x,y,t\right)$ is presented in an explicit formula
(in terms of special functions) and ``$Ct$'' in (\ref{eq:estimate p/p^approx. global case})
can be replaced by an exact expression; therefore, (\ref{eq:estimate p/p^approx. global case})
is easily accessible in computational applications that involve the
fundamental solution to any degenerate diffusion equation in the form
of $\partial_{t}-L=0$. 

We aim to generalize the results in \cite{deg_diff_global}, particularly
the construction of $p\left(x,y,t\right)$ and the short-term near-boundary
approximation $p^{approx.}\left(x,y,t\right)$, to a more general
family of degenerate diffusion equations. The hypotheses \textbf{(H1)}
and \textbf{(H2)} proposed above are more relaxed compared with the
assumption (\ref{eq:assumptions in the global case}) adopted in \cite{deg_diff_global}.
For example, it can be checked with direct computations that in general
(\ref{eq:assumptions in the global case}) does not hold if $b\left(0\right)\neq0$,
which means that, given \textbf{(H1)} and \textbf{(H2)}, (\ref{eq:assumptions in the global case})
is only satisfied when $1\leq\alpha<2$. Moreover, (\ref{eq:assumptions in the global case})
clearly imposes strong global conditions on $a\left(x\right)$ and
$b\left(x\right)$, but with \textbf{(H1)} and \textbf{(H2)}, we have
to find an access to $p\left(x,y,t\right)$ without relying on any
global bound on the coefficients. To tackle this issue, we invoke
a ``localization'' procedure, as inspired by \cite{siamwfeq}.

\cite{siamwfeq} studies the following well known Wright-Fisher diffusion
equation, which has its origin in population genetics:
\begin{equation}
\begin{array}{c}
\partial_{t}u_{f}\left(x,t\right)=x\left(1-x\right)\partial_{x}^{2}u_{f}\left(x,t\right)\text{ for }\left(x,t\right)\in\left(0,1\right)\times\left(0,\infty\right),\\
\lim_{t\searrow0}u_{f}\left(x,t\right)=f\left(x\right)\text{ for }x\in\left(0,1\right)\text{ for some }f\in C_{b}\left(\left(0,1\right)\right),\\
\text{ and }\lim_{x\searrow0}u_{f}\left(x,t\right)=\lim_{x\nearrow1}u_{f}\left(x,t\right)=0\text{ for }t\in\left(0,\infty\right).
\end{array}\label{eq: classical WF eq}
\end{equation}
Different from (\ref{eq:IVP general equation}), (\ref{eq: classical WF eq})
has two-sided Dirichlet boundaries at 0 and 1, and the diffusion operator
degenerates linearly at both boundaries. Set $L_{WF}:=x\left(1-x\right)\partial_{x}^{2}$
and let $p_{WF}\left(x,y,t\right)$ be the fundamental solution to
(\ref{eq: classical WF eq}). Since (\ref{eq: classical WF eq}) is
symmetric on $\left[0,1\right]$, to study $p_{WF}\left(x,y,t\right)$
near the boundaries, it is sufficient to only consider the left boundary
0. In \cite{siamwfeq}, a ``localization'' method is devised to
construct $p_{WF}\left(x,y,t\right)$ near 0: since $p_{WF}\left(x,y,t\right)$
can be viewed as the density of the underlying diffusion process corresponding
to $L_{WF}$, we can acquire information on $p_{WF}\left(x,y,t\right)$
by studying the behaviors of the process near 0; in particular, by
tracking the excursions of the diffusion process near 0, we can ``localize''
$L_{WF}$ and $p_{WF}\left(x,y,t\right)$ within a neighborhood of
0 where only the degeneracy at 0 has a substantial impact. Heuristically
speaking, when restricted near 0, $L_{WF}$ is close to the operator
$x\partial_{x}^{2}$, and hence it is natural to expect that $p_{WF}\left(x,y,t\right)$
with $x,y$ near 0 is close to the fundamental solution $p_{0}\left(x,y,t\right)$
to $\partial_{t}-x\partial_{x}^{2}=0$ (with Dirichlet boundary 0).
Indeed, it is established in \cite{siamwfeq} that, not only can $p_{WF}\left(x,y,t\right)$
be constructed in an explicit way via $p_{0}\left(x,y,t\right)$,
$p_{WF}\left(x,y,t\right)$ is also well approximated by $p_{0}\left(x,y,t\right)$
in the sense that $p_{WF}\left(x,y,t\right)/p_{0}\left(x,y,t\right)$
satisfies (\ref{eq:estimate p/p^approx. global case}) for $x,y$
near 0 and $t$ sufficiently small. In our work we want to adopt a
similar localization procedure and start our investigation of (\ref{eq:IVP general equation})
on a bounded set where the local bounds of the coefficients would
be sufficient for our purposes. 

In addition to treating directly the fundamental solutions, degenerate
diffusion equations in the form of $\partial_{t}-L=0$ have also been
discussed in many other contexts, with most of the existing literature
concerning the case when $\alpha=1$. For example, Epstein \emph{et
al }(\cite{EM11,EM13,EM14,FeynmanKac_HarnackIneq_deg_diff,tran_prob_deg_diff_manifold})
conduct an comprehensive study of the generalized Kimura operators,
which can be viewed as a generalization of $L$ with $\alpha=1$ in
the manifold setting, obtaining results such as the Hölder space of
the solutions, the maximum principle and the Harnack inequality. Related
works on generalized Kimura diffusions include \cite{FeynmanKac_HarnackIneq_deg_diff,tran_prob_deg_diff_manifold,C0_smooth_parab_Kimura_op,ExiUniq_Markov_Kimura_diff_singulardrift}.
From a probabilistic view, there are abundant theories on existence
and uniqueness of solutions to stochastic differential equations with
degenerate diffusion coefficients (see, e.g., \cite{singular_stochastic_diff_equa,Ethier76,est_dist_stoch_integral,deg_SDE_non-Lip_coeff,diff_conti_coeff,path_unique_SDE_non-Lip_coeff,transf_phasesp_diffproc_removedrift}\textcolor{purple}{{}
}and the references therein); when $\alpha=1$, a series of works
(see, e.g., \cite{AthreyaBarlowBassPerkins02,BassPerkins02,BassPerkins08,uniq_law_path_uniq_SDE,strong_Markov_locmart_solu_1D_SDE,Yamada-Watanabe})
provide conditions on $a\left(x\right)$ and $b\left(x\right)$ that
are sufficient for the stochastic differential equation corresponding
to $L$ to be well posed, and some of the results will also be used
later in our discussions. 

Degenerate diffusions have also been treated in the context of the
measure-valued process (see, e.g., \cite{wellposed_mart_deg_diff_dynamic_population,meas_Markov_proc,resolv_est_FV_uniq_marting,FV_processes_popul_gen,measure_value_proc_popul_gen,DW_suerproc_meas_diff}),
as well as via the semigroup approach (see, e.g., \cite{deg_evo_eq_reg_semigroup,deg_selfad_evoeq_unitint,C0_semigroup_diffop_ventcel_bdrycond,deg_parabolic_wentzel_bdrycond,highly_deg_parabolic_bvp,analy_semigroup_deg_ellip_operator_nonlinear_Cauchy}). 

\subsection{Our main results}

Our strategy in solving (\ref{eq:IVP general equation}) and getting
$p\left(x,y,t\right)$ is to combine the ideas and the techniques
from \cite{deg_diff_global} and \cite{siamwfeq}, and tackle the
two challenges we face: general order of degeneracy in $L$ at the
boundary, and lack of global bounds on the coefficients. Below we
briefly describe the main steps we will take to complete this work
(see Table 1 for an illustration). \\

\noindent \emph{1. Localization and transformation }($\mathsection2.1$).
Since the coefficients in $L$ are locally bounded, we first consider
a ``localized'' version of (\ref{eq:IVP general equation}). Given
$I>0$, we study the diffusion equation in (\ref{eq:IVP general equation})
on $\left(0,I\right)$ with an extra Dirichlet boundary at $I$, i.e.,
\[
\begin{array}{c}
\partial_{t}u\left(x,t\right)=Lu\left(x,t\right)\text{ for }\left(x,t\right)\in\left(0,I\right)\times\left(0,\infty\right)\\
\text{with }u\left(x,t\right)\rightarrow0\text{ as }x\searrow0\text{ or }x\nearrow I\text{ for }t\in\left(0,\infty\right).
\end{array}\quad\left(\star\right)
\]
Let $p_{I}\left(x,y,t\right)$ be the fundamental solution to $\left(\star\right)$.
To solve $\left(\star\right)$, we carry out a transformation that
turns $L$ into a diffusion operator that degenerates linearly at
$0$. In fact, with a change of variable $x\mapsto z$, solving $\left(\star\right)$
becomes equivalent to solving the following problem:
\[
\begin{array}{c}
\begin{array}{c}
\partial_{t}v^{V}\left(z,t\right)=\left(z\partial_{z}^{2}+\nu\partial_{z}+V\left(z\right)\right)v^{V}\left(z,t\right)\text{ for }\left(z,t\right)\in\left(0,J\right)\times\left(0,\infty\right)\\
\text{with }v^{V}\left(z,t\right)\rightarrow0\text{ as }z\searrow0\text{ or }z\nearrow J\text{ for }t\in\left(0,\infty\right),
\end{array}\end{array}\quad\left(\dagger\right)
\]
where $J$ is the image of $I$ after the change of variable, $\nu<1$
is a constant, and $V\left(z\right)$ is a function on $\left(0,J\right)$
($J$, $\nu$ and $V$ will be specified in $\mathsection2.1$). If
we can find the fundamental solution to $\left(\dagger\right)$, denoted
by $q_{J}^{V}\left(z,w,t\right)$, then $p_{I}\left(x,y,t\right)$
can be obtained through $q_{J}^{V}\left(z,w,t\right)$ via the transformation
(and its inverse) between $\left(\star\right)$ and $\left(\dagger\right)$.\\

\noindent \emph{2. Model equation }($\mathsection2.2$). Our strategy
for solving $\left(\dagger\right)$ is to treat the operator $z\partial_{z}^{2}+\nu\partial_{z}+V\left(z\right)$
as a perturbation of $z\partial_{z}^{2}+\nu\partial_{z}$. We temporarily
return to the ``global'' view, omit the potential $V\left(z\right)$
and the right boundary $J$, and consider the following \emph{model
equation} on the entire $\left(0,\infty\right)$:
\[
\begin{array}{c}
\partial_{t}v\left(z,t\right)=\left(z\partial_{z}^{2}+\nu\partial_{z}\right)v\left(z,t\right)\text{ for every }\left(z,t\right)\in\left(0,\infty\right)^{2}\\
\text{with }v\left(z,t\right)\rightarrow0\text{ as }z\searrow0\text{ for }t\in\left(0,\infty\right).
\end{array}
\]
This model equation has the advantage that its fundamental solution
$q\left(z,w,t\right)$ has an explicit formula in terms of a Bessel
function (\cite{deg_diff_global}), and properties of the solutions
to the model equation are already known to us (Proposition\ref{prop:results on q(z,w,t)}).
With $q\left(z,w,t\right)$ in hand, we return to the local view of
the model equation (with the Dirichlet boundary condition ``restored''
at $J$) and derive the fundamental solution $q_{J}\left(z,w,t\right)$
to the localized model equation on $\left(0,J\right)$ (Proposition
\ref{prop:properties of q_J}).\\

\noindent \emph{3. Solving the localized equation }($\mathsection3$).
Upon getting $q_{J}\left(z,w,t\right)$, we can start the construction
of the fundamental solutions to $\left(\dagger\right)$ and $\left(\star\right)$.
Viewing $z\partial_{z}^{2}+\nu\partial_{z}+V\left(z\right)$ as a
perturbation of $z\partial_{z}^{2}+\nu\partial_{z}$ with a potential
function $V\left(z\right)$, we invoke Duhamel's perturbation method
to construct $q_{J}^{V}\left(z,w,t\right)$ using $q_{J}\left(z,w,t\right)$
as the ``building block'' (Proposition \ref{prop:properties of q^V_J}).
Although in general $q_{J}^{V}\left(z,w,t\right)$ does not have a
closed-form formula and our representation of $q_{J}^{V}\left(z,w,t\right)$
is in the form of a series, by focusing on the first term of the series
expression we can show that $q_{J}^{V}\left(z,w,t\right)$ is well
approximated by $q\left(z,w,t\right)$ for sufficiently small $t$
(Proposition \ref{prop: approximation for q^V_J}). \\

\noindent \emph{4. Solving the global equation} ($\mathsection4$).
We finally return to (\ref{eq:IVP general equation}) and produce
$p\left(x,y,t\right)$ from $p_{I}\left(x,y,t\right)$ by ``reversing''
the localization procedure. More specifically, we establish the relation
between (\ref{eq:IVP general equation}) and its localized version
$\left(\star\right)$ with the help of the underlying diffusion process
corresponding to $L$. By analyzing the excursions of the diffusion
process over $\left(0,I\right)$, $p\left(x,y,t\right)$ is achieved
as the limit of $p_{I}\left(x,y,t\right)$ as $I$ increases to infinity
(Theorem \ref{thm:main theorem}). Again, although $p\left(x,y,t\right)$
does not have a closed-form formula, we find an approximation $p^{approx.}\left(x,y,t\right)$
for $p\left(x,y,t\right)$ such that $p^{approx.}\left(x,y,t\right)$
has an explicit and relatively simple expression, and $p^{approx.}\left(x,y,t\right)$
is more accurate than the standard heat kernel estimate for $p\left(x,y,t\right)$
(Theorem \ref{thm:approximation of p(x,y,t) }).

\begin{table}[h]
\begin{tabular}{lccl}
 &  & {\footnotesize{}(transformation)} & \tabularnewline
\emph{localized equations:} & $p_{I}\left(x,y,t\right)$ & $\longleftrightarrow$ & $q_{J}^{V}\left(z,w,t\right)$\tabularnewline
 & \multirow{3}{*}{{\footnotesize{}(convergence)}$\;\downarrow\quad\uparrow\;${\footnotesize{}(localization)}} &  & \multirow{2}{*}{$\quad\;\uparrow\;${\footnotesize{}(perturbation)}}\tabularnewline
 &  &  & \tabularnewline
 &  &  & $q_{J}\left(z,w,t\right)$\tabularnewline
\multicolumn{1}{r}{ - - - - - - - - - - - -} & \multicolumn{3}{c}{- - - - - - - - - - - - - - - - - - - - - - - - - - - - - - - - -
- - - - - - - - - - - - - -}\tabularnewline
\emph{global equations:} & $p\left(x,y,t\right)$ &  & \multirow{3}{*}{$\quad\;\uparrow\;${\footnotesize{}(localization)}}\tabularnewline
 & \multirow{2}{*}{{\footnotesize{}(approximation)}$\;\uparrow\quad\quad\quad\quad\quad\quad$} &  & \tabularnewline
 &  &  & \tabularnewline
 & $p^{approx.}\left(x,y,t\right)$ & $\longleftarrow$ & $q\left(z,w,t\right)$\tabularnewline
 &  & {\footnotesize{}(transformation)} & \tabularnewline
 &  &  & \tabularnewline
\end{tabular}\caption{Relation among the fundamental solutions.}
\end{table}

In each of the steps above, in addition to the standard analytic methods
from the study of parabolic equations, we also rely on a probabilistic
point of view towards diffusion equations. Whenever applicable, we
treat the fundamental solution as the transition probability density
function of the underlying diffusion process corresponding to the
concerned operator. In fact, the localization procedure (and the reverse
of it) proposed above is possible because of the (strong) Markov properties
of the diffusion process. We also invoke some classical tools in the
study of stochastic processes, e.g., Itô's formula and Doob's stopping
time theorem, in deriving probabilistic interpretations of the (fundamental)
solutions to the involved diffusion equations. In $\mathsection1.3$
we give a brief overview of the probabilistic components involved
in this work. 

In $\mathsection5$ we consider a generalization of the classical
Wright-Fisher equation (\ref{eq: classical WF eq}), where we assume
that the diffusion operator vanishes with a general order at both
of the degenerate boundaries $0$ and $1$. In particular, for $f\in C_{c}\left(\left(0,1\right)\right)$
and $\alpha,\beta\in\left(0,2\right)$, we consider the equation
\[
\begin{array}{c}
\partial_{t}u_{f}\left(x,t\right)=x^{\alpha}\left(1-x\right)^{\beta}\partial_{x}^{2}u_{f}\left(x,t\right)\text{ for }\left(x,t\right)\in\left(0,1\right)\times\left(0,\infty\right),\\
\lim_{t\searrow0}u_{f}\left(x,t\right)=f\left(x\right)\text{ for }x\in\left(0,1\right)\text{ and }\\
\lim_{x\searrow0}u_{f}\left(x,t\right)=\lim_{x\nearrow1}u_{f}\left(x,t\right)=0\text{ for }t\in\left(0,\infty\right).
\end{array}
\]
Although this problem is in a different setting from (\ref{eq:IVP general equation}),
our methods and results still apply. We can follow the same steps
as above to study its fundamental solution $p\left(x,y,t\right)$
and obtain similar estimates for $p\left(x,y,t\right)$ near either
of the boundaries (Proposition \ref{prop:property of p in two-sided bdry case}). 

\subsection{Stochastic differential equation, underlying diffusion process }

This subsection gives a brief overview of the probabilistic foundation
needed for our investigation. We start with the stochastic differential
equation corresponding to the operator $L=x^{\alpha}a\left(x\right)\partial_{x}^{2}+b\left(x\right)\partial_{x}$,
and that is, given $x>0$,
\begin{equation}
dX\left(x,t\right)=\sqrt{2X^{\alpha}\left(x,t\right)a\left(X\left(x,t\right)\right)}dB\left(t\right)+b\left(X\left(x,t\right)\right)dt\text{ for every }t\geq0\text{ with }X\left(0,x\right)\equiv x.\label{eq:SDE of X}
\end{equation}
For a general stochastic differential equation, there are two notions
of existence/uniqueness of a solution : strong existence/uniqueness
and weak existence/uniqueness. Our work only requires the existence
of a weak solution to (\ref{eq:SDE of X}) and the solution being
unique in the weak sense. We will not expand on the general theory
and refer interested readers to \cite{bm_stochc_calc,multi_dim_diff_proc}
for a comprehensive exposition on these topics. 

We say that (\ref{eq:SDE of X}) has a \emph{(weak) solution} if,
on some filtered probability space $\left(\Omega,\mathcal{F},\left\{ \mathcal{F}_{t}:t\geq0\right\} ,\mathbb{P}\right)$,
there exist two adapted processes $\left\{ B\left(t\right):t\geq0\right\} $
and $\left\{ X\left(x,t\right):t\geq0\right\} $ such that, (i) $\left\{ B\left(t\right):t\geq0\right\} $
is a standard Brownian motion; (ii) $\left\{ X\left(x,t\right):t\geq0\right\} $
has continuous sample paths; (iii) almost surely $\left\{ X\left(x,t\right):t\geq0\right\} $
satisfies that
\[
X\left(x,t\right)=x+\int_{0}^{t}\sqrt{2X^{\alpha}\left(x,s\right)a\left(X\left(x,s\right)\right)}dB\left(s\right)+\int_{0}^{t}b\left(X\left(x,s\right)\right)ds\text{ for every }t\geq0.
\]
In this case, we also refer to $\left\{ X\left(x,t\right):t\geq0\right\} $
as the \emph{underlying diffusion process} corresponding to $L$ starting
from $x$. We say that a solution $\left\{ X\left(x,t\right):t\ge0\right\} $
is \emph{unique (in law)}, if whenever (i)-(iii) are satisfied by
another triple $\left(\Omega^{\prime},\mathcal{F}^{\prime},\left\{ \mathcal{F}_{t}^{\prime}:t\geq0\right\} ,\mathbb{P}^{\prime}\right)$,
$\left\{ B^{\prime}\left(t\right):t\geq0\right\} $ and $\left\{ X^{\prime}\left(x,t\right):t\geq0\right\} $,
it must be that the distribution of $\left\{ X\left(x,t\right):t\geq0\right\} $
under $\mathbb{P}$ is identical with that of $\left\{ X^{\prime}\left(x,t\right):t\geq0\right\} $
under $\mathbb{P}^{\prime}$. We say that the stochastic differential
equation (\ref{eq:SDE of X}) is \emph{well posed} if a solution exists
and is unique. 

In later discussions we will use an important corollary of the wellposedness
property, and that is, if (\ref{eq:SDE of X}) is well posed and $\left\{ X\left(x,t\right):t\geq0\right\} $
is the unique solution, then $\left\{ X\left(x,t\right):t\geq0\right\} $
is a strong Markov process and for every $H\in C\left([0,\infty)^{2}\right)\cap C^{2,1}\left(\left(0,\infty\right)^{2}\right)$,
\[
\left\{ H\left(X\left(x,t\right),t\right)-\int_{0}^{t}\left(\partial_{s}+L\right)H\left(X\left(x,s\right),s\right)ds:t\geq0\right\} 
\]
is a local martingale. 

Now let us examine the wellposedness of (\ref{eq:SDE of X}) under
the hypotheses \textbf{(H1)} and \textbf{(H2)}. There is a rich literature
on the wellposedness of a degenerate stochastic differential equation
with a diffusion operator that degenerates linearly. While the diffusion
coefficient in (\ref{eq:SDE of X}) may have nonlinear degeneracy,
we can convert it into a linear degeneracy case simply through a change
of variable. To be specific, we consider the following diffeomorphism
on $\left(0,\infty\right)$ and its inverse:
\begin{equation}
\xi=\xi\left(x\right):=\frac{x^{2-\alpha}}{\left(2-\alpha\right)^{2}}\text{ and }x=x\left(\xi\right):=\left(2-\alpha\right)^{\frac{2}{2-\alpha}}\xi^{\frac{1}{2-\alpha}}\text{ for }x,\xi>0.\label{eq:change of variable priliminary}
\end{equation}
One can easily verify that $u_{f}\left(x,t\right)\in C^{2,1}\left(\left(0,\infty\right)\right)$
is a solution to (\ref{eq:IVP general equation}) if and only if $w_{g}\left(\xi,t\right):=u_{f}\left(x\left(\xi\right),t\right)$
is the solution to 
\begin{equation}
\begin{array}{c}
\partial_{t}w_{g}\left(\xi,t\right)=\xi c\left(\xi\right)\partial_{\xi}^{2}w_{g}\left(\xi,t\right)+d\left(\xi\right)\partial_{\xi}w_{g}\left(\xi,t\right)\text{ for }\left(\xi,t\right)\in\left(0,\infty\right)^{2},\\
\lim_{t\searrow0}w_{g}\left(\xi,t\right)=g\left(\xi\right)\text{ for }\xi\in\left(0,1\right)\text{ and }\lim_{\xi\searrow0}w_{g}\left(\xi,t\right)=0\text{ for }t\in\left(0,\infty\right),
\end{array}\label{eq:middle step IVP with two sided boundaries}
\end{equation}
where $g\left(\xi\right):=f\left(x\left(\xi\right)\right)$, $c\left(\xi\right):=a\left(x\left(\xi\right)\right)$
and 
\[
d\left(\xi\right):=\frac{1-\alpha}{2-\alpha}a\left(x\left(\xi\right)\right)+\frac{\left(x\left(\xi\right)\right)^{1-\alpha}}{2-\alpha}b\left(x\left(\xi\right)\right).
\]
The stochastic differential equation corresponding to (\ref{eq:middle step IVP with two sided boundaries})
is that, given $\xi>0$,
\begin{equation}
dZ\left(\xi,t\right)=\sqrt{2Z\left(\xi,t\right)c\left(Z\left(\xi,t\right)\right)}dB\left(t\right)+d\left(Z\left(\xi,t\right)\right)dt\text{ for every }t\geq0\text{ with }Z\left(\xi,0\right)\equiv\xi.\label{eq:SDE of Z}
\end{equation}

Assuming \textbf{(H1)} and \textbf{(H2)}, we get down to verifying
the wellposedness of (\ref{eq:SDE of Z}) where the diffusion operator
degenerates linearly at $0$. First, when $\alpha\in\left(1,2\right)$,
by\textbf{ (H1)}, (\ref{eq:change of variable priliminary}) and direct
computations, we see that both $c\left(\xi\right)$ and $d\left(\xi\right)$
are Lipschitz continuous on any bounded subset of $\left(0,\infty\right)$.
Furthermore, \textbf{(H2) }and (\ref{eq:change of variable priliminary})
imply that there exists constant $C>0$ such that for every $\xi\in[0,\infty)$,
\begin{equation}
\left|c\left(\xi\right)\right|+\left|d\left(\xi\right)\right|\leq C\left(1+\left(x\left(\xi\right)\right)^{2-\alpha}\right)\leq C\left(1+\xi\right).\label{eq:linear growth of the coefficients}
\end{equation}
It follows from classical results (e.g., Yamada-Watanabe \cite{Yamada-Watanabe},
Stroock-Varadhan \cite{multi_dim_diff_proc}, Engelbert-Schmidt \cite{strong_Markov_locmart_solu_1D_SDE},
Cherny \cite{uniq_law_path_uniq_SDE}) that (\ref{eq:SDE of Z}) is
well posed for every $\xi>0$ in this case. Next, when $\alpha\in(0,1]$,
we note that
\[
\lim_{\xi\searrow0}d\left(\xi\right)=\lim_{x\searrow0}\left(\frac{1-\alpha}{2-\alpha}a\left(x\right)+\frac{x^{1-\alpha}}{2-\alpha}b\left(x\right)\right)\geq0.
\]
This time \textbf{(H1)} and (\ref{eq:change of variable priliminary})
guarantee that $c\left(\xi\right)$ and $d\left(\xi\right)$ are both
Hölder continuous on any bounded subset of $\left(0,\infty\right)$;
meanwhile, the growth control (\ref{eq:linear growth of the coefficients})
on $c\left(\xi\right)$ and $d\left(\xi\right)$ still applies. Thus,
the results of Bass-Perkins \cite{BassPerkins02} lead to the wellposedness
of (\ref{eq:SDE of Z}) for every $\xi>0$. Therefore, for every $\alpha\in\left(0,2\right)$,
\textbf{(H1)} and \textbf{(H2)} are sufficient for (\ref{eq:SDE of Z})
to be well posed. Assume that $\left\{ Z\left(\xi,t\right):t\ge0\right\} $
is the unique solution to (\ref{eq:SDE of Z}). By setting 
\[
X\left(x,t\right):=x\left(Z\left(\xi\left(x\right),t\right)\right)\text{ for }t\geq0,
\]
we immediately get the following conclusion.
\begin{lem*}
\label{lem:existence of X(x,t)}The stochastic differential equation
(\ref{eq:SDE of X}) is well posed for every $x>0$, $\left\{ X\left(x,t\right):t\geq0\right\} $
defined above is the unique solution to (\ref{eq:SDE of X}), and
$\left\{ X\left(x,t\right):t\geq0\right\} $ is a strong Markov process.
Moreover, if $u\left(x,t\right)\in C^{2,1}\left(\left(0,\infty\right)^{2}\right)$
is a solution to $\partial_{t}u\left(x,t\right)=Lu\left(x,t\right)$,
then given any $\left(x,t\right)\in\left(0,\infty\right)^{2}$, $\left\{ u\left(X\left(x,s\right),t-s\right):s\in\left[0,t\right]\right\} $
is a local martingale.
\end{lem*}
So far there is no constraint on the behavior of $X\left(x,t\right)$
at the boundary 0. Returning to the original problem (\ref{eq:IVP general equation}),
to incorporate the Dirichlet boundary condition, we only need to focus
on $X\left(x,t\right)$ up to the time it hits $0$. Intuitively speaking,
if we set

\[
\zeta_{0}^{X}\left(x\right):=\inf\left\{ s\geq0:X\left(x,s\right)=0\right\} \text{ for }x>0,
\]
then the probability density function of the conditional distribution
of $X\left(x,t\right)$ given $\left\{ t<\zeta_{0}^{X}\left(x\right)\right\} $
should coincide with the fundamental solution to (\ref{eq:IVP general equation}).

\subsection*{Notations.}

\noindent For $\alpha,\beta\in\mathbb{R}$, we write $\alpha\vee\beta:=\max\left\{ \alpha,\beta\right\} $
and $\alpha\wedge\beta:=\min\left\{ \alpha,\beta\right\} $. 

\noindent For every $\Gamma\subseteq[0,\infty)$, $\mathbb{I}_{\Gamma}$
denotes the indicator function of $\Gamma$. 

\noindent Let $\left(\Omega,\mathcal{F},\left\{ \mathcal{F}_{t}:t\geq0\right\} ,\mathbb{P}\right)$
be a filtered probability space. For an integrable random variable
$X$ on $\Omega$ and a set $A\in\mathcal{F}$, we write $\mathbb{E}\left[X;A\right]:=\int_{A}Xd\mathbb{P}$.
For an adapted process $\left\{ W\left(t\right):t\geq0\right\} $
with non-negative continuous sample paths, we set 
\[
\zeta_{y}^{W}\left(x\right):=\inf\left\{ t\geq0:W\left(t\right)=y|W\left(0\right)=x\right\} \text{ for every }x,y\geq0,
\]
i.e., $\zeta_{y}^{W}\left(x\right)$ is the hitting time at $y$ conditioning
on the process starting from $x$; for $y_{1},y_{2}\geq0$, we set
$\zeta_{y_{1},y_{2}}^{W}\left(x\right):=\zeta_{y_{1}}^{W}\left(x\right)\wedge\zeta_{y_{2}}^{W}\left(x\right)$.

\section{Model Equation}

In this section, we will carry out the first three steps outlined
in $\mathsection1.2$. Although we use similar transformations as
those in \cite{deg_diff_global}, we need to adapt the method so that
it applies to $a\left(x\right)$ and $b\left(x\right)$ that are under
weaker conditions.

\subsection{Localization and Transformation}

Let $\alpha\in\left(0,2\right)$ and $I>0$ be fixed throughout this
section. Our first step is to introduce an extra Dirichlet boundary
to the equation $\partial_{t}-L=0$ at $x=I$ and to consider a localized
version of (\ref{eq:IVP general equation}) on $\left(0,I\right)$.
Namely, given $f\in C_{c}\left(\left(0,I\right)\right)$, we look
for $u_{f,\left(0,I\right)}\left(x,t\right)\in C^{2,1}\left(\left(0,I\right)\times\left(0,\infty\right)\right)$
such that
\begin{equation}
\begin{array}{c}
\partial_{t}u_{f,\left(0,I\right)}\left(x,t\right)=Lu_{f,\left(0,I\right)}\left(x,t\right)\text{ for }\left(x,t\right)\in\left(0,I\right)\times\left(0,\infty\right),\\
\lim_{t\searrow}u_{f,\left(0,I\right)}\left(x,t\right)=f\left(x\right)\text{ for }x\in\left(0,I\right)\\
\text{ and }\lim_{x\searrow0}u_{f,\left(0,I\right)}\left(x,t\right)=\lim_{x\nearrow I}u_{f,\left(0,I\right)}\left(x,t\right)=0\text{ for }t\in\left(0,\infty\right),
\end{array}\label{eq:general IVP localized}
\end{equation}
Once restricted on $\left(0,I\right)$, the coefficients (and their
derivatives) in (\ref{eq:general IVP localized}) are all bounded. 

We want to find the fundamental solution $p_{I}\left(x,y,t\right)$
to (\ref{eq:general IVP localized}). Given $x>0$, let $\left\{ X\left(x,t\right):t\geq0\right\} $
be the unique solution to (\ref{eq:SDE of X}), as found in $\mathsection1.3$.
We expect that $y\mapsto p_{I}\left(x,y,t\right)$ coincides with
the probability density function of $\left\{ X\left(x,t\right):t\geq0\right\} $,
conditioning on $\left\{ t<\zeta_{0,I}^{X}\left(x\right)\right\} $.
This probabilistic interpretation of $p_{I}\left(x,y,t\right)$ is
indeed correct and will be justified later. For now, let us conduct
an analysis of (\ref{eq:general IVP localized}) via standard perturbation
methods. 

As mentioned in $\mathsection1.2$, we will transform (\ref{eq:IVP general equation})
into a diffusion equation that has linear degeneracy at $0$. For
$x\in(0,I]$, let $\phi\left(x\right)$ and $\theta\left(x\right)$
be defined as in (\ref{eq:def of phi =000026 theta}). It is clear
that $\phi\in C^{2}\left(\left(0,I\right)\right)$, $\phi$ is strictly
increasing, and $\theta\in C^{1}\left(\left(0,I\right)\right)$. The
constant $\nu$ in the definition of $\theta\left(x\right)$ is chosen
such that
\[
\nu=\frac{1}{2}+\lim_{x\searrow0}\frac{2b\left(x\right)-\left(\alpha x^{\alpha-1}a\left(x\right)+x^{\alpha}a^{\prime}\left(x\right)\right)}{2x^{\frac{\alpha}{2}}\sqrt{a\left(x\right)}}\sqrt{\phi\left(x\right)}.
\]
Under \textbf{(H1)} and \textbf{(H2)}, it is easy to verify that
\[
\nu=\frac{1-\alpha}{2-\alpha}\mathbb{I}_{\left(0,2\right)\backslash\left\{ 1\right\} }\left(\alpha\right)+b\left(0\right)\mathbb{I}_{\left\{ 1\right\} }\left(\alpha\right),
\]
and hence $\nu<1$. Let $J:=\phi\left(I\right)$, $\psi:(0,J]\rightarrow(0,I]$
be the inverse function of $\phi$ and $\tilde{\theta}:=\theta\circ\psi$.
We introduce two more functions on $(0,J]$:
\begin{equation}
\Theta:\;z\in(0,J]\mapsto\Theta\left(z\right):=\exp\left(-\int_{0}^{z}\frac{\tilde{\theta}(u)}{2u}du\right),\label{eq:def of Theta}
\end{equation}
and
\[
V:\;z\in(0,J]\mapsto V\left(z\right):=z\frac{\Theta^{\prime\prime}\left(z\right)}{\Theta\left(z\right)}+\left(\nu+\tilde{\theta}\left(z\right)\right)\frac{\Theta^{\prime}\left(z\right)}{\Theta\left(z\right)},
\]
or equivalently,
\begin{equation}
V\left(z\right)=-\frac{\tilde{\theta}^{2}\left(z\right)}{4z}-\frac{\tilde{\theta}^{\prime}\left(z\right)}{2}+\left(1-\nu\right)\frac{\tilde{\theta}\left(z\right)}{2z}.\label{eq:def of V}
\end{equation}

Now we are ready to state the result on the transformation.
\begin{prop}
\label{prop:change of variable} Given $f\in C_{c}\left(\left(0,I\right)\right)$,
we define 
\[
h\left(z\right):=\frac{f\circ\psi\left(z\right)}{\Theta\left(z\right)}\text{ for }z\in\left(0,J\right).
\]
Then, $u_{f,\left(0,I\right)}\left(x,t\right)\in C^{2,1}\left(\left(0,I\right)\times\left(0,\infty\right)\right)$
is a solution to (\ref{eq:general IVP localized}) if and only if
\begin{equation}
u_{f,\left(0,I\right)}\left(x,t\right)=\Theta\left(\phi\left(x\right)\right)v_{h,\left(0,J\right)}^{V}\left(\phi\left(x\right),t\right)\text{ for every }\left(x,t\right)\in\left(0,I\right)\times\left(0,\infty\right),\label{eq:u_f loc defined through v^V loc}
\end{equation}
where $v_{h,\left(0,J\right)}^{V}\left(z,t\right)\in C^{2,1}\left(\left(0,J\right)\times\left(0,\infty\right)\right)$
is a solution to the following problem:
\begin{equation}
\begin{array}{c}
\begin{array}{c}
\partial_{t}v_{h,\left(0,J\right)}^{V}\left(z,t\right)=\left(z\partial_{z}^{2}+\nu\partial_{z}+V\left(z\right)\right)v_{h,\left(0,J\right)}^{V}\left(z,t\right)\text{ for }\left(z,t\right)\in\left(0,J\right)\times\left(0,\infty\right),\\
\lim_{t\searrow0}v_{h,\left(0,J\right)}^{V}\left(z,t\right)=h\left(z\right)\text{ for }z\in\left(0,J\right)\text{ and }\\
\lim_{z\searrow0}v_{h,\left(0,J\right)}^{V}\left(z,t\right)=\lim_{z\nearrow J}v_{h,\left(0,J\right)}^{V}\left(z,t\right)=0\text{ for }t\in\left(0,\infty\right).
\end{array}\end{array}\label{eq:model equation with potential localized}
\end{equation}
\end{prop}

We omit the proof of Proposition \ref{prop:change of variable} since
it can be verified by direct computations. If $q_{J}^{V}\left(z,w,t\right)$
is the fundamental solution to (\ref{eq:model equation with potential localized}),
then $p_{I}\left(x,y,t\right)$ is connected with $q_{J}^{V}\left(z,w,t\right)$
following the same relation as the one in (\ref{eq:u_f loc defined through v^V loc}).
Set $L^{V}:=z\partial_{z}^{2}+\nu\partial_{z}+V\left(z\right)$. Compared
with $L$, $L^{V}$ has a simpler structure consisting of a linear
diffusion, a constant drift and a potential. In the next subsection
we will solve (\ref{eq:model equation with potential localized})
by treating $L^{V}$ as a perturbation of $L_{0}:=z\partial_{z}^{2}+\nu\partial_{z}$
and invoking Duhamel's perturbation method. As a preparation, we state
below some technical results on $\Theta\left(z\right)$ and $V\left(z\right)$. 
\begin{lem}
\label{lem:estimates on VJ}Let $\Theta\left(z\right)$ be defined
as in (\ref{eq:def of Theta}). Then, for every $z\in(0,J]$,
\begin{equation}
\Theta\left(z\right)=\begin{cases}
\left(\psi\left(z\right)\right)^{\frac{\alpha}{4}}\left(4z\right)^{-\frac{\alpha}{4\left(2-\alpha\right)}}\left(a\left(\psi\left(z\right)\right)\right)^{\frac{1}{4}}\exp\left(-\int_{0}^{\psi\left(z\right)}\frac{b\left(w\right)}{2w^{\alpha}a\left(w\right)}dw\right) & \text{ if }\alpha\neq1,\\
\left(\frac{\psi\left(z\right)}{4z}\right)^{\frac{1}{4}-\frac{b_{0}}{2}}\left(a\left(\psi\left(z\right)\right)\right)^{\frac{1}{4}}\exp\left(-\int_{0}^{\psi\left(z\right)}\frac{1}{2w}\left(\frac{b\left(w\right)}{a\left(w\right)}-b\left(0\right)\right)dw\right) & \text{ if }\alpha=1.
\end{cases}\label{eq:formula of Theta}
\end{equation}
Hence, there exists constant $\Theta_{J}>0$ that can be made explicit
(see (\ref{eq:formula of Theta_J}) in the Appendix) such that 
\begin{equation}
\sup_{z\in\left(0,J\right)}\left(\Theta\left(z\right)\vee\frac{1}{\Theta\left(z\right)}\right)\leq\Theta_{J}.\label{eq:bound Theta_J}
\end{equation}

Let $V\left(z\right)$ be defined as in (\ref{eq:def of V}). Then,
there exists constant $V_{J}>0$ such that for every $z\in(0,J]$,
\begin{equation}
\left|V\left(z\right)\right|\leq\begin{cases}
V_{J}\cdot z^{-\frac{1}{2-\alpha}} & \text{ if }\alpha\in\left(0,1\right)\text{ and }b\left(0\right)\neq0,\\
V_{J}\cdot z^{-\frac{1-\alpha}{2-\alpha}} & \text{ if }\alpha\in\left(0,1\right)\text{ and }b\left(0\right)=0,\\
V_{J} & \text{ if }\alpha\in[1,2).
\end{cases}\label{eq:estimate of V}
\end{equation}
\end{lem}

The proof of Lemma \ref{lem:estimates on VJ} is left in the Appendix
since it is based on straightforward computations that are lengthy
and not crucial to our work. We note that when $\alpha\in\left(0,1\right)$,
the potential function $V\left(z\right)$ may be singular at $0$.
This is a generalization of the case considered in \cite{deg_diff_global}
where $V\left(z\right)$ is assumed to be bounded near 0. 

\subsection{From $q\left(z,w,t\right)$ to $q_{J}\left(z,w,t\right)$}

Let $I$ and $J$ be the same as above. As mentioned in the previous
subsection, to solve (\ref{eq:model equation with potential localized}),
we will first consider the analogous problem with $L^{V}$ replaced
by $L_{0}$. Namely, given $g\in C_{c}\left(\left(0,I\right)\right)$,
we look for $v_{g,\left(0,J\right)}\left(z,t\right)\in C^{2,1}\left(\left(0,J\right)\times\left(0,\infty\right)\right)$
such that 
\begin{equation}
\begin{array}{c}
\partial_{t}v_{g,\left(0,J\right)}\left(z,t\right)=L_{0}v_{g,\left(0,J\right)}\left(z,t\right)\text{ for every }\left(z,t\right)\in\left(0,J\right)\times\left(0,\infty\right)\\
\lim_{t\searrow0}v_{g,\left(0,J\right)}\left(z,t\right)=g\left(z\right)\text{ for }z\in\left(0,J\right)\\
\text{ and }\lim_{z\searrow0}v_{g,\left(0,J\right)}\left(z,t\right)=\lim_{z\nearrow J}v_{g,\left(0,J\right)}\left(z,t\right)=0\text{ for }t\in\left(0,\infty\right).
\end{array}\label{eq:model equation localized}
\end{equation}
Let $q_{J}\left(z,w,t\right)$ be the fundamental solution to (\ref{eq:model equation localized}).

We consider $\partial_{t}-L_{0}=0$ as our \emph{model equation}.
To solve (\ref{eq:model equation localized}), we temporarily return
to the ``global'' view and study the model equation on $\left(0,\infty\right)$
instead of $\left(0,J\right)$. That is, for $g\in C_{c}\left(\left(0,\infty\right)\right)$,
we consider the following problem:
\begin{equation}
\begin{array}{c}
\partial_{t}v_{g}\left(z,t\right)=L_{0}v_{g}\left(z,t\right)\text{ for every }\left(z,t\right)\in\left(0,\infty\right)^{2}\\
\lim_{t\searrow0}v_{g}\left(z,t\right)=g\left(z\right)\text{ for }z\in\left(0,\infty\right)\text{ and }\lim_{z\searrow0}v_{g}\left(z,t\right)=0\text{ for }t\in\left(0,\infty\right).
\end{array}\label{eq:model equation}
\end{equation}
Let $q\left(z,w,t\right)$ be the fundamental solution to (\ref{eq:model equation}).
In fact, $q\left(z,w,t\right)$ is the starting point of our ``journey'',
and from $q\left(z,w,t\right)$ we will derive the (fundamental) solutions
to all the concerned equations. 

The stochastic differential equation corresponding to the model equation
is that, given $z>0$,
\begin{equation}
dY\left(z,t\right)=\sqrt{2Y\left(z,t\right)}dB\left(t\right)+\nu dt\text{ for }t\geq0\text{ with }Y\left(z,0\right)\equiv z.\label{eq:SDE of Y}
\end{equation}
It follows from the discussions in $\mathsection1.2$ that (\ref{eq:SDE of Y})
is well posed, and hence there exists a unique solution $\left\{ Y\left(z,t\right):t\geq0\right\} $
that is also a strong Markov process. 
\begin{rem}
\label{rem:boundary classification}We want to remark that, independent
of the Dirichlet boundary condition imposed in (\ref{eq:model equation}),
the constant $\nu$ determines the \emph{attainability} of the boundary
$0$. Under \textbf{(H1)} and \textbf{(H2)}, we have that $\nu<1$,
and hence $0$ is either an exit boundary or a regular boundary. This
is to say that, no matter what $z$ is, $\left\{ Y\left(z,t\right):t\geq0\right\} $
hits $0$ with a positive probability in finite time. For more details
on the topic of boundary classification, we refer readers to $\mathsection15$
of \cite{Karlin_Taylor}. 
\end{rem}

The operator $L_{0}$, as well as (\ref{eq:model equation}) and (\ref{eq:model equation}),
has been well studied in \cite{deg_diff_global}. Below we will review
some useful facts about $q\left(z,w,t\right)$, $v_{g}\left(z,t\right)$
and their connections to $\left\{ Y\left(z,t\right):t\geq0\right\} $.
The details can be found in $\mathsection2$ of \cite{deg_diff_global}. 
\begin{prop}
\label{prop:results on q(z,w,t)} (Proposition 2.1, 2.3 of \cite{deg_diff_global})
The fundamental solution to (\ref{eq:model equation}) is 
\begin{equation}
q\left(z,w,t\right):=\frac{z^{\frac{1-\nu}{2}}w^{\frac{\nu-1}{2}}}{t}e^{-\frac{z+w}{t}}I_{1-\nu}\left(2\frac{\sqrt{zw}}{t}\right)=\frac{z^{1-\nu}}{t^{2-\nu}}e^{-\frac{z+w}{t}}\sum_{n=0}^{\infty}\frac{\left(zw\right)^{n}}{t^{2n}n!\Gamma\left(n+2-\nu\right)}\label{eq:def of q}
\end{equation}
for $\left(z,w,t\right)\in\left(0,\infty\right)^{3}$, where $I_{1-\nu}$
is the modified Bessel function. $q\left(z,w,t\right)$ is smooth
on $\left(0,\infty\right)^{3}$, and for every $\left(z,w,t\right)\in\left(0,\infty\right)^{3}$,
\begin{equation}
\frac{z^{1-\nu}}{t^{2-\nu}\Gamma\left(2-\nu\right)}e^{-\frac{z+w}{t}}\leq q\left(z,w,t\right)\leq\left(\frac{z^{1-\nu}}{t^{2-\nu}}\right)e^{-\frac{\left(\sqrt{z}-\sqrt{w}\right)^{2}}{t}},\label{eq: estimate of q}
\end{equation}
and
\begin{equation}
w^{1-\nu}q\left(z,w,t\right)=z^{1-\nu}q\left(w,z,t\right).\label{eq:symmetry of q}
\end{equation}

Given $g\in C_{c}\left(\left(0,\infty\right)\right)$, if
\[
v_{g}\left(z,t\right):=\int_{0}^{\infty}g\left(w\right)q\left(z,w,t\right)dw\text{ for }\left(z,t\right)\in\left(0,\infty\right),
\]
then $v_{g}\left(z,t\right)$ is the unique solution in $C^{2,1}\left(\left(0,\infty\right)^{2}\right)$
to (\ref{eq:model equation}), and $v_{g}\left(z,t\right)$ is smooth
on $\left(0,\infty\right)^{2}$. Moreover,
\begin{equation}
v_{g}\left(z,t\right)=\mathbb{E}\left[g\left(Y\left(z,t\right)\right);t<\zeta_{0}^{Y}\left(z\right)\right]\text{ for every }\left(z,t\right)\in\left(0,\infty\right)^{2},\label{eq:prob interpretation of v_g}
\end{equation}
which implies that for every Borel set $\Gamma\subseteq\left(0,\infty\right)$,
\begin{equation}
\int_{\Gamma}q\left(z,w,t\right)dw=\mathbb{P}\left(Y\left(z,t\right)\in\Gamma,t<\zeta_{0}^{Y}\left(z\right)\right).\label{eq:prob interpretation of q}
\end{equation}

Finally, $q\left(z,w,t\right)$ satisfies the Chapman-Kolmogorov equation,
i.e., for every $z,w>0$ and $t,s>0$, 
\begin{equation}
q\left(z,w,t+s\right)=\int_{0}^{\infty}q\left(z,\xi,t\right)q\left(\xi,w,s\right)d\xi.\label{eq:CK equation for q_nu}
\end{equation}
\end{prop}

It is clear from (\ref{eq:prob interpretation of q}) that, for every
$\left(z,t\right)\in\left(0,\infty\right)^{2}$, $w\mapsto q\left(z,w,t\right)$
is the probability density function of $Y\left(z,t\right)$, provided
that $t<\zeta_{0}^{Y}\left(z\right)$. Now we turn our attention to
$q_{J}\left(z,w,t\right)$, the fundamental solution to (\ref{eq:model equation with potential localized})
which has an extra Dirichlet boundary at $J$. Intuitively speaking,
to get $q_{J}\left(z,w,t\right)$, we need to remove from $q\left(z,w,t\right)$
the ``contribution'' of $Y\left(z,t\right)$ once $Y\left(z,t\right)$
exists the interval $\left(0,J\right)$. Based on this idea combined
with the fact that $\left\{ Y\left(z,t\right):t\geq0\right\} $ is
a strong Markov process, we define 
\begin{equation}
\begin{split}q_{J}\left(z,w,t\right) & :=q\left(z,w,t\right)-\mathbb{E}\left[q\left(J,w,t-\zeta_{J}^{Y}\left(z\right)\right);\zeta_{J}^{Y}\left(z\right)\leq t\wedge\zeta_{0}^{Y}\left(z\right)\right]\end{split}
\label{eq:relation between q and q_J}
\end{equation}
for every $\left(z,w,t\right)\in\left(0,J\right)^{2}\times\left(0,\infty\right)$.
Again, by (\ref{eq:prob interpretation of q}), we see that for every
Borel set $\Gamma\subseteq\left(0,J\right)$,
\begin{equation}
\begin{split}\int_{\Gamma}q_{J}\left(z,w,t\right)dw & =\mathbb{P}\left(Y\left(z,t\right)\in\Gamma,t<\zeta_{0}^{Y}\left(z\right)\right)-\mathbb{P}\left(Y\left(z,t\right)\in\Gamma,\zeta_{J}^{Y}\left(z\right)\leq t<\zeta_{0}^{Y}\left(z\right)\right)\\
 & =\mathbb{P}\left(Y\left(z,t\right)\in\Gamma,t<\zeta_{0,J}^{Y}\left(z\right)\right).
\end{split}
\label{eq:prob interpretation of q_J}
\end{equation}
In other words, $q_{J}\left(z,w,t\right)$ is the probability density
function of $Y\left(z,t\right)$ provided that $t<\zeta_{0,J}^{Y}\left(z\right)$. 

To better analyze $q_{J}\left(z,w,t\right)$, we need the following
probability estimates on the hitting times of $Y\left(z,t\right)$.
\begin{lem}
\label{lem: prob estimate for hitting time of Y}For every $z\in\left(0,J\right)$,
\begin{equation}
\mathbb{P}\left(\zeta_{J}^{Y}\left(z\right)\leq\zeta_{0}^{Y}\left(z\right)\right)=\frac{z^{1-\nu}}{J^{1-\nu}}.\label{eq:prob for Y hitting J before 0}
\end{equation}
For $t>0$ and $J-z\geq\left|\nu\right|t$, we have that
\begin{equation}
\mathbb{P}\left(\zeta_{J}^{Y}\left(z\right)\leq t\right)\leq\exp\left(-\frac{\left(J-z-t\nu\right)^{2}}{4tJ}\right).\label{eq:prob estimate on hitting time of Y less than t}
\end{equation}
Furthermore, almost surely 
\[
\lim_{J\nearrow\infty}\zeta_{J}^{Y}\left(z\right)=\infty\text{ and }\lim_{z\nearrow J}\zeta_{J}^{Y}\left(z\right)=0.
\]
\end{lem}

\begin{proof}
Based on (\ref{eq:SDE of Y}), one can apply Itô's formula (see, e.g.,
$\mathsection5$ of \cite{bm_stochc_calc}) to check that, for every
$z>0$, $\left\{ \left(Y\left(z,t\right)\right)^{1-\nu}:t\geq0\right\} $
is a local martingale, and hence by Doob's stopping time theorem (see,
e.g., $\mathsection8$ of \cite{probability}), $\left\{ \left(Y\left(z,t\wedge\zeta_{0,J}^{Y}\left(z\right)\right)\right)^{1-\nu}:t\geq0\right\} $
is a bounded martingale. Thus,
\[
z^{1-\nu}=\mathbb{E}\left[\left(Y\left(z,\zeta_{0,J}^{Y}\left(z\right)\right)\right)^{1-\nu}\right]=\mathbb{P}\left(\zeta_{J}^{Y}\left(z\right)\leq\zeta_{0}^{Y}\left(z\right)\right)J^{1-\nu};
\]

To show (\ref{eq:prob estimate on hitting time of Y less than t}),
we check that for every $z\in\left(0,J\right)$ and every $\lambda>0$,
if 
\[
E\left(z,t\right):=\exp\left(\lambda Y\left(z,t\right)-\lambda\nu t-\lambda^{2}\int_{0}^{t}Y\left(z,s\right)ds\right)\text{ for }t\geq0,
\]
then $\left\{ E\left(z,t\wedge\zeta_{J}^{Y}\left(z\right)\right):t\geq0\right\} $
is a martingale. By a similar argument as above and Fatou's lemma,
we get that 
\begin{equation}
e^{\lambda J}\mathbb{E}\left[e^{-\left(\lambda\nu+\lambda^{2}J\right)\zeta_{J}^{Y}\left(z\right)};\zeta_{J}^{Y}\left(z\right)<\infty\right]\leq\liminf_{t\nearrow\infty}\mathbb{E}\left[E\left(z,t\wedge\zeta_{J}^{Y}\left(z\right)\right)\right]=e^{\lambda z}.\label{eq:bound for char. function of hitting time of Y}
\end{equation}
Set $\lambda:=\frac{J-z-\nu t}{2tJ}$. Since $J>z+t\left|\nu\right|$,
we have that $\lambda>0$ and $\lambda\nu+\lambda^{2}J>0$. Therefore,
a simple application of Markov's inequality leads to
\[
\begin{split}\mathbb{P}\left(\zeta_{J}^{Y}\left(z\right)\leq t\right) & =\mathbb{P}\left(e^{-\left(\lambda\nu+\lambda^{2}J\right)\zeta_{J}^{Y}\left(z\right)}\geq e^{-\left(\lambda\nu+\lambda^{2}J\right)t}\right)\\
 & \leq e^{\left(\lambda\nu+\lambda^{2}J\right)t}\mathbb{E}\left[e^{-\left(\lambda\nu+\lambda^{2}J\right)\zeta_{J}^{Y}\left(z\right)};\zeta_{J}^{Y}\left(z\right)<\infty\right]\\
 & \leq\exp\left(\lambda^{2}tJ-\lambda\left(J-z-\nu t\right)\right).
\end{split}
\]
Plugging the value of $\lambda$ into the right hand side yields (\ref{eq:prob estimate on hitting time of Y less than t}).
The fact that $\zeta_{J}^{Y}\left(z\right)$ converges to $\infty$
as $J\nearrow\infty$ almost surely follows from (\ref{eq:prob estimate on hitting time of Y less than t})
and the monotonicity of $\zeta_{J}^{Y}\left(z\right)$ in $J$.

Finally, we observe that $\zeta:=\lim_{z\nearrow J}\zeta_{J}^{Y}\left(z\right)$
exists almost surely. Take $\lambda\in\mathbb{R}$ such that $\lambda\nu\geq0$.
It follows from (\ref{eq:bound for char. function of hitting time of Y})
and the reverse Fatou's lemma that 
\[
\begin{split}e^{\lambda J}\mathbb{E}\left[e^{-\lambda\nu\zeta}\right] & \geq\mathbb{E}\left[\limsup_{z\nearrow J}E\left(z,\zeta_{J}^{Y}\left(z\right)\right)\right]\geq\limsup_{z\nearrow J}\mathbb{E}\left[E\left(z,\zeta_{J}^{Y}\left(z\right)\right)\right]\\
 & \geq\limsup_{z\nearrow J}\,\limsup_{t\nearrow\infty}\mathbb{E}\left[E\left(z,t\wedge\zeta_{J}^{Y}\left(z\right)\right)\right]=e^{\lambda J},
\end{split}
\]
which implies that $\mathbb{E}\left[e^{-\lambda\nu\zeta}\right]=1$
and hence $\zeta=0$ almost surely.
\end{proof}
\begin{prop}
\label{prop:properties of q_J}Let $q_{J}\left(z,w,t\right)$ be defined
as in (\ref{eq:relation between q and q_J}). Then, $q_{J}\left(z,w,t\right)$
is continuous on $\left(0,J\right)^{2}\times\left(0,\infty\right)$,
and for every $\left(z,w,t\right)\in\left(0,J\right)^{2}\times\left(0,\infty\right)$,
we have that 
\begin{equation}
w^{1-\nu}q_{J}\left(z,w,t\right)=z^{1-\nu}q_{J}\left(w,z,t\right).\label{eq:symmetry for q_J}
\end{equation}
$q_{J}\left(z,w,t\right)$ satisfies the Chapman-Kolmogorov equation,
i.e., for every $z,w\in\left(0,J\right)$ and $t,s>0$,
\begin{equation}
q_{J}\left(z,w,t+s\right)=\int_{0}^{J}q_{J}\left(z,\xi,t\right)q_{J}\left(\xi,w,s\right)d\xi.\label{eq:CK equation for q_J}
\end{equation}

For every $w\in\left(0,J\right)$, $\left(z,t\right)\mapsto q_{J}\left(z,w,t\right)$
is a smooth solution to the Kolmogorov backward equation corresponding
to $L_{0}$:
\begin{equation}
\partial_{t}q_{J}\left(z,w,t\right)=L_{0}q_{J}\left(z,w,t\right);\label{eq:q_J satisfies backward eq}
\end{equation}
for every $z\in\left(0,J\right)$, $\left(w,t\right)\mapsto q_{J}\left(z,w,t\right)$
is a smooth solution to the Kolmogorov forward equation corresponding
to $L_{0}$:
\begin{equation}
\partial_{t}q_{J}\left(z,w,t\right)=L_{0}^{*}q_{J}\left(z,w,t\right)\label{eq:q_J satisfies forward eq}
\end{equation}
where $L_{0}^{*}=w\partial_{w}^{2}+\left(2-\nu\right)\partial_{w}$
is the formal adjoint of $L_{0}$. 

Moreover, $q_{J}\left(z,w,t\right)$ is the fundamental solution to
(\ref{eq:model equation localized}). Given $g\in C_{c}\left(\left(0,J\right)\right)$,
if
\begin{equation}
v_{g,\left(0,J\right)}\left(z,t\right):=\int_{0}^{J}g\left(w\right)q_{J}\left(z,w,t\right)dw\text{ for }\left(z,t\right)\in\left(0,J\right)\times\left(0,\infty\right),\label{eq:def of v_g,(0,J)}
\end{equation}
then $v_{g,\left(0,J\right)}\left(z,t\right)$ is the unique solution
in $C^{,2,1}\left(\left(0,J\right)\times\left(0,\infty\right)\right)$
to (\ref{eq:model equation localized}), and $v_{g,\left(0,I\right)}\left(z,t\right)$
is smooth on $\left(0,J\right)\times\left(0,\infty\right)$.
\end{prop}

\begin{proof}
We start with (\ref{eq:CK equation for q_J}) since its proof is straightforward.
Given $z,w\in\left(0,J\right)$, $t,s>0$ and Borel set $\Gamma\subseteq\left(0,J\right)$,
by (\ref{eq:prob interpretation of q_J}) and the strong Markov property
of $Y\left(z,t\right)$, we can write
\[
\begin{split}\int_{\Gamma}q_{J}\left(z,w,t+s\right)dw & =\mathbb{P}\left(Y\left(z,t+s\right)\in\Gamma,t+s<\zeta_{0,J}^{Y}\left(z\right)\right)\\
 & =\int_{\Gamma}\mathbb{E}\left[q_{J}\left(Y\left(z,t\right),w,s\right);t<\zeta_{0,J}^{Y}\left(z\right)\right]dw\\
 & =\int_{\Gamma}\int_{0}^{J}q_{J}\left(z,\xi,t\right)q_{J}\left(\xi,w,s\right)d\xi dw,
\end{split}
\]
which implies (\ref{eq:CK equation for q_J}). Next, given $t>0$,
we take any $m\in\mathbb{N}$, any $0=s_{0}<s_{1}<s_{2}<\cdots<s_{m-1}<s_{m}=t$
and $\varphi_{0},\varphi_{2},\cdots,\varphi_{m}\in C_{c}\left(\left(0,J\right)\right)$.
By (\ref{eq:symmetry of q}), (\ref{eq:prob interpretation of q_J})
and, again, the Markov property of $Y\left(z,t\right)$, we have that
\[
\begin{split} & \int_{0}^{J}\mathbb{E}\left[\prod_{k=0}^{m}\varphi_{k}\left(Y\left(z,s_{k}\right)\right);t<\zeta_{0,J}^{Y}\left(z\right)\right]\frac{dz}{z^{1-\nu}}\\
= & \int_{0}^{J}\idotsint_{\left(0,J\right)^{m}}\varphi_{0}\left(z\right)\prod_{k=1}^{m}\varphi_{k}\left(\xi_{k}\right)q_{J}\left(z,\xi_{1},s_{1}\right)q_{J}\left(\xi_{1},\xi_{2},s_{2}-s_{1}\right)\\
 & \hspace{2cm}\hspace{2cm}\cdots q_{J}\left(\xi_{m-1},\xi_{m},t-s_{m-1}\right)d\xi_{m}\cdots d\xi_{1}\frac{dz}{z^{1-\nu}}\\
= & \int_{0}^{J}\idotsint_{\left(0,J\right)^{m}}\varphi_{0}\left(z\right)\prod_{k=1}^{m}\varphi_{k}\left(\xi_{k}\right)q_{J}\left(\xi_{1},z,t-\left(t-s_{1}\right)\right)q_{J}\left(\xi_{2},\xi_{1},\left(t-s_{1}\right)-\left(t-s_{2}\right)\right)\\
 & \hspace{2cm}\hspace{2cm}\hspace{1cm}\cdots q_{J}\left(\xi_{m},\xi_{m-1},t-s_{m-1}\right)\frac{d\xi_{m}}{\xi_{m}^{1-\nu}}d\xi_{m-1}\cdots d\xi_{1}dz\\
= & \int_{0}^{J}\mathbb{E}\left[\prod_{k=0}^{m}\varphi_{k}\left(Y\left(\xi_{m},t-s_{k}\right)\right);t<\zeta_{0}^{Y}\left(\xi_{m}\right)\right]\frac{d\xi_{m}}{\xi_{m}^{1-\nu}}.
\end{split}
\]
Since $t\mapsto Y\left(z,t\right)$ is almost surely continuous and
$s_{0},\cdots,s_{m},\varphi_{0},\cdots,\varphi_{m}$ are chosen arbitrarily,
the relation above implies that for every measurable functional $F$
on $C\left(\left[0,t\right]\right)$, 
\[
\int_{0}^{J}\mathbb{E}\left[F\left(\left.Y\left(z,\cdot\right)\right|_{\left[0,t\right]}\right);t<\zeta_{0,J}^{Y}\left(z\right)\right]\frac{dz}{z^{1-\nu}}=\int_{0}^{J}\mathbb{E}\left[F\left(\left.\overleftarrow{Y}\left(w,\cdot\right)\right|_{\left[0,t\right]}\right);t<\zeta_{0,J}^{Y}\left(w\right)\right]\frac{dw}{w^{1-\nu}}
\]
where $\overleftarrow{Y}\left(z,s\right):=Y\left(z,t-s\right)$ for
every $s\in\left[0,t\right]$. In particular, for arbitrary $\varphi$,
$\varphi^{*}\in C_{c}\left(\left(0,J\right)\right)$, if $F$ is chosen
such that for every $y\left(\cdot\right)\in C\left(\left[0,t\right]\right)$,
\[
F\left(y\left(\cdot\right)\right)=\begin{cases}
\varphi\left(y\left(0\right)\right)\varphi^{*}\left(y\left(t\right)\right), & \text{if }0<y\left(s\right)<J\text{ for every }s\in\left[0,t\right],\\
0 & \text{otherwise},
\end{cases}
\]
then we have that
\[
\int_{0}^{J}\int_{0}^{J}\varphi\left(z\right)\varphi^{*}\left(w\right)q_{J}\left(z,w,t\right)\frac{dwdz}{z^{1-\nu}}=\int_{0}^{J}\int_{0}^{J}\varphi\left(z\right)\varphi^{*}\left(w\right)q_{J}\left(w,z,t\right)\frac{dwdz}{w^{1-\nu}}.
\]
This is sufficient for us to conclude (\ref{eq:symmetry for q_J}). 

Now we turn attention to (\ref{eq:q_J satisfies backward eq}) and
(\ref{eq:q_J satisfies forward eq}). By (\ref{eq:symmetry for q_J}),
it suffices to prove only one of them, say, (\ref{eq:q_J satisfies forward eq}).
To this end, we take $\varphi\in C_{c}^{\infty}\left(\left(0,J\right)\right)$
and consider $v_{\varphi,\left(0,J\right)}\left(z,t\right)$, which,
according to (\ref{eq:prob interpretation of q_J}), can be written
as 
\[
v_{\varphi,\left(0,J\right)}\left(z,t\right)=\mathbb{E}\left[\varphi\left(Y\left(z,t\right)\right);t<\zeta_{0,J}^{Y}\left(z\right)\right]\text{ for every }\left(z,t\right)\in\left(0,J\right)\times\left(0,\infty\right).
\]
As reviewed in $\mathsection1.2$, for every $z\in\left(0,J\right)$,
\[
\left\{ \varphi\left(Y\left(z,t\wedge\zeta_{0,J}^{Y}\left(z\right)\right)\right)-\int_{0}^{t\wedge\zeta_{0,J}^{Y}\left(z\right)}\left(L_{0}\varphi\right)\left(Y\left(z,s\right)\right)ds:t\geq0\right\} 
\]
is a bounded martingale. Thus,
\[
\begin{split}\varphi\left(z\right) & =\mathbb{E}\left[\varphi\left(Y\left(z,t\right)\right);t<\zeta_{0,J}^{Y}\left(z\right)\right]-\int_{0}^{t}\mathbb{E}\left[L_{0}\varphi\left(Y\left(z,s\right)\right);s<\zeta_{0,J}^{Y}\left(z\right)\right]ds\\
 & =v_{\varphi,\left(0,J\right)}\left(z,t\right)-\int_{0}^{t}\int_{0}^{J}L_{0}\varphi\left(w\right)q_{J}\left(z,w,s\right)dwds,
\end{split}
\]
and hence 
\[
\partial_{t}\left(\int_{0}^{J}\varphi\left(w\right)q_{J}\left(z,w,t\right)dw\right)=\int_{0}^{J}L_{0}\varphi\left(w\right)q_{J}\left(z,w,t\right)dw.
\]
This means that for every $z\in\left(0,J\right)$, $\left(w,t\right)\mapsto q_{J}\left(z,w,t\right)$
solves the equation $\left(\partial_{t}-L_{0}^{*}\right)q_{J}\left(z,w,t\right)=0$
in the sense of distribution. Since $\partial_{t}-L_{0}^{*}$ is a
hypoelliptic operator (see, e.g., $\mathsection7.4$ of \cite{PDEStroock}),
$\left(w,t\right)\mapsto q_{J}\left(z,w,t\right)$ is a smooth solution
to (\ref{eq:q_J satisfies forward eq}). 

For $\left(z,w,t\right)\in\left(0,J\right)^{2}\times\left(0,\infty\right)$,
we set
\begin{equation}
r\left(z,w,t\right):=q\left(z,w,t\right)-q_{J}\left(z,w,t\right)=\mathbb{E}\left[q\left(J,w,t-\zeta_{J}^{Y}\left(z\right)\right);\zeta_{J}^{Y}\left(z\right)\leq t\wedge\zeta_{0}^{Y}\left(z\right)\right].\label{eq:def of r(z,w,t)}
\end{equation}
Then, for every $w\in\left(0,J\right)$, $\left(z,t\right)\mapsto r\left(z,w,t\right)$
is smooth on $\left(0,J\right)\times\left(0,\infty\right)$. It is
easy to see that $w\mapsto r\left(z,w,t\right)$ is equicontinuous
in $\left(z,t\right)$ from any bounded subset of $\left(0,J\right)\times\left(0,\infty\right)$,
which implies that $r\left(z,w,t\right)$, as well as $q_{J}\left(z,w,t\right)$,
is continuous on $\left(0,J\right)^{2}\times\left(0,\infty\right)$. 

We proceed to the proof of the last statement. Again, by the hypoellipticity
of $\partial_{t}-L_{0}$, to show that $v_{g,\left(0,J\right)}\left(z,t\right)$
is a smooth solution to the model equation, we only need to show that
it solves the equation as a distribution. Let us take $\varphi\in C_{c}^{\infty}\left(\left(0,J\right)\right)$
and consider, for every $t>0$,
\[
\left\langle \varphi,v_{g,\left(0,J\right)}\left(\cdot,t\right)\right\rangle :=\int_{0}^{J}\varphi\left(z\right)v_{g,\left(0,J\right)}\left(z,t\right)dz=\int_{0}^{J}\int_{0}^{J}\varphi\left(z\right)q_{J}\left(z,w,t\right)dzg\left(w\right)dw.
\]
By (\ref{eq:q_J satisfies backward eq}), we have that
\begin{align*}
\frac{d}{dt}\left\langle \varphi,v_{g,\left(0,J\right)}\left(\cdot,t\right)\right\rangle  & =\int_{0}^{J}\left(\int_{0}^{J}\varphi\left(z\right)\left(L_{0}q_{J}\left(\cdot,w,t\right)\right)\left(z\right)dz\right)g\left(w\right)dw\\
 & =\int_{0}^{J}\int_{0}^{J}L_{0}^{*}\varphi\left(z\right)q_{J}\left(z,w,t\right)g\left(w\right)dwdz\\
 & =\left\langle L_{0}^{*}\varphi,v_{g,\left(0,J\right)}\left(\cdot,t\right)\right\rangle .
\end{align*}
The only remaining thing to do is to verify that $v_{g,\left(0,J\right)}\left(z,t\right)$
satisfies the initial value and the boundary value conditions in (\ref{eq:model equation localized}).
Given $g\in C_{c}\left(\left(0,J\right)\right)$, by (\ref{eq:prob interpretation of v_g})
and (\ref{eq:prob interpretation of q_J}), we have that for every
$z\in\left(0,I\right)$,
\[
\begin{split}\left|v_{g,\left(0,J\right)}\left(z,t\right)-v_{g}\left(z,t\right)\right| & \leq\mathbb{E}\left[\left|g\left(Y\left(z,t\right)\right)\right|;\zeta_{J}^{Y}\left(z\right)\le t<\zeta_{0}^{Y}\left(z\right)\right]\\
 & \leq\left\Vert g\right\Vert _{u}\mathbb{P}\left(\zeta_{J}^{Y}\left(z\right)\le t\right)
\end{split}
\]
which, according to (\ref{eq:prob estimate on hitting time of Y less than t}),
goes to $0$ as $t\searrow0$, and the convergence is uniformly fast
for $z$ on any compact subset of $\left(0,J\right)$. Therefore,
\[
\lim_{t\searrow0}v_{g,\left(0,J\right)}\left(z,t\right)=\lim_{t\searrow0}v_{g}\left(z,t\right)=0.
\]
To verify that $v_{g,\left(0,I\right)}\left(z,t\right)$ satisfies
the boundary condition, it is sufficient to show that
\[
\lim_{z\searrow0}r\left(z,w,t\right)=0\text{ and }\lim_{z\nearrow J}r\left(z,w,t\right)=q\left(J,w,t\right)\text{ for every }\left(w,t\right)\in\left(0,J\right)\times\left(0,\infty\right).
\]
We observe that, by (\ref{eq: estimate of q}), $q\left(J,w,t-\zeta_{J}^{Y}\left(z\right)\right)$
is bounded uniformly in $z$ by 
\[
J^{1-\nu}\left(\sqrt{J}-\sqrt{w}\right)^{2\left(\nu-2\right)}\left(\frac{2-\nu}{e}\right)^{2-\nu}
\]
where we used the fact that
\[
\sup_{s>0}\,s^{2-\nu}e^{-s}=\left(\frac{2-\nu}{e}\right)^{2-\nu}.
\]
Therefore, (\ref{eq:prob for Y hitting J before 0}) implies that
\[
\lim_{z\searrow0}r\left(z,w,t\right)\leq J^{1-\nu}\left(\left(\sqrt{J}-\sqrt{w}\right)^{2}\frac{2-\nu}{e}\right)^{2-\nu}\lim_{z\searrow0}\mathbb{P}\left(\zeta_{J}^{Y}\left(z\right)\leq\zeta_{0}^{Y}\left(z\right)\right)=0.
\]
Finally, the last statement in Lemma \ref{lem: prob estimate for hitting time of Y}
and the dominated convergence theorem lead to 
\[
\lim_{z\nearrow J}r\left(z,w,t\right)=q\left(J,w,t\right).
\]
\end{proof}
We will close this subsection with a result on the comparison between
$q_{J}\left(z,w,t\right)$ and $q\left(z,w,t\right)$. Intuitively
speaking, given $z\in\left(0,J\right)$ sufficiently far from the
boundary $J$ and $t$ sufficiently small, $Y\left(z,t\right)$ would
not have exited $\left(0,J\right)$ by time $t$, which means that
$q\left(z,w,t\right)$ and $q_{J}\left(z,w,t\right)$ should be close
to each other. We will make this statement rigorous by proving that,
as $t\searrow0$, $q_{J}\left(z,w,t\right)/q\left(z,w,t\right)$ converges
to 1 uniformly fast in $\left(z,w\right)$ away from $J$. 
\begin{cor}
\label{cor:estimate on ratio q_J/q}Set $t_{J}:=\frac{4J}{9\left(2-\nu\right)}$.
Then, for every $t\in\left(0,t_{J}\right)$,
\begin{equation}
\sup_{\left(z,w\right)\in\left(0,\frac{1}{9}J\right)^{2}}\left|\frac{q_{J}\left(z,w,t\right)}{q\left(z,w,t\right)}-1\right|\leq\exp\left(-\frac{2J}{9t}\right).\label{eq: estimate on ratio q_J/q}
\end{equation}
\end{cor}

\begin{proof}
It is easy to verify that $t_{J}$ is chosen such that the function
$s\mapsto s^{\nu-2}\exp\left(-\frac{4J}{9s}\right)$ is increasing
on $\left(0,t_{J}\right)$. By (\ref{eq: estimate of q}) and (\ref{eq:prob for Y hitting J before 0}),
we have that for every $\left(z,w\right)\in\left(0,\frac{1}{9}J\right)$
and $t\in\left(0,t_{J}\right)$,
\[
\begin{split}\left|\frac{q_{J}\left(z,w,t\right)}{q\left(z,w,t\right)}-1\right|=\frac{\left|r\left(z,w,t\right)\right|}{q\left(z,w,t\right)} & \leq\frac{J^{1-\nu}\mathbb{P}\left(\zeta_{J}^{Y}\left(z\right)\leq\zeta_{0}^{Y}\left(z\right)\right)\cdot\sup_{s\in\left(0,t\right)}s^{\nu-2}\exp\left(-\frac{\left(\sqrt{J}-\sqrt{w}\right)^{2}}{s}\right)}{z^{1-\nu}t^{\nu-2}\exp\left(-\frac{z+w}{t}\right)}\\
 & \leq\frac{\sup_{s\in\left(0,t\right)}s^{\nu-2}\exp\left(-\frac{4J}{9s}\right)}{t^{\nu-2}\exp\left(-\frac{z+w}{t}\right)}\leq\exp\left(-\frac{4J}{9t}+\frac{z+w}{t}\right)\leq\exp\left(\frac{-2J}{9t}\right).
\end{split}
\]
\end{proof}

\section{Localized Equation}

\subsection{From $q_{J}\left(z,w,t\right)$ to $q_{J}^{V}\left(z,w,t\right)$}

Now we get down to solving (\ref{eq:model equation with potential localized})
by the perturbation method of Duhamel. First we want to find a function
$q_{J}^{V}\left(z,w,t\right)$ on $\left(0,J\right)^{2}\times\left(0,\infty\right)$
that solves the integral equation 
\begin{equation}
q_{J}^{V}\left(z,w,t\right)=q_{J}\left(z,w,t\right)+\int_{0}^{t}\int_{0}^{J}q_{J}\left(z,\xi,t-s\right)q_{J}^{V}\left(\xi,w,s\right)V\left(\xi\right)d\xi ds\label{eq: duhamel integral eq}
\end{equation}
for every $\left(z,w,t\right)\in\left(0,J\right)^{2}\times\left(0,\infty\right)$,
and then verify that $q_{J}^{V}\left(z,w,t\right)$ is the fundamental
solution to (\ref{eq:model equation with potential localized}). To
this end, for every $\left(z,w,t\right)\in\left(0,J\right)^{2}\times\left(0,\infty\right)$
and $n\in\mathbb{N}$, we define 
\begin{equation}
q_{J,0}\left(z,w,t\right):=q_{J}\left(z,w,t\right)\text{ and }q_{J,n+1}\left(z,w,t\right):=\int_{0}^{t}\int_{0}^{J}q_{J}\left(z,\xi,t-s\right)q_{J,n}\left(\xi,w,s\right)V\left(\xi\right)d\xi ds.\label{eq:recursion n->n+1 for q^V_J}
\end{equation}
To state the technical results on $\left\{ q_{J,n}\left(z,w,t\right):n\geq0\right\} $,
we need to introduce more notations. Set 
\begin{equation}
\mathfrak{b}:=\begin{cases}
\nu & \text{ if }\alpha\in\left(0,1\right)\text{ and }b\left(0\right)\neq0,\\
1-\nu & \text{ if }\alpha\in\left(0,1\right)\text{ and }b\left(0\right)=0,\\
1 & \text{ if }\alpha\in[1,2).
\end{cases}\label{eq:def of mathfrak(b)}
\end{equation}
We have that $0<\mathfrak{b}\leq1$, and if $V_{J}$ is the constant
found in Lemma \ref{lem:estimates on VJ}, then (\ref{eq:estimate of V})
can be rewritten as 
\[
\left|V\left(z\right)\right|\leq V_{J}\cdot z^{\mathfrak{b}-1}\text{ for every }z\in\left(0,J\right).
\]
For $n\in\mathbb{N}$ and $t>0$, we define
\begin{equation}
m_{n}\left(t\right):=\frac{\Gamma^{n+1}\left(\mathfrak{\mathfrak{b}}\right)\left(\mathfrak{c}t^{\mathfrak{b}}V_{J}\right)^{n}}{\Gamma\left(\left(n+1\right)\mathfrak{b}\right)}\text{ and }M\left(t\right):=\sum_{n=0}^{\infty}m_{n}\left(t\right).\label{eq:def of m_n(t)}
\end{equation}
It follows from a simple application of Stirling's formula that $m_{n}\left(t\right)$
is summable in $n\in\mathbb{N}$, and hence $M\left(t\right)$ is
well defined.
\begin{lem}
\label{lem:def of q^V}There exists a universal constant $\mathfrak{c}\geq1$
such that for every $n\in\mathbb{N}$ and $\left(z,w,t\right)\in\left(0,J\right)^{2}\times\left(0,\infty\right)$,
\begin{equation}
\left|q_{J,n}\left(z,w,t\right)\right|\leq m_{n}\left(t\right)q\left(z,w,t\right),\label{eq:estimate for q_J,n}
\end{equation}
and hence
\begin{equation}
q_{J}^{V}\left(z,w,t\right):=\sum_{n=0}^{\infty}q_{J,n}\left(z,w,t\right)\label{eq:def of q^V_J}
\end{equation}
is well defined as an absolutely convergent series. Moreover, for
every $\left(z,w,t\right)\in\left(0,J\right)^{2}\times\left(0,\infty\right)$,
\begin{equation}
\left|q_{J}^{V}\left(z,w,t\right)\right|\leq M\left(t\right)q\left(z,w,t\right),\label{eq:exp estimate for q^V_J}
\end{equation}
and $q_{J}^{V}\left(z,w,t\right)$ satisfies (\ref{eq: duhamel integral eq}). 
\end{lem}

\begin{proof}
Without causing any substantial change, we will assume that $V\left(z\right)$
is defined on $\left(0,\infty\right)$ with $V\left(z\right)\equiv0$
for $z\geq J.$ When $1\le\alpha<2$, since $V\left(z\right)$ is
bounded on $\left(0,\infty\right)$ with $V_{J}=\left\Vert V\right\Vert _{u}$,
(\ref{eq:estimate for q_J,n})-(\ref{eq:exp estimate for q^V_J})
can be derived in exactly the same way as in \cite{deg_diff_global}
(Lemma 3.4) with 
\[
m_{n}\left(t\right)=\frac{\left(tV_{J}\right)^{n}}{n!}\text{ and }M\left(t\right)=e^{tV_{J}}.
\]
There is nothing we need to do in this case. Hence, we will assume
$\alpha\in\left(0,1\right)$ for the rest of the proof, and only treat
the case when $V\left(z\right)$ has a singularity at $0$. 

First, we claim that there exists a universal constant $\mathfrak{c}>0$
such that 
\begin{equation}
\int_{0}^{\infty}q\left(z,\xi,t\right)q\left(\xi,w,s\right)\xi^{\mathfrak{b}-1}d\xi\leq\mathfrak{c}\left(\frac{t+s}{ts}\right)^{1-\mathfrak{b}}q\left(z,w,t+s\right)\label{eq: CK inequality}
\end{equation}
for every $z,w\in\left(0,J\right)^{2}$ and $t,s>0$. To see this,
we use (\ref{eq:def of q}) and (\ref{eq:symmetry of q}) to write
the integral in (\ref{eq: CK inequality}) as 
\[
\begin{split} & \frac{z^{1-\nu}}{\left(ts\right)^{2-\nu}}e^{-\frac{z}{t}-\frac{w}{s}}\int_{0}^{\infty}e^{-\frac{(t+s)\xi}{ts}}\xi^{\mathfrak{b}-\nu}\left(\sum_{n=0}^{\infty}\frac{\left(z\xi\right)^{n}}{n!\Gamma\left(n+2-\nu\right)t^{2n}}\right)\left(\sum_{n=0}^{\infty}\frac{\left(w\xi\right)^{n}}{n!\Gamma\left(n+2-\nu\right)s^{2n}}\right)d\xi\\
= & \frac{z^{1-\nu}}{\left(ts\right)^{2-\nu}}e^{-\frac{z}{t}-\frac{w}{s}}\int_{0}^{\infty}e^{-\frac{(t+s)\xi}{ts}}\xi^{\mathfrak{b}-\nu}\sum_{n=0}^{\infty}\xi^{2n}\omega_{n}\left(z,w,t,s\right)d\xi
\end{split}
\]
where for every $n\in\mathbb{N}$,
\[
\omega_{n}\left(z,w,t,s\right):=\sum_{k=0}^{n}\frac{z^{k}w^{n-k}}{k!\left(n-k\right)!\Gamma\left(k+2-\nu\right)\Gamma\left(n-k+2-\nu\right)t^{2k}s^{2(n-k)}}.
\]
Interchanging the order of summation and integration yields
\[
\begin{split} & \frac{z^{1-\nu}}{\left(ts\right)^{2-\nu}}e^{-\frac{z}{t}-\frac{w}{s}}\sum_{n=0}^{\infty}\omega_{n}\left(z,w,t,s\right)\int_{0}^{\infty}e^{-\frac{(t+s)\xi}{ts}}\xi^{2n+\mathfrak{b}-\nu}d\xi\\
= & \frac{z^{1-\nu}}{\left(ts\right)^{2-\nu}}e^{-\frac{z}{t}-\frac{w}{s}}\sum_{n=0}^{\infty}\omega_{n}\left(z,w,t,s\right)\left(\frac{t+s}{ts}\right)^{\nu-\mathfrak{b}-1-2n}\Gamma\left(2n+1+\mathfrak{b}-\nu\right)\\
= & \left(\frac{t+s}{ts}\right)^{1-\mathfrak{b}}\frac{z^{1-\nu}}{\left(ts\right)^{2-\nu}}e^{-\frac{z}{t}-\frac{w}{s}}\sum_{n=0}^{\infty}\left(\frac{t+s}{ts}\right)^{\nu-2-2n}\Gamma\left(2n+1+\mathfrak{b}-\nu\right)\omega_{n}\left(z,w,t,s\right).
\end{split}
\]
Since $0<\mathfrak{b}\leq1$, we have that for $n\in\mathbb{N}$,
\[
\frac{\Gamma\left(2n+1+\mathfrak{b}-\nu\right)}{\Gamma\left(2n+2-\nu\right)}=\frac{B\left(2n+1+\mathfrak{b}-\nu,1-\mathfrak{b}\right)}{\Gamma\left(1-\mathfrak{b}\right)}\leq\frac{1}{\left(1-\mathfrak{b}\right)\Gamma\left(1-\mathfrak{b}\right)}=\frac{1}{\Gamma\left(2-\mathfrak{b}\right)}\leq\mathfrak{c},
\]
where $B\left(u,v\right)$ (with $u,v>0$) is the beta function and
\begin{equation}
\mathfrak{c}:=\frac{1}{\min_{s\in\left[1,2\right]}\Gamma\left(s\right)}\thickapprox1.12917.\label{eq:def of mathfrak(c)}
\end{equation}
Therefore, we have that 
\[
\begin{split} & \left(\frac{t+s}{ts}\right)^{1-\mathfrak{b}}\frac{z^{1-\nu}}{\left(ts\right)^{2-\nu}}e^{-\frac{z}{t}-\frac{w}{s}}\sum_{n=0}^{\infty}\left(\frac{t+s}{ts}\right)^{\nu-2-2n}\Gamma\left(2n+1+\mathfrak{b}-\nu\right)\omega_{n}\left(z,w,t,s\right)\\
\leq & \mathfrak{c}\left(\frac{t+s}{ts}\right)^{1-\mathfrak{b}}\frac{z^{1-\nu}}{\left(ts\right)^{2-\nu}}e^{-\frac{z}{t}-\frac{w}{s}}\sum_{n=0}^{\infty}\left(\frac{t+s}{ts}\right)^{\nu-2-2n}\Gamma\left(2n+2-\nu\right)\omega_{n}\left(z,w,t,s\right)\\
= & \mathfrak{c}\left(\frac{t+s}{ts}\right)^{1-\mathfrak{b}}\frac{z^{1-\nu}}{\left(ts\right)^{2-\nu}}e^{-\frac{z}{t}-\frac{w}{s}}\int_{0}^{\infty}e^{-\frac{(t+s)\xi}{ts}}\sum_{n=0}^{\infty}\xi^{2n+1-\nu}\omega_{n}\left(z,w,t,s\right)d\xi\\
= & \mathfrak{c}\left(\frac{t+s}{ts}\right)^{1-\mathfrak{b}}\int_{0}^{\infty}q\left(z,\xi,t\right)q\left(\xi,w,s\right)d\xi\\
= & \mathfrak{c}\left(\frac{t+s}{ts}\right)^{1-\mathfrak{b}}q\left(z,w,t+s\right),
\end{split}
\]
which confirms the claim (\ref{eq: CK inequality}). 

To proceed, we notice that by Lemma \ref{lem:estimates on VJ}, (\ref{eq:CK equation for q_nu})
and (\ref{eq:recursion n->n+1 for q^V_J}),
\[
\begin{split}\left|q_{J,1}\left(z,w,t\right)\right| & \leq\int_{0}^{t}\int_{0}^{\infty}q_{J}\left(z,\xi,t-s\right)q_{J}\left(\xi,w,s\right)\left|V\left(\xi\right)\right|d\xi ds\\
 & \leq V_{J}\int_{0}^{t}\int_{0}^{\infty}q\left(z,\xi,t-s\right)q\left(\xi,w,s\right)\xi^{\mathfrak{b}-1}d\xi ds\\
 & \leq\mathfrak{c}V_{J}\int_{0}^{t}\frac{t^{1-\mathfrak{b}}}{s^{1-\mathfrak{b}}\left(t-s\right)^{1-\mathfrak{b}}}ds\cdot q\left(z,w,t\right)\\
 & =\mathfrak{c}t^{\mathfrak{b}}V_{J}B\left(\mathfrak{b},\mathfrak{b}\right)q\left(z,w,t\right)
\end{split}
\]
for every $\left(z,w,t\right)\in\left(0,\infty\right)^{3}$. Assume
that up to some $n\geq1$, for every $\left(z,w,t\right)\in\left(0,\infty\right)^{3}$,
\[
\left|q_{J,n}\left(z,w,t\right)\right|\leq\left(\mathfrak{c}t^{\mathfrak{b}}V_{J}\right)^{n}\left(\prod_{j=1}^{n}B\left(\mathfrak{b},j\mathfrak{b}\right)\right)q\left(z,w,t\right).
\]
For $n+1$, we have that
\[
\begin{split}\left|q_{J,n+1}\left(z,w,t\right)\right| & \leq V_{J}\int_{0}^{t}\int_{0}^{\infty}q\left(z,\xi,t-s\right)\left|q_{J,n}\left(\xi,w,s\right)\right|\xi^{\mathfrak{b}-1}d\xi ds\\
 & \leq\mathfrak{c}^{n}V_{J}^{n+1}\left(\prod_{j=1}^{n}B\left(\mathfrak{b},j\mathfrak{b}\right)\right)\int_{0}^{t}s^{n\mathfrak{b}}\int_{0}^{\infty}q\left(z,\xi,t-s\right)q\left(z,w,t\right)\xi^{\mathfrak{b}-1}d\xi ds\\
 & \leq\left(\mathfrak{c}V_{J}\right)^{n+1}\left(\prod_{j=1}^{n}B\left(\mathfrak{b},j\mathfrak{b}\right)\right)\int_{0}^{t}s^{n\mathfrak{b}}\frac{t^{1-\mathfrak{b}}}{s^{1-\mathfrak{b}}\left(t-s\right)^{1-\mathfrak{b}}}ds\cdot q\left(z,w,t\right)\\
 & =\left(\mathfrak{c}t^{\mathfrak{b}}V_{J}\right)^{n+1}\left(\prod_{j=1}^{n+1}B\left(\mathfrak{b},j\mathfrak{b}\right)\right)q\left(z,w,t\right).
\end{split}
\]
Upon rewriting $\prod_{j=1}^{n}B\left(\mathfrak{b},j\mathfrak{b}\right)$
as $\frac{\left(\Gamma\left(\mathfrak{b}\right)\right)^{n+1}}{\Gamma\left(\left(n+1\right)\mathfrak{b}\right)}$,
we immediately obtain (\ref{eq:estimate for q_J,n})-(\ref{eq:exp estimate for q^V_J}).
Finally, (\ref{eq: duhamel integral eq}) can be verified by plugging
the series representation of $q_{J}^{V}\left(z,w,t\right)$ into the
right hand side of (\ref{eq: duhamel integral eq}) and integrating
term by term. 
\end{proof}
We are now ready to solve (\ref{eq:model equation with potential localized}). 
\begin{prop}
\label{prop:properties of q^V_J} Let $q_{J}^{V}\left(z,w,t\right)$
be defined as in (\ref{eq:def of q^V_J}). Then, $q_{J}^{V}\left(z,w,t\right)$
is continuous on $\left(0,J\right)^{2}\times\left(0,\infty\right)$,
and for every $\left(z,w,t\right)\in\left(0,J\right)^{2}\times\left(0,\infty\right)$,
we have that 

\begin{equation}
w^{1-\nu}q_{J}^{V}\left(z,w,t\right)=z^{1-\nu}q_{J}^{V}\left(w,z,t\right).\label{eq:symmetry of q^V_J}
\end{equation}
$q_{J}^{V}\left(z,w,t\right)$ also satisfies the following integral
equation:
\begin{equation}
q_{J}^{V}\left(z,w,t\right)=q_{J}\left(z,w,t\right)+\int_{0}^{t}\int_{0}^{\infty}q_{J}^{V}\left(z,\xi,t-s\right)q_{J}\left(\xi,w,s\right)V\left(\xi\right)d\xi ds.\label{eq: duhamel integral equiv}
\end{equation}

Moreover, $q_{J}^{V}\left(z,w,t\right)$ is the fundamental solution
to (\ref{eq:model equation with potential localized}). Given $h\in C_{c}\left(\left(0,J\right)\right)$,
\begin{equation}
v_{h,\left(0,J\right)}^{V}\left(z,t\right):=\int_{0}^{J}q_{J}^{V}\left(z,w,t\right)h\left(w\right)dw\text{ for }\left(z,t\right)\in\left(0,J\right)\times\left(0,\infty\right)\label{eq:def of v^V_h,(0,J)}
\end{equation}
is a smooth solution to (\ref{eq:model equation with potential localized}). 
\end{prop}

\begin{proof}
To prove (\ref{eq:symmetry of q^V_J}), we first note that if, for
$\left(z,w,t\right)\in\left(0,J\right)^{2}\times\left(0,\infty\right)$
and $n\in\mathbb{N}$, we define
\begin{equation}
\tilde{q}_{J,0}\left(z,w,t\right):=q_{J}\left(z,w,t\right)\text{ and }\tilde{q}_{J,n+1}\left(z,w,t\right):=\int_{0}^{t}\int_{0}^{J}\tilde{q}_{J,n}\left(z,\xi,t-s\right)q_{J}\left(\xi,w,s\right)V\left(\xi\right)d\xi ds,\label{eq:alternative of n->n+1}
\end{equation}
then $\tilde{q}_{J,n}\left(z,w,t\right)=q_{J,n}\left(z,w,t\right)$.
In other words, (\ref{eq:alternative of n->n+1}) is an equivalent
recursive relation to (\ref{eq:recursion n->n+1 for q^V_J}). To see
this, one can expand both the right hand side of (\ref{eq:recursion n->n+1 for q^V_J})
and that of (\ref{eq:alternative of n->n+1}) into two respective
$2n-$fold integrals, and confirm that the two integrals are identical.
Next, we will show by induction that for every $\left(z,w,t\right)\in\left(0,J\right)^{2}\times\left(0,\infty\right)$
and $n\in\mathbb{N}$,
\[
w^{1-\nu}q_{J,n}\left(z,w,t\right)=z^{1-\nu}q_{J,n}\left(w,z,t\right).
\]
When $n=0$, this relation is simply (\ref{eq:symmetry for q_J}).
Assume that this relation holds up to some $n\in\mathbb{N}$. By (\ref{eq:symmetry for q_J})
and the equivalence between (\ref{eq:recursion n->n+1 for q^V_J})
and (\ref{eq:alternative of n->n+1}), we have that
\[
\begin{split}w^{1-\nu}q_{J,n+1}\left(z,w,t\right) & =z^{1-\nu}\int_{0}^{t}\int_{0}^{J}q_{J}\left(\xi,z,t-s\right)q_{J,n}\left(w,\xi,s\right)V\left(\xi\right)d\xi ds\\
 & =z^{1-\nu}\int_{0}^{t}\int_{0}^{J}\tilde{q}_{J,n}\left(w,\xi,s\right)q_{J}\left(\xi,z,t-s\right)V\left(\xi\right)d\xi ds\\
 & =z^{1-\nu}\tilde{q}_{J,n+1}\left(w,z,t\right)=z^{1-\nu}q_{J,n+1}\left(w,z,t\right).
\end{split}
\]
(\ref{eq:symmetry of q^V_J}) follows immediately. To establish (\ref{eq: duhamel integral equiv}),
we write its right hand side as
\[
\begin{split} & q_{J}\left(z,w,t\right)+\int_{0}^{t}\int_{0}^{\infty}q_{J}^{V}\left(z,\xi,t-s\right)q_{J}\left(\xi,w,s\right)V\left(\xi\right)d\xi ds\\
= & q_{J}\left(z,w,t\right)+\sum_{n=0}^{\infty}\int_{0}^{t}\int_{0}^{\infty}\tilde{q}_{J,n}\left(z,\xi,t-s\right)q_{J}\left(\xi,w,s\right)V\left(\xi\right)d\xi ds\\
= & q_{J}\left(z,w,t\right)+\sum_{n=0}^{\infty}\tilde{q}_{J,n+1}\left(z,w,t\right)\\
= & q_{J}^{V}\left(z,w,t\right),
\end{split}
\]
where, again, we used the equivalence between (\ref{eq:recursion n->n+1 for q^V_J})
and (\ref{eq:alternative of n->n+1}). By (\ref{eq: duhamel integral eq})
and (\ref{eq:exp estimate for q^V_J}), $\left(z,t\right)\mapsto q_{J}^{V}\left(z,w,t\right)$
is continuous for every $w\in\left(0,J\right)$, and by (\ref{eq: duhamel integral equiv}),
$w\mapsto q_{J}^{V}\left(z,w,t\right)$ is equicontinuous in $\left(z,t\right)$
from any bounded subset of $\left(0,J\right)\times\left(0,\infty\right)$.
From here one can easily derives the continuity of $q_{J}^{V}\left(z,w,t\right)$
in $\left(z,w,t\right)$ on $\left(0,J\right)^{2}\times\left(0,\infty\right)$.

Given $h\in C_{c}\left(\left(0,J\right)\right)$, for every $\left(z,t\right)\in\left(0,J\right)\times\left(0,\infty\right)$,
let $v_{h,\left(0,J\right)}^{V}\left(z,t\right)$ and $v_{h,\left(0,J\right)}\left(z,t\right)$
be defined as in (\ref{eq:def of v^V_h,(0,J)}) and (\ref{eq:def of v_g,(0,J)})
respectively. It follows from (\ref{eq: duhamel integral eq}) that
\begin{equation}
\begin{split}v_{h,\left(0,J\right)}^{V}\left(z,t\right) & =v_{h,\left(0,J\right)}\left(z,t\right)+\end{split}
\int_{0}^{t}\int_{0}^{J}q_{J}\left(z,\xi,t-s\right)v_{h,\left(0,J\right)}^{V}\left(\xi,s\right)V\left(\xi\right)d\xi ds.\label{eq:relation between v_h  and v_h^V}
\end{equation}
Let $\mathfrak{b}$ and $\mathfrak{c}$ be as in (\ref{eq:def of mathfrak(b)})
and (\ref{eq:def of mathfrak(c)}) respectively. By (\ref{eq:exp estimate for q^V_J})
and (\ref{eq: CK inequality}), we have that
\[
\begin{split}\left|v_{h,\left(0,J\right)}^{V}\left(z,t\right)-v_{h,\left(0,J\right)}\left(z,t\right)\right| & =\left|\int_{0}^{t}\int_{0}^{J}q_{J}\left(z,\xi,t-s\right)v_{h,\left(0,J\right)}^{V}\left(\xi,s\right)V\left(\xi\right)d\xi ds\right|\\
 & \leq\left\Vert h\right\Vert _{u}\int_{0}^{J}\int_{0}^{t}\int_{0}^{J}q_{J}\left(z,\xi,t-s\right)\left|q_{J}^{V}\left(\xi,u,s\right)\right|\left|V\left(\xi\right)\right|d\xi dsdu\\
 & \leq\left\Vert h\right\Vert _{u}M\left(t\right)\int_{0}^{J}\int_{0}^{t}\int_{0}^{J}q_{J}\left(z,\xi,t-s\right)q\left(\xi,u,s\right)\left|V\left(\xi\right)\right|d\xi dsdu\\
 & \leq\left\Vert h\right\Vert _{u}M\left(t\right)\mathfrak{c}t^{\mathfrak{b}}V_{J}B\left(\mathfrak{b},\mathfrak{b}\right)\int_{0}^{J}q\left(z,u,t\right)du.
\end{split}
\]
Since $v_{h,\left(0,J\right)}\left(z,t\right)$ is a solution to (\ref{eq:model equation localized}),
the second last inequality implies that
\[
\lim_{z\searrow0}v_{h,\left(0,J\right)}^{V}\left(z,t\right)=\lim_{z\nearrow J}v_{h,\left(0,J\right)}^{V}\left(z,t\right)=0,
\]
and the last inequality leads to $\lim_{t\searrow0}v_{h,\left(0,J\right)}^{V}\left(z,t\right)=h\left(z\right)$.

The only thing that remains to be proven is that $v_{h,\left(0,J\right)}^{V}\left(z,t\right)$
is a smooth solution to the equation in (\ref{eq:model equation with potential localized}),
which, by the hypoellipticity of the operator $\partial_{t}-L^{V}$,
can be reduced to showing that $v_{h,\left(0,J\right)}^{V}\left(z,t\right)$
is a solution in the sense of distribution. We take $\varphi\in C_{c}^{\infty}\left(\left(0,J\right)\right)$
and consider 
\[
\left\langle \varphi,v_{h,\left(0,J\right)}^{V}\left(\cdot,t\right)\right\rangle :=\int_{0}^{J}v_{h,\left(0,J\right)}^{V}\left(z,t\right)\varphi\left(z\right)dz\text{ for }t\geq0,
\]
and use (\ref{eq:relation between v_h  and v_h^V}) to write it as
\[
\left\langle \varphi,v_{h,\left(0,J\right)}^{V}\left(\cdot,t\right)\right\rangle =\left\langle \varphi,v_{h,\left(0,J\right)}\left(\cdot,t\right)\right\rangle +\int_{0}^{t}\int_{0}^{J}\left\langle \varphi,q_{J}\left(\cdot,u,t-s\right)\right\rangle v_{h,\left(0,J\right)}^{V}\left(u,s\right)V\left(u\right)duds.
\]
Therefore,  
\[
\begin{split}\frac{d}{dt}\left\langle \varphi,v_{h,\left(0,J\right)}^{V}\left(\cdot,t\right)\right\rangle  & =\frac{d}{dt}\left\langle \varphi,v_{h,\left(0,J\right)}\left(\cdot,t\right)\right\rangle +\left\langle V\varphi,v_{h,\left(0,J\right)}^{V}\left(\cdot,t\right)\right\rangle \\
 & \qquad\qquad+\int_{0}^{t}\int_{0}^{J}\frac{d}{dt}\left\langle \varphi,q_{J}\left(\cdot,u,t-s\right)\right\rangle v_{h,\left(0,J\right)}^{V}\left(u,s\right)V\left(u\right)duds\\
 & =\left\langle L_{0}^{*}\varphi,v_{h,\left(0,J\right)}\left(\cdot,t\right)\right\rangle +\left\langle V\varphi,v_{h,\left(0,J\right)}^{V}\left(\cdot,t\right)\right\rangle \\
 & \qquad\qquad+\int_{0}^{t}\int_{0}^{J}\left\langle L_{0}^{*}\varphi,q_{J}\left(\cdot,u,t-s\right)\right\rangle v_{h,\left(0,J\right)}^{V}\left(u,s\right)V\left(u\right)duds\\
 & =\left\langle \left(L_{0}^{*}+V\right)\varphi,v_{h,\left(0,J\right)}^{V}\left(\cdot,t\right)\right\rangle .
\end{split}
\]
\end{proof}

\subsection{Approximation of $q_{J}^{V}\left(z,w,t\right)$}

In general we do not expect to find a closed-form formula for $q_{J}^{V}\left(z,w,t\right)$,
but when $t$ is sufficiently small, the above construction does provide
accurate approximations for $q_{J}^{V}\left(z,w,t\right)$ whose exact
formulas are explicit or even in closed forms. Intuitively speaking,
when $t$ is small, the effect of the potential $V\left(z\right)$
in $L^{V}$ has not become ``substantial'' so that $L^{V}$ is close
to $L_{0}$, and hence it is natural to expect that $q_{J}^{V}\left(z,w,t\right)$
is close to $q_{J}\left(z,w,t\right)$ which, as we have seen in Corollary
\ref{cor:estimate on ratio q_J/q}, is well approximated by $q\left(z,w,t\right)$
for sufficiently small $t$. To make it rigorous, we take $t_{J}$
to be the same as in Corollary \ref{cor:estimate on ratio q_J/q}
and use (\ref{eq: estimate on ratio q_J/q}) and (\ref{eq:estimate for q_J,n})
to derive that for every $t\in\left(0,t_{J}\right)$, 
\begin{align*}
\sup_{z,w\in\left(0,\frac{1}{9}J\right)^{2}}\left|\frac{q_{J}^{V}\left(z,w,t\right)}{q\left(z,w,t\right)}-1\right| & \leq\sup_{z,w\in\left(0,\frac{1}{9}J\right)^{2}}\left(\left|\frac{q_{J}^{V}\left(z,w,t\right)-q_{J}\left(z,w,t\right)}{q\left(z,w,t\right)}\right|+\left|\frac{q_{J}\left(z,w,t\right)}{q\left(z,w,t\right)}-1\right|\right)\\
 & \leq M\left(t\right)-1+\exp\left(-\frac{2J}{9t}\right).
\end{align*}
Hence, for some constant $C>0$ uniformly in $t\in\left(0,t_{J}\right)$
($C$ may depend on $J$ and $\alpha$), 
\begin{equation}
\sup_{z,w\in\left(0,\frac{1}{9}J\right)^{2}}\left|\frac{q_{J}^{V}\left(z,w,t\right)}{q\left(z,w,t\right)}-1\right|\le Ct^{\mathfrak{b}}.\label{eq:estimate for q^V_J/q}
\end{equation}

(\ref{eq:estimate for q^V_J/q}) confirms that when $t$ is small,
$q_{J}^{V}\left(z,w,t\right)$ is indeed well approximated by $q\left(z,w,t\right)$.
However, viewing from (\ref{eq:def of q^V_J}), $q\left(z,w,t\right)$
is only the ``first order'' approximation to $q_{J}^{V}\left(z,w,t\right)$,
since the error $t^{\mathfrak{b}}$ in (\ref{eq:estimate for q^V_J/q})
is generated by keeping only the first term in the series in (\ref{eq:def of q^V_J}).
It is possible to derive a more general ``$k-$th order'' approximation
for $q_{J}^{V}\left(z,w,t\right)$ with $k\in\mathbb{N}$, and obtain
an analog of (\ref{eq:estimate for q^V_J/q}) with the error being
$t^{k\mathfrak{b}}$. To achieve this purpose, we introduce a new
sequence of functions. For $\left(z,w,t\right)\in\left(0,\infty\right)^{3}$
and $n\in\mathbb{N}$, we set
\begin{equation}
q_{0}\left(z,w,t\right):=q\left(z,w,t\right)\text{ and }q_{n}\left(z,w,t\right):=\int_{0}^{t}\int_{0}^{\infty}q\left(z,\xi,t-s\right)V\left(\xi\right)q_{n}\left(\xi,w,s\right)d\xi ds,\label{eq: recursion n->n+1 for q^V}
\end{equation}
where, again, we assume that $V\left(z\right)\equiv0$ for $z>J$.
By following the proof of (\ref{eq:estimate for q_J,n}) line by line
with $q_{J,n}\left(z,w,t\right)$ replaced by $q_{n}\left(z,w,t\right)$,
we also get that for every $\left(z,w,t\right)\in\left(0,\infty\right)^{3}$
and $n\in\mathbb{N}$,
\begin{equation}
\left|q_{n}\left(z,w,t\right)\right|\leq m_{n}\left(t\right)q\left(z,w,t\right).\label{eq:estimate for q_n}
\end{equation}
Clearly, $q_{n}\left(z,w,t\right)$ is the ``global'' counterpart
of $q_{J,n}\left(z,w,t\right)$, and we will justify that $q_{J,n}\left(z,w,t\right)$
is close to $q_{n}\left(z,w,t\right)$ when $t$ is sufficiently small.
\begin{lem}
For every $\left(z,w,t\right)\in\left(0,J\right)^{2}\times\left(0,\infty\right)$
and $n\in\mathbb{N}$, 
\begin{equation}
\left|q_{J,n}\left(z,w,t\right)-q_{n}\left(z,w,t\right)\right|\leq\left(2\mathfrak{c}t^{\mathfrak{b}}B\left(\mathfrak{b},\mathfrak{b}\right)V_{J}\right)^{n}r\left(z,w,t\right),\label{eq:estimate on q_J,n - q_n}
\end{equation}
where $r\left(z,w,t\right)$ is as in (\ref{eq:def of r(z,w,t)}).
\end{lem}

\begin{proof}
When $n=0$, (\ref{eq:estimate on q_J,n - q_n}) simply becomes (\ref{eq:def of r(z,w,t)}).
Assume that (\ref{eq:estimate on q_J,n - q_n}) holds up to some $n\geq0$.
Following (\ref{eq:recursion n->n+1 for q^V_J}) and (\ref{eq: recursion n->n+1 for q^V}),
we write 
\begin{equation}
\begin{split} & q_{n+1}\left(z,w,t\right)-q_{J,n+1}\left(z,w,t\right)\\
= & \int_{0}^{t}\int_{0}^{\infty}r\left(z,\xi,t-s\right)V\left(\xi\right)q_{n}\left(\xi,w,s\right)d\xi ds\\
 & \qquad\qquad+\int_{0}^{t}\int_{0}^{\infty}q_{J}\left(z,\xi,t-s\right)V\left(\xi\right)\left(q_{n}\left(\xi,w,s\right)-q_{J,n}\left(\xi,w,s\right)\right)d\xi ds.
\end{split}
\label{eq:difference between q_J,n and q_n}
\end{equation}
We use Fubini's theorem and (\ref{eq:def of r(z,w,t)}) to rewrite
the first term on the right hand side of (\ref{eq:difference between q_J,n and q_n})
as
\begin{equation}
\mathbb{E}\left[\int_{0}^{t-\zeta_{J}^{Y}\left(z\right)}\int_{0}^{\infty}q\left(J,\xi,t-s-\zeta_{J}^{Y}\left(z\right)\right)V\left(\xi\right)q_{n}\left(\xi,w,s\right)d\xi ds;\zeta_{J}^{Y}\left(z\right)\leq t\wedge\zeta_{0}^{Y}\left(z\right)\right],\label{eq:1st term in q_n+1 - q_J,n+1}
\end{equation}
which, by (\ref{eq:estimate for q_n}), is bounded by 
\[
\frac{\Gamma^{n+1}\left(\mathfrak{b}\right)\left(\mathfrak{c}V_{J}\right)^{n}}{\Gamma\left(\left(n+1\right)\mathfrak{b}\right)}\mathbb{E}\left[\int_{0}^{t-\zeta_{J}^{Y}\left(z\right)}\int_{0}^{\infty}q\left(J,\xi,t-s-\zeta_{J}^{Y}\left(z\right)\right)\left|V\left(\xi\right)\right|s^{n\mathfrak{b}}q\left(\xi,w,s\right)d\xi ds;\zeta_{J}^{Y}\left(z\right)\leq t\wedge\zeta_{0}^{Y}\left(z\right)\right].
\]
By (\ref{eq: CK inequality}) and the fact that 
\[
\frac{\Gamma^{n+1}\left(\mathfrak{b}\right)}{\Gamma\left(\left(n+1\right)\mathfrak{b}\right)}=\prod_{j=1}^{n}B\left(\mathfrak{b},j\mathfrak{b}\right)\leq B^{n}\left(\mathfrak{b},\mathfrak{b}\right),
\]
we can further bound (\ref{eq:1st term in q_n+1 - q_J,n+1}) from
above by 
\[
\begin{split} & \mathfrak{c}^{n}B^{n}\left(\mathfrak{b},\mathfrak{b}\right)V_{J}^{n+1}\mathbb{E}\left[\int_{0}^{t-\zeta_{J}^{Y}\left(z\right)}s^{n\mathfrak{b}}\int_{0}^{\infty}q\left(J,\xi,t-s-\zeta_{J}^{Y}\left(z\right)\right)\xi^{\mathfrak{b}-1}q\left(\xi,w,s\right)d\xi ds;\zeta_{J}^{Y}\left(z\right)\leq t\wedge\zeta_{0}^{Y}\left(z\right)\right]\\
\leq & \mathfrak{c}^{n+1}B^{n}\left(\mathfrak{b},\mathfrak{b}\right)V_{J}^{n+1}\mathbb{E}\left[q\left(J,w,t-\zeta_{J}^{Y}\left(z\right)\right)\int_{0}^{t-\zeta_{J}^{Y}\left(z\right)}\frac{\left(t-\zeta_{J}^{Y}\left(z\right)\right)^{1-\mathfrak{b}}s^{n\mathfrak{b}}}{\left(t-s-\zeta_{J}^{Y}\left(z\right)\right)^{1-\mathfrak{b}}s^{1-\mathfrak{b}}}ds;\zeta_{J}^{Y}\left(z\right)\leq t\wedge\zeta_{0}^{Y}\left(z\right)\right]\\
\leq & \left(\mathfrak{c}t^{\mathfrak{b}}B\left(\mathfrak{b},\mathfrak{b}\right)V_{J}\right)^{n+1}r\left(z,w,t\right).
\end{split}
\]
According to the inductive assumption, the second term on the right
hand side of (\ref{eq:difference between q_J,n and q_n}) is bounded
by 
\[
2^{n}\mathfrak{c}^{n}B^{n}\left(\mathfrak{b},\mathfrak{b}\right)V_{J}^{n+1}\int_{0}^{t}\int_{0}^{\infty}q\left(z,\xi,t-s\right)\xi^{\mathfrak{b}-1}s^{n\mathfrak{b}}r\left(\xi,w,s\right)d\xi ds,
\]
which, by (\ref{eq:symmetry of q}) and (\ref{eq:symmetry for q_J}),
is equal to 
\[
\frac{2^{n}\mathfrak{c}^{n}B^{n}\left(\mathfrak{b},\mathfrak{b}\right)V_{J}^{n+1}}{w^{1-\nu}z^{\nu-1}}\int_{0}^{t}\int_{0}^{\infty}q\left(\xi,z,t-s\right)\xi^{\mathfrak{b}-1}s^{n\mathfrak{b}}r\left(w,\xi,s\right)d\xi ds.
\]
We use Fubini's theorem again to rewrite the expression above as
\[
\begin{split} & \frac{2^{n}\mathfrak{c}^{n}B^{n}\left(\mathfrak{b},\mathfrak{b}\right)V_{J}^{n+1}}{w^{1-\nu}z^{\nu-1}}\mathbb{E}\left[\int_{\zeta_{J}^{Y}\left(w\right)}^{t}\int_{0}^{\infty}q\left(\xi,z,t-s\right)\xi^{\mathfrak{b}-1}s^{n\mathfrak{b}}q\left(J,\xi,s-\zeta_{J}^{Y}\left(w\right)\right)d\xi ds;\zeta_{J}^{Y}\left(w\right)\leq t\wedge\zeta_{0}^{Y}\left(w\right)\right]\\
\leq & \frac{2^{n}\mathfrak{c}^{n+1}B^{n}\left(\mathfrak{b},\mathfrak{b}\right)V_{J}^{n+1}}{w^{1-\nu}z^{\nu-1}}\mathbb{E}\left[q\left(J,z,t-\zeta_{J}^{Y}\left(w\right)\right)\int_{\zeta_{J}^{Y}\left(w\right)}^{t}\frac{\left(t-\zeta_{J}^{Y}\left(w\right)\right)^{1-\mathfrak{b}}s^{n\mathfrak{b}}}{\left(t-s\right)^{1-\mathfrak{b}}\left(s-\zeta_{J}^{Y}\left(w\right)\right)^{1-\mathfrak{b}}}ds;\zeta_{J}^{Y}\left(w\right)\leq t\wedge\zeta_{0}^{Y}\left(w\right)\right]\\
\leq & \frac{2^{n}\left(\mathfrak{c}t^{\mathfrak{b}}B\left(\mathfrak{b},\mathfrak{b}\right)V_{J}\right)^{n+1}}{w^{1-\nu}z^{\nu-1}}r\left(w,z,t\right)=2^{n}\left(\mathfrak{c}t^{\mathfrak{b}}B\left(\mathfrak{b},\mathfrak{b}\right)V_{J}\right)^{n+1}r\left(z,w,t\right).
\end{split}
\]
Thus, combining the estimates of the two terms on the right hand side
of (\ref{eq:difference between q_J,n and q_n}), we obtain that
\[
\begin{split}\left|q_{J,n+1}\left(z,w,t\right)-q_{n+1}\left(z,w,t\right)\right| & \leq\left(1+2^{n}\right)\left(\mathfrak{c}t^{\mathfrak{b}}B\left(\mathfrak{b},\mathfrak{b}\right)V_{J}\right)^{n+1}r\left(z,w,t\right)\\
 & \leq\left(2\mathfrak{c}t^{\mathfrak{b}}B\left(\mathfrak{b},\mathfrak{b}\right)V_{J}\right)^{n+1}r\left(z,w,t\right).
\end{split}
\]
\end{proof}
\begin{prop}
\label{prop: approximation for q^V_J}Let $t_{J}:=\frac{4J}{9\left(2-\nu\right)}$.
Then, for every $t\in\left(0,t_{J}\right)$ and $k\in\mathbb{N}\backslash\left\{ 0\right\} $,
\begin{equation}
\sup_{z,w\in\left(0,\frac{1}{9}J\right)^{2}}\left|\frac{q_{J}^{V}\left(z,w,t\right)-\sum_{n=0}^{k-1}q_{n}\left(z,w,t\right)}{q\left(z,w,t\right)}\right|\leq m_{k}\left(t\right)M\left(t\right)+D_{k}\left(t\right)\exp\left(-\frac{2J}{9t}\right),\label{eq: estimate of q^V/q}
\end{equation}
where
\begin{equation}
D_{k}\left(t\right):=\sum_{n=0}^{k-1}\left(2\mathfrak{c}t^{\mathfrak{b}}B\left(\mathfrak{b},\mathfrak{b}\right)V_{J}\right)^{n}.\label{eq:def of D_k}
\end{equation}
In particular, there exists $C>0$ uniformly in $t\in\left(0,t_{J}\right)$
and $k\in\mathbb{N}\backslash\left\{ 0\right\} $ ($C$ may depend
on $J$ and $\nu$) such that
\begin{equation}
\sup_{z,w\in\left(0,\frac{1}{9}J\right)^{2}}\left|\frac{q_{J}^{V}\left(z,w,t\right)-\sum_{n=0}^{k-1}q_{n}\left(z,w,t\right)}{q\left(z,w,t\right)}\right|\leq Ct^{k\mathfrak{b}}.\label{eq:estimate for q^V -kth order}
\end{equation}
\end{prop}

\begin{proof}
Only (\ref{eq: estimate of q^V/q}) requires proof, since (\ref{eq:estimate for q^V -kth order})
follows from (\ref{eq: estimate of q^V/q}) trivially. By (\ref{eq:estimate for q_J,n})
and (\ref{eq:def of q^V_J}), we know that for every $\left(z,w,t\right)\in\left(0,J\right)^{2}\times\left(0,\infty\right)$
and $k\in\mathbb{N}\backslash\left\{ 0\right\} $
\[
\left|\frac{q_{J}^{V}\left(z,w,t\right)-\sum_{n=0}^{k-1}q_{J,n}\left(z,w,t\right)}{q\left(z,w,t\right)}\right|\leq\sum_{n=k}^{\infty}m_{n}\left(t\right)
\]
and we further derive that
\[
\begin{split}\sum_{n=k}^{\infty}m_{n}\left(t\right) & =\sum_{n=k}^{\infty}\frac{\Gamma^{n+1}\left(\mathfrak{b}\right)\left(\mathfrak{c}t^{\mathfrak{b}}V_{J}\right)^{n}}{\Gamma\left(\left(n+1\right)\mathfrak{b}\right)}\\
 & =\Gamma^{k}\left(\mathfrak{b}\right)\left(\mathfrak{c}t^{\mathfrak{b}}V_{J}\right)^{k}\sum_{l=0}^{\infty}\frac{\Gamma^{l+1}\left(\mathfrak{b}\right)\left(\mathfrak{c}t^{\mathfrak{b}}V_{J}\right)^{l}}{\Gamma\left(\left(l+1\right)\mathfrak{b}+k\mathfrak{b}\right)}\\
 & =\frac{\Gamma^{k}\left(\mathfrak{b}\right)\left(\mathfrak{c}t^{\mathfrak{b}}V_{J}\right)^{k}}{\Gamma\left(k\mathfrak{b}\right)}\sum_{l=0}^{\infty}\frac{\Gamma^{l+1}\left(\mathfrak{b}\right)\left(\mathfrak{c}t^{\mathfrak{b}}V_{J}\right)^{l}}{\Gamma\left(\left(l+1\right)\mathfrak{b}\right)}B\left(\left(l+1\right)\mathfrak{b},k\mathfrak{b}\right)\\
 & \leq\frac{\Gamma^{k+1}\left(\mathfrak{b}\right)\left(\mathfrak{c}t^{\mathfrak{b}}V_{J}\right)^{k}}{\Gamma\left(\left(k+1\right)\mathfrak{b}\right)}\sum_{l=0}^{\infty}\frac{\Gamma^{l+1}\left(\mathfrak{b}\right)\left(\mathfrak{c}t^{\mathfrak{b}}V_{J}\right)^{l}}{\Gamma\left(\left(l+1\right)\mathfrak{b}\right)}\\
 & =m_{k}\left(t\right)M\left(t\right),
\end{split}
\]
where we again used the fact that $B\left(\left(l+1\right)\mathfrak{b},k\mathfrak{b}\right)\leq B\left(\mathfrak{b},k\mathfrak{b}\right)$
for every $l\in\mathbb{N}$. Meanwhile, by (\ref{eq:estimate on q_J,n - q_n}),
we have that for every $\left(z,w,t\right)\in\left(0,J\right)^{2}\times\left(0,\infty\right)$,
\[
\left|\sum_{n=0}^{k-1}q_{J,n}\left(z,w,t\right)-\sum_{n=0}^{k-1}q_{n}\left(z,w,t\right)\right|\leq D_{k}\left(t\right)r\left(z,w,t\right),
\]
which, combined with (\ref{eq: estimate on ratio q_J/q}), leads to
(\ref{eq: estimate of q^V/q}).
\end{proof}

\subsection{From $q_{J}^{V}\left(z,w,t\right)$ to $p_{I}\left(z,w,t\right)$}

Now we are ready to return to the localized equation (\ref{eq:general IVP localized}).
Recall that $I>0$, $\phi\left(x\right)$ and $\theta\left(x\right)$
are functions on $\left(0,I\right)$ defined by (\ref{eq:def of phi =000026 theta}),
and $I$ and $J$ are related by $J=\phi\left(I\right)$; with $z\in\left(0,J\right)$,
$\psi\left(z\right)$ is the inverse function of $\phi$, $\tilde{\theta}\left(z\right)=\theta\left(\psi\left(z\right)\right)$
and $\Theta\left(z\right)$ is as defined in (\ref{eq:def of Theta}).
Guided by Proposition \ref{prop:change of variable}, we define 
\begin{equation}
p_{I}\left(x,y,t\right):=q_{J}^{V}\left(\phi\left(x\right),\phi\left(y\right),t\right)\frac{\Theta\left(\phi\left(x\right)\right)}{\Theta\left(\phi\left(y\right)\right)}\phi^{\prime}\left(y\right).\label{eq:def of p_I}
\end{equation}
for every $\left(x,y,t\right)\in\left(0,I\right)^{2}\times\left(0,\infty\right)$.
We immediately obtain several results on $p_{I}\left(x,y,t\right)$
based on Proposition \ref{prop:change of variable} and Proposition
\ref{prop:properties of q^V_J}. In addition, we can establish the
connection between $p_{I}\left(x,y,t\right)$ and $\left\{ X\left(x,t\right):t\geq0\right\} $
the unique solution to (\ref{eq:SDE of X}) and underlying diffusion
process corresponding to $L=x^{\alpha}a\left(x\right)\partial_{x}^{2}+b\left(x\right)\partial_{x}$.
\begin{prop}
\label{prop:properties of p_I}Let $p_{I}\left(x,y,t\right)$ be defined
as in (\ref{eq:def of p_I}). Then, $p_{I}\left(x,y,t\right)$ is
continuous on $\left(0,I\right)^{2}\times\left(0,\infty\right)$ and
\begin{equation}
\frac{\left(\phi\left(y\right)\right)^{1-\nu}}{\phi^{\prime}\left(y\right)}\Theta^{2}\left(\phi\left(y\right)\right)p_{I}\left(x,y,t\right)=\frac{\left(\phi\left(x\right)\right)^{1-\nu}}{\phi^{\prime}\left(x\right)}\Theta^{2}\left(\phi\left(x\right)\right)p_{I}\left(y,x,t\right)\label{eq:symmetry of p_I}
\end{equation}
for every $\left(x,y,t\right)\in\left(0,I\right)^{2}\times\left(0,\infty\right)$.

$p_{I}\left(x,y,t\right)$ is the fundamental solution to (\ref{eq:general IVP localized}).
Given $f\in C_{c}\left(\left(0,I\right)\right)$,
\begin{equation}
u_{f,\left(0,I\right)}\left(x,t\right):=\int_{0}^{I}f\left(y\right)p_{I}\left(x,y,t\right)dy\label{eq:def of u_f,(0,I)}
\end{equation}
is the unique solution in $C^{2,1}\left(\left(0,I\right)\times\left(0,\infty\right)\right)$
to (\ref{eq:general IVP localized}), and $u_{f,\left(0,I\right)}\left(x,t\right)$
is smooth on $\left(0,I\right)\times\left(0,\infty\right)$. Moreover,
\begin{equation}
u_{f,\left(0,I\right)}\left(x,t\right)=\mathbb{E}\left[f\left(X\left(x,t\right)\right);t<\zeta_{0,I}^{X}\left(x\right)\right],\label{eq: prob interp. of u_f,(0,I)}
\end{equation}
and hence for every Borel set $\Gamma\subseteq\left(0,I\right)$,
\begin{equation}
\int_{\Gamma}p_{I}\left(x,y,t\right)dy=\mathbb{P}\left(X\left(x,t\right)\in\Gamma,t<\zeta_{0,I}^{X}\left(x\right)\right).\label{eq: prob interp. of p_I}
\end{equation}

Finally, $p_{I}\left(x,y,t\right)$ satisfies the Chapman-Kolmogorov
equation, i.e., for every $x,y\in\left(0,I\right)$ and $t,s>0$,
\begin{equation}
p_{I}\left(x,y,t+s\right)=\int_{0}^{I}p_{I}\left(x,\xi,t\right)p_{I}\left(\xi,y,s\right)d\xi.\label{eq:CK equation for p_I}
\end{equation}
\end{prop}

\begin{proof}
(\ref{eq:symmetry of p_I}) follows directly from (\ref{eq:symmetry of q^V_J}).
Given $f\in C_{c}\left(\left(0,I\right)\right)$, we set $h\left(z\right):=\frac{f\circ\psi\left(z\right)}{\Theta\left(z\right)}$
for $z\in\left(0,J\right)$. By (\ref{eq:def of v^V_h,(0,J)}), it
is straightforward to check that
\[
\begin{split}u_{f,\left(0,I\right)}\left(x,t\right) & =\Theta\left(\phi\left(x\right)\right)\int_{0}^{I}f\left(y\right)q_{J}^{V}\left(\phi\left(x\right),\phi\left(y\right),t\right)\frac{\phi^{\prime}\left(y\right)}{\Theta\left(\phi\left(y\right)\right)}dy\\
 & =\Theta\left(\phi\left(x\right)\right)\int_{0}^{J}f\left(\psi\left(w\right)\right)q_{J}^{V}\left(\phi\left(x\right),w,t\right)\frac{dw}{\Theta\left(w\right)}\\
 & =\Theta\left(\phi\left(x\right)\right)v_{h,\left(0,J\right)}^{V}\left(\phi\left(x\right),t\right),
\end{split}
\]
and hence it follows from Proposition \ref{prop:properties of q^V_J}
that $u_{f,\left(0,I\right)}\left(x,t\right)$ is a smooth solution
to (\ref{eq:general IVP localized}). Since
\[
\left\{ u_{f,\left(0,I\right)}\left(X\left(x,s\wedge\zeta_{0,I}^{X}\left(x\right)\right),t-s\wedge\zeta_{0,I}^{X}\left(x\right)\right):0\leq s\leq t\right\} 
\]
is a bounded martingale, by equating its expectation at $s=0$ and
$s=t$, we obtain (\ref{eq: prob interp. of u_f,(0,I)}), which further
leads to (\ref{eq: prob interp. of p_I}). Since $\left\{ X\left(x,t\right):t\geq0\right\} $
is the unique solution to (\ref{eq:SDE of X}), $u_{f,\left(0,I\right)}\left(x,t\right)$
is the unique $C^{2,1}\left(\left(0,I\right)\times\left(0,\infty\right)\right)$
solution to (\ref{eq:general IVP localized}). Finally, (\ref{eq:CK equation for p_I})
follows from (\ref{eq: prob interp. of p_I}) and the strong Markov
property of $\left\{ X\left(x,t\right):t\geq0\right\} $.
\end{proof}
\begin{rem}
Note that the properties developed above for $p_{I}\left(x,y,t\right)$
and $u_{f,\left(0,I\right)}\left(x,t\right)$ also lead to corresponding
results on $q_{J}^{V}\left(z,w,t\right)$ and $v_{h,\left(0,J\right)}^{V}\left(z,t\right)$.
For example, we see from (\ref{eq:CK equation for p_I}) that $q_{J}^{V}\left(z,w,t\right)$
also satisfies the Chapman-Kolmogorov equation, i.e., for every $z,w\in\left(0,J\right)$
and $t,s>0$,
\[
q_{J}^{V}\left(z,w,t+s\right)=\int_{0}^{J}q_{J}^{V}\left(z,\xi,t\right)q_{J}^{V}\left(\xi,w,s\right)d\xi,
\]
and the uniqueness of $u_{f,\left(0,I\right)}\left(z,t\right)$ implies
that, given $h\in C_{c}\left(\left(0,J\right)\right)$, $v_{h,\left(0,J\right)}^{V}\left(z,t\right)$
is the unique $C^{2,1}\left(\left(0,J\right)\times\left(0,\infty\right)\right)$
solution to (\ref{eq:model equation with potential localized}).
\end{rem}

The approximations we obtained in Proposition \ref{prop: approximation for q^V_J}
for $q_{J}^{V}\left(z,w,t\right)$ can also be ``transported'' to
$p_{I}\left(x,y,t\right)$ in a straightforward way. To see this,
we define, for $\left(x,y,t\right)\in\left(0,I\right)^{2}\times\left(0,\infty\right)$,
\begin{equation}
p^{approx.}\left(x,y,t\right):=q\left(\phi\left(x\right),\phi\left(y\right),t\right)\frac{\Theta\left(\phi\left(x\right)\right)}{\Theta\left(\phi\left(y\right)\right)}\phi^{\prime}\left(y\right),\label{eq:def of p^approx.}
\end{equation}
and more generally for $k\in\mathbb{N}\backslash\left\{ 0\right\} $,
\begin{equation}
p^{k-approx.}\left(x,y,t\right):=\sum_{n=0}^{k-1}q_{n}\left(\phi\left(x\right),\phi\left(y\right),t\right)\frac{\Theta\left(\phi\left(x\right)\right)}{\Theta\left(\phi\left(y\right)\right)}\phi^{\prime}\left(y\right).\label{eq:def of p^k-approx.}
\end{equation}
Then Proposition \ref{prop: approximation for q^V_J} can be rewritten
as follows.
\begin{cor}
\label{cor:approximation for p_I}There exists $t_{I}>0$ such that
for every $t\in\left(0,t_{I}\right)$ and $k\in\mathbb{N}\backslash\left\{ 0\right\} $,
\begin{equation}
\sup_{\left(x,y\right)\in\left(0,\psi\left(\frac{1}{9}\phi\left(I\right)\right)\right)^{2}}\left|\frac{p_{I}\left(x,y,t\right)-p^{k-approx.}\left(x,y,t\right)}{p^{approx.}\left(x,y,t\right)}\right|\leq m_{k}\left(t\right)M\left(t\right)+D_{k}\left(t\right)\exp\left(-\frac{2\phi\left(I\right)}{9t}\right)\label{eq: estimate of (p_I - p_k,approx.)/p_approx.}
\end{equation}
where $D_{k}\left(t\right)$ is as in (\ref{eq:def of D_k}). In particular,
\[
\begin{split}\sup_{\left(x,y\right)\in\left(0,\psi\left(\frac{1}{9}\phi\left(I\right)\right)\right)^{2}}\left|\frac{p_{I}\left(x,y,t\right)}{p^{approx.}\left(x,y,t\right)}-1\right| & =M\left(t\right)-1+D_{k}\left(t\right)\exp\left(-\frac{2\phi\left(I\right)}{9t}\right).\end{split}
\]
\end{cor}

\begin{rem}
We want to point out that, from now on, whenever $J=\phi\left(I\right)$,
the constants $\Theta_{J}$, $V_{J}$ and $t_{J}$ that were introduced
in $\mathsection3$ will also be written as $\Theta_{I}$, $V_{I}$
and $t_{I}$ respectively. In addition, by plugging $\phi\left(x\right)$
into (\ref{eq:def of V}) and (\ref{eq:formula of Theta}), we get
that 
\begin{equation}
V\left(\phi\left(x\right)\right)=-\frac{\theta^{2}\left(x\right)}{4\phi\left(x\right)}-\frac{\theta^{\prime}\left(x\right)}{2\phi^{\prime}\left(x\right)}+\frac{1-\nu}{2}\frac{\theta\left(x\right)}{\phi\left(x\right)}.\label{eq:formula of V(phi)}
\end{equation}
and 
\begin{equation}
\Theta\left(\phi\left(x\right)\right)=\begin{cases}
\frac{x^{\frac{\alpha}{4}}a^{\frac{1}{4}}\left(x\right)}{2^{\frac{\alpha}{2\left(2-\alpha\right)}}\left(\phi\left(x\right)\right)^{\frac{\alpha}{4\left(2-\alpha\right)}}}\exp\left(-\int_{0}^{x}\frac{b\left(w\right)}{2w^{\alpha}a\left(w\right)}dw\right) & \text{ if }\alpha\neq1,\\
\frac{x^{\frac{1}{4}}a^{\frac{1}{4}}\left(x\right)}{2^{\frac{1}{2}-b_{0}}\left(\phi\left(x\right)\right)^{\frac{1}{4}-\frac{1}{2}b_{0}}x^{\frac{1}{2}b_{0}}}\exp\left(-\int_{0}^{x}\frac{1}{2w}\left(\frac{b\left(w\right)}{a\left(w\right)}-b\left(0\right)\right)dw\right) & \text{ if }\alpha=1.
\end{cases}\label{eq:formula of Theta(phi)}
\end{equation}
With (\ref{eq:formula of Theta(phi)}) and (\ref{eq:formula of V(phi)}),
it is possible to rewrite some of the expressions that appeared above
(e.g., (\ref{eq:def of p_I}) and (\ref{eq:symmetry of p_I})) in
a more explicit way, see, e.g., (\ref{eq:transformation factor between q^V and p})
and (\ref{eq:symmetry factor for p}) in the Appendix. Especially
when $b\left(x\right)\equiv0$, these expressions take much simpler
forms than in the general case, as we will see with a concrete example
in $\mathsection5$. 
\end{rem}

\section{Global Equation}

In the previous section we have solved the localized equation (\ref{eq:general IVP localized})
and obtained its fundamental solution $p_{I}\left(x,y,t\right)$.
Now we proceed with the last step to complete our project, which is
to build the ``link'' between (\ref{eq:general IVP localized})
and the original problem (\ref{eq:IVP general equation}). To achieve
this goal, we rely on the strong Markov property of $\left\{ X\left(x,t\right):t\geq0\right\} $
and the probabilistic interpretations of the solutions found in the
previous sections. 

\subsection{From $p_{I}\left(x,y,t\right)$ to $p\left(x,y,t\right)$}

We introduce two more notations for this section: given $I>0$,
\[
a_{I}:=\max_{x\in\left[0,I\right]}\left\{ \frac{1}{a\left(x\right)},a\left(x\right)\right\} \text{ and }b_{I}:=\max_{x\in\left[0,I\right]}\left|b\left(x\right)\right|.
\]
Our first task is to derive probability estimates for the hitting
times of $\left\{ X\left(x,t\right):t\geq0\right\} $. 
\begin{lem}
\label{lem:prob estimate on hitting time of X} We define, for $x\geq0$,
\begin{equation}
S\left(x\right):=\int_{0}^{\phi\left(x\right)}w^{-\nu}\Theta^{2}\left(w\right)dw.\label{eq:def of S(x)}
\end{equation}
Then, for every $0<x<y\leq I$,
\begin{equation}
\mathbb{P}\left(\zeta_{y}^{X}\left(x\right)<\zeta_{0}^{X}\left(x\right)\right)=\frac{S\left(x\right)}{S\left(y\right)};\label{eq:prob of X hitting y before 0}
\end{equation}
if $\Theta_{I}$ is the constant found in Lemma \ref{lem:estimates on VJ}
(upon identifying $\Theta_{J}$ with $\Theta_{I}$ for $J=\phi\left(I\right)$),
then 
\begin{equation}
\mathbb{P}\left(\zeta_{y}^{X}\left(x\right)<\zeta_{0}^{X}\left(x\right)\right)\leq\Theta_{I}^{4}\left(\frac{\phi\left(x\right)}{\phi\left(y\right)}\right)^{1-\nu}.\label{eq:prob estimate for X hitting H before absorbed}
\end{equation}

Moreover, for every $G\in\left(0,I\right)$, $x\in\left(0,G\right)$
and $t>0$ such that $I-G>tb_{I}$, we have that 
\begin{equation}
\mathbb{P}\left(\zeta_{I}^{X}\left(x\right)\leq t\right)\leq\exp\left(-\frac{\left(I-x-tb_{I}\right)^{2}}{4tI^{\alpha}a_{I}}\right).\label{eq:prob estimate for hitting time of X}
\end{equation}
\end{lem}

\begin{proof}
We use Itô's formula to verify that, for every $y>x$, $\left\{ S\left(X\left(x,t\wedge\zeta_{0,y}^{X}\left(x\right)\right)\right)\right\} $
is a bounded martingale, and hence (\ref{eq:prob of X hitting y before 0})
follows immediately. Further, by (\ref{eq:bound Theta_J}), we have
that for every $x\in\left(0,I\right)$,
\begin{equation}
\Theta_{I}^{-2}\frac{\left(\phi\left(x\right)\right)^{1-\nu}}{1-\nu}\leq S\left(x\right)\leq\Theta_{I}^{2}\frac{\left(\phi\left(x\right)\right)^{1-\nu}}{1-\nu},\label{eq:estimate of S}
\end{equation}
which leads to (\ref{eq:prob estimate for X hitting H before absorbed}).

Now we get down to proving (\ref{eq:prob estimate for hitting time of X}),
and the proof is similar to that of Lemma \ref{lem: prob estimate for hitting time of Y}.
For every $\lambda\geq0$ and $t\geq0$, 
\[
\left\{ \exp\left(\lambda X\left(x,t\wedge\zeta_{I}^{X}\left(x\right)\right)-\lambda\int_{0}^{t\wedge\zeta_{I}^{X}\left(x\right)}b\left(X\left(x,s\right)\right)ds-\lambda^{2}\int_{0}^{t\wedge\zeta_{I}^{X}\left(x\right)}X^{\alpha}\left(x,s\right)a\left(X\left(x,s\right)\right)ds\right):t\geq0\right\} ,
\]
is a bounded martingale, from where we get that
\[
\mathbb{E}\left[\exp\left(-\lambda\int_{0}^{\zeta_{I}^{X}\left(x\right)}b\left(X\left(x,s\right)\right)ds-\lambda^{2}\int_{0}^{\zeta_{I}^{X}\left(x\right)}X^{\alpha}\left(x,s\right)a\left(X\left(x,s\right)\right)ds\right);\zeta_{I}^{X}\left(x\right)<\infty\right]\leq\exp\left(\lambda x-\lambda I\right).
\]
Since
\[
\left|\lambda\int_{0}^{\zeta_{I}^{X}\left(x\right)}b\left(X\left(x,s\right)\right)ds+\lambda^{2}\int_{0}^{\zeta_{I}^{X}\left(x\right)}X^{\alpha}\left(x,s\right)a\left(X\left(x,s\right)\right)ds\right|\leq\left(\lambda b_{I}+\lambda^{2}I^{\alpha}a_{I}\right)\zeta_{I}^{X}\left(x\right),
\]
we further have that
\begin{equation}
\mathbb{E}\left[\exp\left(-\left(\lambda b_{I}+\lambda^{2}I^{\alpha}a_{I}\right)\zeta_{I}^{X}\left(x\right)\right);\zeta_{I}^{X}\left(x\right)<\infty\right]\leq\exp\left(\lambda x-\lambda I\right).\label{eq:laplace transform of hitting time of X}
\end{equation}
By Markov's inequality, 
\[
\mathbb{P}\left(\zeta_{I}^{X}\left(x\right)\leq t\right)=\mathbb{P}\left(e^{-\left(\lambda b_{I}+\lambda^{2}I^{\alpha}a_{I}\right)\zeta_{I}^{X}\left(x\right)}\geq e^{-\left(\lambda b_{I}+\lambda^{2}I^{\alpha}a_{I}\right)t}\right)\leq e^{\lambda^{2}tI^{\alpha}a_{I}-\lambda\left(I-x-tb_{I}\right)}.
\]
(\ref{eq:prob estimate for hitting time of X}) is obtained by minimizing
the right hand side above over $\lambda\geq0$.
\end{proof}
Next, we consider $\left\{ p_{I}\left(x,y,t\right):I>0\right\} $
as a family parametrized by $I$, and for every $0<I<H$, we want
to find out the link between $p_{I}\left(x,y,t\right)$ and $p_{H}\left(x,y,t\right)$,
i.e., the fundamental solutions to (\ref{eq:general IVP localized})
with the right boundary at $I$ and $H$ respectively. To this end,
we choose a third constant $G\in\left(0,I\right)$ and define for
each $x\in\left(0,G\right)$ a sequence of hitting times of $\left\{ X\left(x,t\right):t\geq0\right\} $
where $\eta_{0}\left(x\right):=0$ and for $n\in\mathbb{N}\backslash\left\{ 0\right\} $,
\begin{equation}
\eta_{2n-1}\left(x\right):=\inf\left\{ s\geq\eta_{2n-2}\left(x\right):X\left(x,s\right)\geq I\right\} ,\eta_{2n}\left(x\right):=\inf\left\{ s\geq\eta_{2n-1}\left(x\right):X\left(x,s\right)\leq G\right\} .\label{eq:def of eta_n}
\end{equation}
In other words, the sequence $\left\{ \eta_{n}\left(x\right):n\in\mathbb{N}\right\} $
records the \emph{downward crossings of $\left\{ X\left(x,t\right):t\geq0\right\} $
from $I$ to $G$}. With the help of $\left\{ \eta_{n}\left(x\right):n\in\mathbb{N}\right\} $
and the strong Markov property of $\left\{ X\left(x,t\right):t\geq0\right\} $,
we are able to connect $p_{H}\left(x,y,t\right)$ and $p_{I}\left(x,y,t\right)$
as follows.
\begin{prop}
\label{prop:relation between p and p^theta}For $\left(x,y,t\right)\in\left(0,G\right)^{2}\times\left(0,\infty\right)$,
\begin{equation}
\begin{split}p_{H}\left(x,y,t\right) & =p_{I}\left(x,y,t\right)+\sum_{n=1}^{\infty}\mathbb{E}\left[p_{I}\left(G,y,t-\eta_{2n}\left(x\right)\right);\eta_{2n}\left(x\right)\leq t,\eta_{2n}\left(x\right)<\zeta_{0,H}^{X}\left(x\right)\right]\end{split}
.\label{eq:relation between p_I and p_H}
\end{equation}
\end{prop}

\begin{proof}
Given $f\in C_{c}\left(\left(0,G\right)\right)$, we use (\ref{eq: prob interp. of u_f,(0,I)})
to write 
\[
\int_{0}^{G}f\left(y\right)p_{H}\left(x,y,t\right)dy=\mathbb{E}\left[f\left(X\left(x,t\right)\right);t<\zeta_{0,H}^{X}\left(x\right)\right].
\]
According to the number of downward crossings (from $I$ to $G$)
completed by $\left\{ X\left(x,s\right):0\leq s\leq t\right\} $,
we further decompose $\mathbb{E}\left[f\left(X\left(x,t\right)\right);t<\zeta_{0,H}^{X}\left(x\right)\right]$
as
\[
\mathbb{E}\left[f\left(X\left(x,t\right)\right);t<\zeta_{0,I}^{X}\left(x\right)\right]+\sum_{n=1}^{\infty}\mathbb{E}\left[f\left(X\left(x,t\right)\right);\eta_{2n}\left(x\right)\leq t<\eta_{2n+1}\left(x\right)\wedge\zeta_{0,H}^{X}\left(x\right),\eta_{2n}\left(x\right)<\zeta_{0,H}^{X}\left(x\right)\right].
\]
By the strong Markov property of $X\left(x,t\right)$, we have that
for each $n\geq1$, 
\[
\begin{array}{c}
\mathbb{E}\left[f\left(X\left(x,t\right)\right);\eta_{2n}\left(x\right)\leq t<\eta_{2n+1}\left(x\right)\wedge\zeta_{0,H}^{X}\left(x\right),\eta_{2n}\left(x\right)<\zeta_{0,H}^{X}\left(x\right)\right]\\
\hspace{1cm}\hspace{1cm}\hspace{1cm}\hspace{1cm}=\mathbb{E}\left[\int_{0}^{G}f\left(y\right)p_{I}\left(G,y,t-\eta_{2n}^{X}\left(x\right)\right);\eta_{2n}\left(x\right)\leq t,\eta_{2n}\left(x\right)<\zeta_{0,H}^{X}\left(x\right)\right].
\end{array}
\]
On one hand, by (\ref{eq: estimate of q}), (\ref{eq:exp estimate for q^V_J})
and (\ref{eq:def of p_I}), 
\begin{align}
p_{I}\left(G,y,t-\eta_{2n}\left(x\right)\right) & \leq M\left(t\right)\frac{\left(\phi\left(G\right)\right)^{1-\nu}}{\left(t-\eta_{2n}\left(x\right)\right)^{2-\nu}}\exp\left(-\frac{\left(\sqrt{\phi\left(G\right)}-\sqrt{\phi\left(y\right)}\right)^{2}}{t-\eta_{2n}\left(x\right)}\right)\frac{\Theta\left(\phi\left(G\right)\right)}{\Theta\left(\phi\left(y\right)\right)}\phi^{\prime}\left(y\right)\nonumber \\
 & \leq M\left(t\right)\left(\frac{2-\nu}{e}\right)^{2-\nu}\left(\phi\left(G\right)\right)^{1-\nu}\left(\sqrt{\phi\left(G\right)}-\sqrt{\phi\left(y\right)}\right)^{2\left(\nu-2\right)}\frac{\Theta\left(\phi\left(G\right)\right)}{\Theta\left(\phi\left(y\right)\right)}\phi^{\prime}\left(y\right).\label{eq:estimate for p_I}
\end{align}
On the other hand, if $\eta_{2n}\left(x\right)<\zeta_{0,H}^{X}\left(x\right)$,
then it must be that (i) $\zeta_{G}^{X}\left(x\right)<\zeta_{0}^{X}\left(x\right)$,
(ii) during the time interval $\left[\zeta_{G}^{X}\left(x\right),\eta_{1}\left(x\right)\right]$,
the process starts from $G$ and hits $I$ before $0$, and (iii)
for each $j=0,\cdots,n-1$, during the time interval $\left[\eta_{2j}\left(x\right),\eta_{2j+1}\left(x\right)\right]$,
the process starts from $G$ and hits $I$ before $0$. Hence, by
(\ref{eq:prob of X hitting y before 0}) and the strong Markov property
of $X\left(x,t\right)$, we have that 
\begin{equation}
\mathbb{P}\left(\eta_{2n}\left(x\right)<\zeta_{0,H}^{X}\left(x\right)\right)\leq\mathbb{P}\left(\zeta_{G}^{X}\left(x\right)<\zeta_{0}^{X}\left(x\right)\right)\left(\mathbb{P}\left(\zeta_{I}^{X}\left(G\right)<\zeta_{0}^{X}\left(G\right)\right)\right)^{n}=\frac{S\left(x\right)}{S\left(G\right)}\left(\frac{S\left(G\right)}{S\left(I\right)}\right)^{n}.\label{eq:prob of eta_2n < zeta_0,H}
\end{equation}
Combining the above, we obtain that for every $\left(x,y,t\right)\in\left(0,G\right)^{2}\times\left(0,\infty\right)$,
\[
\begin{split} & \sum_{n=1}^{\infty}\mathbb{E}\left[p_{I}\left(G,y,t-\eta_{2n}\left(x\right)\right);\eta_{2n}\left(x\right)\leq t,\eta_{2n}\left(x\right)<\zeta_{0}^{X}\left(x\right)\right]\\
\leq & M\left(t\right)\left(\frac{2-\nu}{e}\right)^{2-\nu}\left(\phi\left(G\right)\right)^{1-\nu}\left(\sqrt{\phi\left(G\right)}-\sqrt{\phi\left(y\right)}\right)^{2\left(\nu-2\right)}\frac{\Theta\left(\phi\left(G\right)\right)}{\Theta\left(\phi\left(y\right)\right)}\phi^{\prime}\left(y\right)\frac{S\left(x\right)}{S\left(I\right)-S\left(G\right)}.
\end{split}
\]
This guarantees that the series in the right hand of (\ref{eq:relation between p_I and p_H})
is absolutely convergent.
\end{proof}
With Lemma \ref{lem:prob estimate on hitting time of X} and Proposition
\ref{prop:relation between p and p^theta}, we are ready to prove
our main result. 
\begin{thm}
\label{thm:main theorem}For every $\left(x,y,t\right)\in\left(0,\infty\right)^{3}$,
we set
\[
p\left(x,y,t\right):=\lim_{I\nearrow\infty}p_{I}\left(x,y,t\right).
\]
Given $0<G<I<H$, let $\left\{ \eta_{n}\left(x\right):n\in\mathbb{N}\right\} $
be the sequence of hitting times defined as in (\ref{eq:def of eta_n})
(for the downward crossings of $\left\{ X\left(x,t\right):t\geq0\right\} $
from $I$ to $G$). Then, for every $\left(x,y,t\right)\in\left(0,G\right)^{2}\times\left(0,\infty\right)$,
\begin{equation}
\begin{split}p\left(x,y,t\right) & =p_{I}\left(x,y,t\right)+\sum_{n=1}^{\infty}\mathbb{E}\left[p_{I}\left(G,y,t-\eta_{2n}\left(x\right)\right);\eta_{2n}\left(x\right)\leq t,\eta_{2n}\left(x\right)<\zeta_{0}^{X}\left(x\right)\right]\end{split}
.\label{eq:relation between p and p_I}
\end{equation}

$p\left(x,y,t\right)$ is continuous on $\left(0,\infty\right)^{3}$,
and for every $\left(x,y,t\right)\in\left(0,\infty\right)^{3}$,
\begin{equation}
\frac{\left(\phi\left(y\right)\right)^{1-\nu}}{\phi^{\prime}\left(y\right)}\Theta^{2}\left(\phi\left(y\right)\right)p\left(x,y,t\right)=\frac{\left(\phi\left(x\right)\right)^{1-\nu}}{\phi^{\prime}\left(x\right)}\Theta^{2}\left(\phi\left(x\right)\right)p\left(y,x,t\right).\label{eq:symmetry of p}
\end{equation}
For every $y>0$, $\left(x,t\right)\mapsto p\left(x,y,t\right)$ is
a smooth solution to the Kolmogorov backward equation corresponding
to $L$, i.e., 
\[
\partial_{t}p\left(x,y,t\right)=x^{\alpha}a\left(x\right)\partial_{x}^{2}p\left(x,y,t\right)+b\left(x\right)\partial_{x}p\left(x,y,t\right);
\]
for every $x>0$, $\left(y,t\right)\mapsto p\left(x,y,t\right)$ is
a smooth solution to the Kolmogorov forward equation corresponding
to $L$, i.e., 
\[
\partial_{t}p\left(x,y,t\right)=\partial_{y}^{2}\left(y^{\alpha}a\left(y\right)p\left(x,y,t\right)\right)-\partial_{y}\left(b\left(y\right)p\left(x,y,t\right)\right).
\]

$p\left(x,y,t\right)$ is the fundamental solution to (\ref{eq:IVP general equation}).
Given $f\in C_{c}\left(\left(0,\infty\right)\right)$, 
\[
u_{f}\left(x,t\right):=\int_{0}^{\infty}f\left(y\right)p\left(x,y,t\right)dy\text{ for }\left(x,t\right)\in\left(0,\infty\right)^{2}
\]
is the unique solution in $C^{2,1}\left(\left(0,\infty\right)^{2}\right)$
to (\ref{eq:IVP general equation}), and $u_{f}\left(x,t\right)$
is smooth on $\left(0,\infty\right)^{2}$. Moreover, for every $\left(x,t\right)\in\left(0,\infty\right)^{2}$,

\[
u_{f}\left(x,t\right)=\mathbb{E}\left[f\left(X\left(x,t\right)\right);t<\zeta_{0}^{X}\left(x\right)\right],
\]
and hence for every Borel set $\Gamma\subseteq\left(0,\infty\right)$,
\[
\int_{\Gamma}p\left(x,y,t\right)dy=\mathbb{P}\left(X\left(x,t\right)\in\Gamma,t<\zeta_{0}^{X}\left(x\right)\right).
\]

Finally, $p\left(x,y,t\right)$ satisfies the Chapman-Kolmogorov equation,
i.e., for every $x,y>0$ and $t,s>0$,
\begin{equation}
p\left(x,y,t+s\right)=\int_{0}^{\infty}p\left(x,\xi,t\right)p\left(\xi,y,s\right)d\xi.\label{eq:CK equation for p}
\end{equation}
\end{thm}

\begin{proof}
It is clear from (\ref{eq:relation between p_I and p_H}) that for
every $\left(x,y,t\right)\in\left(0,\infty\right)^{3}$, by taking
$G>x\vee y$, we know that the family $I\in\left(G,\infty\right)\mapsto p_{I}\left(x,y,t\right)$
is non-decreasing, so $p\left(x,y,t\right)$ as the limit of $p_{I}\left(x,y,t\right)$
(as $I\nearrow\infty$) is well defined. Since $\left\{ X\left(x,t\right):t\geq0\right\} $
is the unique solution to (\ref{eq:SDE of X}), $\zeta_{0,H}^{X}\left(x\right)\rightarrow\zeta_{0}^{X}\left(x\right)$
almost surely as $H\nearrow\infty$ (see, e.g., $\mathsection10$
of \cite{multi_dim_diff_proc}). Thus, (\ref{eq:relation between p and p_I})
follows from (\ref{eq:relation between p_I and p_H}) by sending $H$
to infinity, and (\ref{eq:symmetry of p}) follows from (\ref{eq:symmetry of p_I}). 

Now we examine the continuity of $p\left(x,y,t\right)$. First, (\ref{eq:estimate for p_I})
and (\ref{eq:prob of eta_2n < zeta_0,H}) guarantee that the series
in the right hand side of (\ref{eq:relation between p and p_I}) converges
uniformly on any bounded subset of $\left(0,G\right)^{2}\times\left(0,\infty\right)$,
from where it is easy to see that for every $x\in\left(0,G\right)$,
$\left(y,t\right)\mapsto p\left(x,y,t\right)$ is continuous on $\left(0,G\right)\times\left(0,\infty\right)$.
Furthermore in the proof of Proposition \ref{prop:properties of q^V_J}
we have seen that $x\mapsto p_{I}\left(G,x,s\right)$ is equicontinuous
in $s$ from any bounded subset of $\left(0,\infty\right)$, which,
combined with (\ref{eq:symmetry of p}), leads to the continuity of
$p\left(x,y,t\right)$ in all three variables.

Next, we turn our attention to $u_{f}\left(x,t\right)$ for $f\in C_{c}\left(\left(0,\infty\right)\right)$.
It is clear that
\begin{align*}
u_{f}\left(x,t\right) & =\lim_{I\nearrow\infty}\int_{0}^{I}f\left(y\right)p_{I}\left(x,y,t\right)dy=\lim_{I\nearrow\infty}u_{f,\left(0,I\right)}\left(x,t\right),
\end{align*}
and further by (\ref{eq: prob interp. of u_f,(0,I)}), 
\[
u_{f}\left(x,t\right)=\lim_{I\nearrow\infty}\mathbb{E}\left[f\left(X\left(x,t\right)\right);t<\zeta_{0,I}^{X}\left(x\right)\right]=\mathbb{E}\left[f\left(X\left(x,t\right)\right);t<\zeta_{0}^{X}\left(x\right)\right],
\]
which means that $p\left(x,y,t\right)$ is indeed the probability
density function of $X\left(x,t\right)$ provided that $t<\zeta_{0}^{X}\left(x\right)$.
(\ref{eq:CK equation for p}) follows from the strong Markov property
of $\left\{ X\left(x,t\right):t\geq0\right\} $. Furthermore, by (\ref{eq:relation between p and p_I}),
if $G$ and $I$ are sufficiently large such that $x\in\left(0,G\right)$
and $\text{supp}\left(f\right)\subseteq\left(0,I\right)$, then
\[
u_{f}\left(x,t\right)=u_{f,\left(0,I\right)}\left(x,t\right)+\sum_{n=1}^{\infty}\mathbb{E}\left[u_{f,\left(0,I\right)}\left(G,t-\eta_{2n}\left(x\right)\right);\eta_{2n}\left(x\right)\leq t,\eta_{2n}\left(x\right)<\zeta_{0}^{X}\left(x\right)\right].
\]
Let us re-examine the event $\left\{ \eta_{2n}\left(x\right)\leq t,\eta_{2n}\left(x\right)<\zeta_{0}^{X}\left(x\right)\right\} $
involved in the series above. If $\eta_{2n}\left(x\right)\leq t$,
then we must have that $\zeta_{G}^{X}\left(x\right)<\zeta_{0}^{X}\left(x\right)$,
$\eta_{1}\left(x\right)-\zeta_{G}^{X}\left(x\right)\leq t$, and for
each $j=0,\cdots,n-1$, $\eta_{2j+1}\left(x\right)-\eta_{2j}\left(x\right)\leq t$.
Thus,
\[
\mathbb{P}\left(\eta_{2n}\left(x\right)\leq t,\eta_{2n}\left(x\right)<\zeta_{0}^{X}\left(x\right)\right)\leq\mathbb{P}\left(\zeta_{G}^{X}\left(x\right)<\zeta_{0}^{X}\left(x\right)\right)\left(\mathbb{P}\left(\zeta_{I}^{X}\left(G\right)<t\right)\right)^{n}.
\]
By (\ref{eq:prob of X hitting y before 0}) and (\ref{eq:prob estimate for hitting time of X}),
we have that when $I-G>tb_{I}$,
\begin{equation}
\mathbb{P}\left(\eta_{2n}\left(x\right)\leq t,\eta_{2n}\left(x\right)<\zeta_{0}^{X}\left(x\right)\right)\leq\frac{S\left(x\right)}{S\left(G\right)}\exp\left(-\frac{n\left(I-G-tb_{I}\right)^{2}}{4tI^{\alpha}a_{I}}\right).\label{eq:prob of eta_2n < min =00007Bt,zeta_0,H=00007D}
\end{equation}
Therefore, when $t$ is sufficiently small,
\[
\begin{split} & \left|\sum_{n=1}^{\infty}\mathbb{E}\left[u_{f,\left(0,I\right)}\left(G,t-\eta_{2n}\left(x\right)\right);\eta_{2n}\left(x\right)\leq t,\eta_{2n}\left(x\right)<\zeta_{0}^{X}\left(x\right)\right]\right|\\
\leq & \left\Vert f\right\Vert _{u}\sum_{n=1}^{\infty}\mathbb{P}\left(\eta_{2n}\left(x\right)\leq t,\eta_{2n}\left(x\right)<\zeta_{0}^{X}\left(x\right)\right)\\
\leq & \left\Vert f\right\Vert _{u}\frac{S\left(x\right)}{S\left(G\right)}\exp\left(-\frac{\left(I-G-tb_{I}\right)^{2}}{4tI^{\alpha}a_{I}}\right)\frac{4tI^{\alpha}\alpha_{I}}{\left(I-G-tb_{I}\right)^{2}}
\end{split}
\]
which tends to $0$ as $t\searrow0$ or as $x\searrow0$. Therefore,
we have that
\[
\lim_{x\searrow0}u_{f}\left(x,t\right)=\lim_{x\searrow0}u_{f,\left(0,I\right)}\left(x,t\right)=0\text{ and }\lim_{t\searrow0}u_{f}\left(x,t\right)=\lim_{t\searrow0}u_{f,\left(0,I\right)}\left(x,t\right)=f\left(x\right).
\]

The only remaining thing is to prove the statement on $p\left(x,y,t\right)$
and $u_{f}\left(x,t\right)$ being smooth solutions to the concerned
equations, which, again, by the hypoellipticity of $\partial_{t}-L$,
is reduced to showing that they are distribution solutions. Take $u_{f}\left(x,t\right)$
for instance. We observe that for any $\varphi\in C_{c}^{\infty}\left(\left(0,\infty\right)\right)$,
\[
\begin{split}\left\langle \varphi,u_{f}\left(\cdot,t\right)\right\rangle  & =\lim_{I\nearrow\infty}\int_{0}^{\infty}\varphi\left(x\right)u_{f,\left(0,I\right)}\left(x,t\right)dx\\
 & =\left\langle \varphi,f\right\rangle +\lim_{I\nearrow\infty}\int_{0}^{t}\int_{0}^{\infty}\varphi\left(x\right)Lu_{f,\left(0,I\right)}\left(x,s\right)dxds\\
 & =\left\langle \varphi,f\right\rangle +\lim_{I\nearrow\infty}\int_{0}^{t}\int_{0}^{\infty}\left(L^{*}\varphi\right)\left(x\right)u_{f,\left(0,I\right)}\left(x,s\right)dxds\\
 & =\left\langle \varphi,f\right\rangle +\int_{0}^{t}\int_{0}^{\infty}\left(L^{*}\varphi\right)\left(x\right)u_{f}\left(x,t\right)dxds,
\end{split}
\]
which implies that 
\[
\frac{d}{dt}\left\langle \varphi,u_{f}\left(\cdot,t\right)\right\rangle =\left\langle L^{*}\varphi,u_{f}\left(\cdot,t\right)\right\rangle .
\]
This confirms that $u_{f}\left(x,t\right)$ is a solution to (\ref{eq:IVP general equation})
as a distribution. The statements on $p\left(x,y,t\right)$ follow
from similar arguments.
\end{proof}
\begin{rem}
We want to point out that the function $S\left(x\right)$ defined
in (\ref{eq:def of S(x)}) has a specific role in the boundary classification
for diffusion process. In fact, $S\left(x\right)$ is the \emph{scale
function} for the underlying diffusion process corresponding to $L$,
and as $x$ approaches a boundary, whether $S\left(x\right)$ remains
bounded or not is a factor in boundary classification (see $\mathsection15.6$
of \cite{Karlin_Taylor}). In particular, when viewing $\infty$ as
a boundary of $\left(0,\infty\right)$, we introduce the \emph{escape
probability} at $G>0$ (escaping from $G$ to $\infty$) as 
\begin{equation}
\mathfrak{p}_{G}:=\lim_{I\rightarrow\infty}\mathbb{P}\left(\zeta_{I}^{X}\left(G\right)<\zeta_{0}^{X}\left(G\right)\right).\label{eq:escaping probability at G}
\end{equation}
Then, when $\lim_{x\rightarrow\infty}S\left(x\right)=\infty$, $\infty$
is \emph{non-attracting,} in which case (\ref{eq:prob of X hitting y before 0})
implies that $\mathfrak{p}_{G}=0$; when $\lim_{x\rightarrow\infty}S\left(x\right)<\infty$,
$\infty$ is \emph{attracting} and $\mathfrak{p}_{G}>0$. 
\end{rem}

\subsection{Approximation of $p\left(x,y,t\right)$}

In the previous sections, for the fundamental solutions that do not
have explicit formulas, we provide approximations that are accessible
and of high accuracy, at least for small time. These approximations
can be useful in computational applications of degenerate diffusion
equations studied in this work. Below we will present an approximation
for $p\left(x,y,t\right)$ in the same spirit. In particular, we find
explicitly defined approximations to $p\left(x,y,t\right)$ such that
(i) these approximations are more accurate than the standard heat
kernel estimates, and (ii) when $t$ is sufficiently small, these
approximations are ``close'' to $p\left(x,y,t\right)$ uniformly
in $\left(x,y\right)$ in any compact set. Note that this result is
a generalization of \cite{deg_diff_global} for that the error estimates
we derive here only depend on the local bounds of $a\left(x\right)$
and $b\left(x\right)$.
\begin{thm}
\label{thm:approximation of p(x,y,t) }Let $p^{approx.}\left(x,y,t\right)$
and $p^{k-approx.}\left(x,y,t\right)$, $k\in\mathbb{N}\backslash\left\{ 0\right\} $,
be defined as in (\ref{eq:def of p^approx.}) and (\ref{eq:def of p^k-approx.})
respectively. For any $G>0$, set $t_{G}:=\frac{4\phi\left(G\right)}{9\left(2-\nu\right)}$.
Then, for every $t\in\left(0,t_{G}\right)$, $I>G$ and $k\in\mathbb{N}\backslash\left\{ 0\right\} $,
\begin{equation}
\begin{split} & \sup_{\left(x,y\right)\in\left(0,\psi\left(\frac{\phi\left(G\right)}{9}\right)\right)^{2}}\left|\frac{p\left(x,y,t\right)-p^{k-approx.}\left(x,y,t\right)}{p^{approx.}\left(x,y,t\right)}\right|\\
 & \hspace{0.5cm}\hspace{0.5cm}\hspace{0.5cm}\leq m_{k}\left(t\right)M\left(t\right)+\left[D_{k}\left(t\right)+\frac{\Theta_{G}^{4}M\left(t\right)\left(\phi\left(G\right)\right)^{1-\nu}}{\left(1-\nu\right)\left(S\left(I\right)-S\left(G\right)\right)}\right]\exp\left(-\frac{2\phi\left(G\right)}{9t}\right),
\end{split}
\label{eq:estimate p/p^k-approx}
\end{equation}
where $m_{k}\left(t\right)$, $M\left(t\right)$ and $D_{k}\left(t\right)$
are as in (\ref{eq:def of m_n(t)}) and (\ref{eq:def of D_k}) respectively.

In particular, there exists constant $C>0$ uniformly in $t\in\left(0,t_{G}\right)$
and $k\in\mathbb{N}\backslash\left\{ 0\right\} $ ($C$ may depend
on $G$ and $\nu$) such that 
\[
\sup_{\left(x,y\right)\in\left(0,\psi\left(\frac{\phi\left(G\right)}{9}\right)\right)^{2}}\left|\frac{p\left(x,y,t\right)-p^{k-approx.}\left(x,y,t\right)}{p^{approx.}\left(x,y,t\right)}\right|\leq Ct^{k\mathfrak{b}},
\]
where $\mathfrak{b}$ is the constant defined in (\ref{eq:def of mathfrak(b)}).
\end{thm}

\begin{proof}
Only (\ref{eq:estimate p/p^k-approx}) requires proof. For every $\left(x,y,t\right)\in\left(0,G\right)^{2}\times\left(0,\infty\right)$,
we have that
\[
\begin{split}\left|\frac{p\left(x,y,t\right)-p^{k-approx.}\left(x,y,t\right)}{p^{approx.}\left(x,y,t\right)}\right| & \leq\left|\frac{p\left(x,y,t\right)-p_{I}\left(x,y,t\right)}{p^{approx.}\left(x,y,t\right)}\right|+\left|\frac{p_{I}\left(x,y,t\right)-p^{k-approx.}\left(x,y,t\right)}{p^{approx.}\left(x,y,t\right)}\right|.\end{split}
\]
By (\ref{eq: estimate of (p_I - p_k,approx.)/p_approx.}), we have
that for every $t\in\left(0,t_{G}\right)$, the second term on the
right hand side above is bounded uniformly in $\left(x,y\right)\in\left(0,\psi\left(\frac{\phi\left(G\right)}{9}\right)\right)^{2}$
by 
\[
m_{k}\left(t\right)M\left(t\right)+D_{k}\left(t\right)\exp\left(-\frac{2\phi\left(G\right)}{9t}\right).
\]
We define hitting times $\left\{ \eta_{n}\left(x\right):n\in\mathbb{N}\right\} $
as in (\ref{eq:def of eta_n}) (for the downward crossings from $I$
to $G$). Then, according to (\ref{eq:relation between p and p_I}),
\[
p\left(x,y,t\right)-p_{I}\left(x,y,t\right)=\sum_{n=1}^{\infty}\mathbb{E}\left[p_{I}\left(G,y,t-\eta_{2n}\left(x\right)\right);\eta_{2n}\left(x\right)\leq t,\eta_{2n}\left(x\right)<\zeta_{0}^{X}\left(x\right)\right].
\]
It follows from (\ref{eq: estimate of q}), (\ref{eq:exp estimate for q^V_J}),
(\ref{eq:def of p_I}) and (\ref{eq:prob of eta_2n < zeta_0,H}) that
for every $t\in\left(0,t_{G}\right)$ and $\left(x,y\right)\in\left(0,\psi\left(\frac{\phi\left(G\right)}{9}\right)\right)^{2}$,
\[
\begin{split} & \left|p\left(x,y,t\right)-p_{I}\left(x,y,t\right)\right|\\
\leq & M\left(t\right)\left(\sup_{s\in\left(0,t\right)}s^{\nu-2}e^{-\frac{4\phi\left(G\right)}{9s}}\right)\left(\phi\left(G\right)\right)^{1-\nu}\frac{\Theta\left(\phi\left(G\right)\right)}{\Theta\left(\phi\left(y\right)\right)}\phi^{\prime}\left(y\right)\frac{S\left(x\right)}{S\left(I\right)-S\left(G\right)}\\
\leq & M\left(t\right)t^{\nu-2}\left(\phi\left(G\right)\right)^{1-\nu}\frac{\Theta\left(\phi\left(G\right)\right)}{\Theta\left(\phi\left(y\right)\right)}\phi^{\prime}\left(y\right)\exp\left(-\frac{4\phi\left(G\right)}{9t}\right)\frac{S\left(x\right)}{S\left(I\right)-S\left(G\right)},
\end{split}
\]
and further by (\ref{eq:def of p^approx.}) and (\ref{eq:estimate of S})
we have that
\begin{align*}
\left|\frac{p\left(x,y,t\right)-p_{I}\left(x,y,t\right)}{p^{approx.}\left(x,y,t\right)}\right| & \leq M\left(t\right)\frac{\Theta\left(\phi\left(G\right)\right)}{\Theta\left(\phi\left(x\right)\right)}\left(\frac{\phi\left(G\right)}{\phi\left(x\right)}\right)^{1-\nu}\exp\left(-\frac{2\phi\left(G\right)}{9t}\right)\frac{S\left(x\right)}{S\left(I\right)-S\left(G\right)}\\
 & \le\frac{\Theta_{G}^{4}M\left(t\right)\left(\phi\left(G\right)\right)^{1-\nu}}{\left(1-\nu\right)\left(S\left(I\right)-S\left(G\right)\right)}\exp\left(-\frac{2\phi\left(G\right)}{9t}\right).
\end{align*}
\end{proof}
We close this section with two variations of (\ref{eq:estimate p/p^k-approx}).
First, by (\ref{eq:estimate of S}), we note that
\[
\frac{1}{S\left(I\right)-S\left(G\right)}=\frac{1}{S\left(G\right)}\frac{S\left(G\right)/S\left(I\right)}{1-S\left(G\right)/S\left(I\right)}\leq\Theta_{G}^{2}\frac{1-\nu}{\left(\phi\left(G\right)\right)^{1-\nu}}\frac{S\left(G\right)/S\left(I\right)}{1-S\left(G\right)/S\left(I\right)}.
\]
 Therefore, by sending $I$ to $\infty$ in (\ref{eq:estimate p/p^k-approx}),
we get the following estimate.
\begin{cor}
\label{cor:approx for p(x,y,t) variation 1}For every $G>0$, let
$t_{G}>0$ be the same as in Theorem \ref{thm:approximation of p(x,y,t) },
and $\mathfrak{p}_{G}$ be defined as in (\ref{eq:escaping probability at G}).
Then, for every $t\in\left(0,t_{G}\right)$,
\[
\begin{split} & \sup_{\left(x,y\right)\in\left(0,\psi\left(\frac{\phi\left(G\right)}{9}\right)\right)^{2}}\left|\frac{p\left(x,y,t\right)-p^{k-approx.}\left(x,y,t\right)}{p^{approx.}\left(x,y,t\right)}\right|\\
 & \hspace{0.5cm}\hspace{0.5cm}\hspace{0.5cm}\hspace{0.5cm}\leq m_{k}\left(t\right)M\left(t\right)+\left(D_{k}\left(t\right)+\Theta_{G}^{6}M\left(t\right)\frac{\mathfrak{p}_{G}}{1-\mathfrak{p}_{G}}\right)\exp\left(-\frac{2\phi\left(G\right)}{9t}\right).
\end{split}
\]
\end{cor}

Second, by making $t_{G}$ in Theorem \ref{thm:approximation of p(x,y,t) }
smaller if necessary, we can derive an estimate analogous to (\ref{eq:estimate p/p^k-approx})
but independent of $\mathfrak{p}_{G}$. Intuitively speaking, when
$t$ is sufficiently small, how well $p^{approx.}\left(x,y,t\right)$
approximates $p\left(x,y,t\right)$ should not depend on the probability
of the process escaping to infinity. To make it rigorous, we first
observe that \textbf{(H2)} guarantees the existence of $t_{G}^{\prime}>0$
such that 
\begin{equation}
I-G-t_{G}^{\prime}b_{I}>2\sqrt{t_{G}^{\prime}I^{\alpha}\alpha_{I}}\text{ for every }I>2G;\label{eq:choice of t_G}
\end{equation}
then, by using (\ref{eq:prob of eta_2n < min =00007Bt,zeta_0,H=00007D})
instead of (\ref{eq:prob of eta_2n < zeta_0,H}) in the proof of (\ref{eq:estimate p/p^k-approx}),
we get that for every $t\in\left(0,t_{G}^{\prime}\right)$ and $\left(x,y\right)\in\left(0,\psi\left(\frac{\phi\left(G\right)}{9}\right)\right)^{2}$,
$\left|p\left(x,y,t\right)-p_{I}\left(x,y,t\right)\right|$ is bounded
from above by 
\[
\begin{split} & M\left(t\right)t^{\nu-2}\left(\phi\left(G\right)\right)^{1-\nu}\frac{\Theta\left(\phi\left(G\right)\right)}{\Theta\left(\phi\left(y\right)\right)}\phi^{\prime}\left(y\right)\exp\left(-\frac{4\phi\left(G\right)}{9t}-\frac{\left(I-G-tb_{I}\right)^{2}}{4tI^{\alpha}a_{I}}\right)\frac{S\left(x\right)}{S\left(G\right)}\frac{4tI^{\alpha}\alpha_{I}}{\left(I-G-tb_{I}\right)^{2}}\\
\leq & M\left(t\right)t^{\nu-2}\left(\phi\left(G\right)\right)^{1-\nu}\frac{\Theta\left(\phi\left(G\right)\right)}{\Theta\left(\phi\left(y\right)\right)}\phi^{\prime}\left(y\right)\frac{S\left(x\right)}{S\left(G\right)}\exp\left(-\frac{4\phi\left(G\right)}{9t}\right).
\end{split}
\]
It follows that for every $\left(x,y,t\right)\in\left(0,\psi\left(\frac{\phi\left(G\right)}{9}\right)\right)^{2}\times\left(0,t_{G}^{\prime}\right)$,
\begin{align*}
\left|\frac{p\left(x,y,t\right)-p_{I}\left(x,y,t\right)}{p^{approx.}\left(x,y,t\right)}\right| & \leq M\left(t\right)\frac{\Theta\left(\phi\left(G\right)\right)}{\Theta\left(\phi\left(x\right)\right)}\left(\frac{\phi\left(G\right)}{\phi\left(x\right)}\right)^{1-\nu}\frac{S\left(x\right)}{S\left(G\right)}\exp\left(-\frac{2\phi\left(G\right)}{9t}\right)\\
 & \le\Theta_{G}^{6}M\left(t\right)\exp\left(-\frac{2\phi\left(G\right)}{9t}\right).
\end{align*}
Therefore, we have the following estimate on the error between $p^{k-approx.}\left(x,y,t\right)$
and $p\left(x,y,t\right)$, which is a potential improvement of (\ref{eq:estimate p/p^k-approx})
for small $t$.
\begin{cor}
\label{cor:approx for p(x,y,t) variation 2}For every $G>0$, let
$t_{G}^{\prime}>0$ be such that (\ref{eq:choice of t_G}) holds.
Then, for every $t\in\left(0,t_{G}^{\prime}\right)$,
\[
\begin{split} & \sup_{\left(x,y\right)\in\left(0,\psi\left(\frac{\phi\left(G\right)}{9}\right)\right)^{2}}\left|\frac{p\left(x,y,t\right)-p^{k-approx.}\left(x,y,t\right)}{p^{approx.}\left(x,y,t\right)}\right|\\
 & \hspace{0.5cm}\hspace{0.5cm}\hspace{0.5cm}\hspace{0.5cm}\leq m_{k}\left(t\right)M\left(t\right)+\left(D_{k}\left(t\right)+\Theta_{G}^{6}M\left(t\right)\right)\exp\left(-\frac{2\phi\left(G\right)}{9t}\right).
\end{split}
\]
\end{cor}

\section{Generalized Wright-Fisher Diffusion}

As reviewed in $\mathsection1.1$, the classical Wright-Fisher diffusion
equation given by (\ref{eq: classical WF eq}) has two degenerate
boundaries at $0$ and $1$, and the localization method was adopted
in \cite{siamwfeq} so that one only needs to focus on one boundary
at a time. Although in our setting only degenerate diffusions with
one-sided boundary are concerned, the framework developed in the previous
sections can also be applied to degenerate diffusions with two-sided
boundaries. In this section we discuss a variation of the Wright-Fisher
diffusion where the diffusion operator has general order of degeneracy
at both boundaries 0 and 1.

For two constants $\alpha,\beta\in\left(0,2\right)$, we consider
the following Cauchy problem with two-sided boundaries on $\left(0,1\right)$,
where, given $f\in C_{b}\left(\left(0,1\right)\right)$, we look for
$u_{f}\left(x,t\right)\in C^{2,1}\left(\left(0,1\right)\times\left(0,\infty\right)\right)$
such that
\begin{equation}
\begin{array}{c}
\partial_{t}u_{f}\left(x,t\right)=x^{\alpha}\left(1-x\right)^{\beta}\partial_{x}^{2}u_{f}\left(x,t\right)\text{ for every }\left(x,t\right)\in\left(0,1\right)\times\left(0,\infty\right),\\
\lim_{t\searrow0}u_{f}\left(x,t\right)=f\left(x\right)\text{ for every }x\in\left(0,1\right),\\
\lim_{x\searrow0}u_{f}\left(x,t\right)=\lim_{x\nearrow1}u_{f}\left(x,t\right)=0\text{ for every }t\in\left(0,\infty\right).
\end{array}\label{eq: IVP 2sided boundary}
\end{equation}
Set $L_{\alpha,\beta}:=x^{\alpha}\left(1-x\right)^{\beta}\partial_{x}^{2}$.
We want to apply the method developed in the previous sections to
construct and study the fundamental solution $p\left(x,y,t\right)$
to (\ref{eq: IVP 2sided boundary}). $L_{\alpha,\beta}$ has two degenerate
boundaries 0 and 1 with (possibly distinct) general order of degeneracy,
and both boundaries are attainable according to the boundary classification
mentioned in Remark \ref{rem:boundary classification}. 

Although having a second degenerate boundary at $1$, $L_{\alpha,\beta}$
has the advantage that its coefficient $x^{\alpha}\left(1-x\right)^{\beta}$
is bounded on $\left(0,1\right)$. Therefore, for every $x\in\left(0,1\right)$,
the stochastic differential equation
\begin{equation}
dX\left(x,t\right)=\sqrt{2X^{\alpha}\left(x,t\right)\left(1-X\left(x,t\right)\right)^{\beta}}dB\left(t\right)\text{ with }X\left(x,0\right)\equiv x\label{eq:SDE two-sided bdry case}
\end{equation}
always has a solution in the sense described in $\mathsection1.3$
(see, e.g., of \cite{PDEStroock}). Although we are not yet ready
to claim the uniqueness of this solution, we can follow the theory
in $\mathsection12$ of \cite{PDEStroock} to extract a solution to
(\ref{eq:SDE two-sided bdry case}) that has strong Markov property.
In other words, (\ref{eq:SDE two-sided bdry case}) always has a solution
$\left\{ X\left(x,t\right):t\geq0\right\} $ that is a strong Markov
process. 

The existence of a strong Markovian solution to (\ref{eq:SDE two-sided bdry case})
enables us to follow the steps in $\mathsection2-\mathsection4$ to
tackle (\ref{eq: IVP 2sided boundary}). In particular, with the localization
procedure, we have the option of placing our ``focal point'' in
the neighborhood of either $0$ or $1$ while constructing $p\left(x,y,t\right)$.
We will see that these two views are consistent and will lead to the
same $p\left(x,y,t\right)$. 

Let us start with the construction of $p\left(x,y,t\right)$ with
a focus only on the left boundary $0$, and we will follow the steps
in the previous sections with $a\left(x\right)=\left(1-x\right)^{\beta}$
and $b\left(x\right)\equiv0$. Here we only state the results of each
step but leave the computational details in the Appendix (i.e., (\ref{eq:Theta_I(L) two-sided bdry case})-(\ref{eq:VI in two-sided bdry case})).
We add a superscript ``$^{(L)}$'' to relevant quantities and functions
to indicate that only the left boundary $0$ is ``effective'' in
this construction. 

We take $I\in\left(0,1\right)$ and localize (\ref{eq: IVP 2sided boundary})
onto $\left(0,I\right)$. All the functions involved in the transformation
are as follows:
\[
\phi^{(L)}\left(x\right)=\frac{1}{4}\left(b^{(L)}\left(x\right)\right)^{2}\text{ and }\theta^{(L)}\left(x\right)=\frac{\alpha}{2\left(2-\alpha\right)}-\frac{\alpha-\alpha x-\beta x}{4}x^{\frac{\alpha-2}{2}}\left(1-x\right)^{\frac{\beta-2}{2}}b^{(L)}\left(x\right),
\]
where $b^{(L)}\left(x\right):=\int_{0}^{x}s^{-\alpha/2}\left(1-s\right)^{-\beta/2}ds$
is the incomplete beta function; furthermore,
\[
\Theta\left(\phi^{(L)}\left(x\right)\right)=\frac{x^{\frac{\alpha}{4}}\left(1-x\right)^{\frac{\beta}{4}}}{\left(b^{(L)}\left(x\right)\right)^{\frac{\alpha}{2\left(2-\alpha\right)}}}\text{ with }\Theta_{I}^{(L)}=\left(\frac{2}{2-\alpha}\right)^{\frac{\alpha}{2\left(2-\alpha\right)}}\left(1-I\right)^{-\frac{\beta}{2\left(2-\alpha\right)}};
\]
in addition, for every $x\in\left(0,I\right)$,
\[
V\left(\phi^{(L)}\left(x\right)\right)=-\frac{\alpha\left(\alpha-4\right)}{4\left(2-\alpha\right)^{2}\left(b^{(L)}\left(x\right)\right)^{2}}+x^{\alpha-2}\left(1-x\right)^{\beta-2}\left(\frac{\left(\alpha-\alpha x-\beta x\right)^{2}}{16}-\frac{\alpha\left(1-x\right)^{2}+\beta x^{2}}{4}\right),
\]
and hence
\[
\left|V\left(\phi^{(L)}\left(x\right)\right)\right|\leq V_{I}^{(L)}\left(\phi^{(L)}\left(x\right)\right)^{-\frac{1-\alpha}{2-\alpha}}\text{ with }V_{I}^{(L)}=\frac{\beta}{16}\left(4-\beta+2\alpha\right)\left(1-I\right)^{\frac{\beta}{2-\alpha}-2}.
\]
This confirms that the statement in Lemma \ref{lem:estimates on VJ}
still holds in this case.

Next, for the model equation discussed in $\mathsection2.2$, we plug
in $\nu^{(L)}:=\frac{1-\alpha}{2-\alpha}$ and obtain $q^{\left(L\right)}\left(z,w,t\right)$
as in (\ref{eq:def of q}) and $q_{\phi^{(L)}\left(I\right)}^{(L)}\left(z,w,t\right)$
as in (\ref{eq:relation between q and q_J}) accordingly. We then
follow exactly the same steps as in $\mathsection3.1$ to derive $q_{\phi^{(L)}\left(I\right)}^{(L),V}\left(z,w,t\right)$
based on $q_{\phi^{(L)}\left(I\right)}^{(L)}\left(z,w,t\right)$,
and to obtain $p_{I}^{(L)}\left(x,y,t\right)$ through reversing the
transformation $z=\phi^{(L)}\left(x\right)$, i.e., 
\[
p_{I}^{(L)}\left(x,y,t\right)=q_{\phi^{(L)}\left(I\right)}^{(L),V}\left(\phi^{(L)}\left(x\right),\phi^{(L)}\left(y\right),t\right)\frac{x^{\frac{\alpha}{4}}\left(1-x\right)^{\frac{\beta}{4}}}{2y^{\frac{3\alpha}{4}}\left(1-y\right)^{\frac{3\beta}{4}}}\frac{\left(b^{(L)}\left(y\right)\right)^{\frac{4-\alpha}{2\left(2-\alpha\right)}}}{\left(b^{(L)}\left(x\right)\right)^{\frac{\alpha}{2\left(2-\alpha\right)}}}.
\]

\noindent To proceed, we follow the arguments in $\mathsection4$
to obtain the fundamental solution to (\ref{eq: IVP 2sided boundary})
as 
\[
p\left(x,y,t\right)=\lim_{I\nearrow1}p_{I}^{(L)}\left(x,y,t\right)\text{ for }\left(x,y,t\right)\in\left(0,1\right)^{2}\times\left(0,\infty\right).
\]

By (\ref{eq:SDE two-sided bdry case}), $\left\{ X\left(x,t\right):t\geq0\right\} $
itself is a martingale, and as in Lemma \ref{lem:prob estimate on hitting time of X},
we can derive probability estimates for the hitting times of $X\left(x,t\right)$
as
\begin{equation}
\mathbb{P}\left(\zeta_{y}^{X}\left(x\right)<\zeta_{0}^{X}\left(x\right)\right)=\frac{x}{y}\text{ and }\mathbb{P}\left(\zeta_{I}^{X}\left(x\right)\leq t\right)\leq\exp\left(-\frac{\left(I-x\right)^{2}}{4M_{\alpha,\beta}t}\right)\label{eq:prob estimate for hitting time of X two-sided bdry case}
\end{equation}
for every $0<x<y<I$ and $t>0$, where
\[
M_{\alpha,\beta}:=\max_{x\in\left[0,1\right]}x^{\alpha}\left(1-x\right)^{\beta}=\frac{\alpha^{\alpha}\beta^{\beta}}{\left(\alpha+\beta\right)^{\alpha+\beta}}.
\]
For every $0<G<I<H<1$, if $\left\{ \eta_{n}\left(x\right):n\in\mathbb{N}\right\} $
is the sequence of hitting times as in (\ref{eq:def of eta_n}) (for
the downward crossings of $X\left(x,t\right)$ from $I$ to $G$),
then for every $\left(x,y,t\right)\in\left(0,G\right)^{2}\times\left(0,\infty\right)$,
\begin{equation}
p\left(x,y,t\right)=p_{I}^{(L)}\left(x,y,t\right)+\sum_{n=1}^{\infty}\mathbb{E}\left[p_{I}^{(L)}\left(G,y,t-\eta_{2n}\left(x\right)\right);\eta_{2n}\left(x\right)\leq t,\eta_{2n}\left(x\right)<\zeta_{0,1}^{X}\left(x\right)\right],\label{eq:def of p two-sided bdry case left bdry}
\end{equation}
where the series on the right hand side is absolutely convergent. 

Let us rewrite the results in Theorem \ref{thm:main theorem} for
$p\left(x,y,t\right)$ found above.
\begin{prop}
\label{prop:property of p in two-sided bdry case} $p\left(x,y,t\right)$
is smooth on $\left(0,1\right)^{2}\times\left(0,\infty\right)$, and
for every $\left(x,y,t\right)\in\left(0,1\right)^{2}\times\left(0,\infty\right)$,
\begin{equation}
y^{\alpha}\left(1-y\right)^{\beta}p\left(x,y,t\right)=x^{\alpha}\left(1-x\right)^{\beta}p\left(y,x,t\right).\label{eq:symmetry of p two-sided}
\end{equation}
For every $y\in\left(0,1\right)$, $\left(x,t\right)\mapsto p\left(x,y,t\right)$
is a smooth solution to the Kolmogorov backward equation corresponding
to $L_{\alpha,\beta}$, i.e., 
\[
\partial_{t}p\left(x,y,t\right)=x^{\alpha}\left(1-x\right)^{\beta}\partial_{x}^{2}p\left(x,y,t\right);
\]
for every $x\in\left(0,1\right)$, $\left(y,t\right)\mapsto p\left(x,y,t\right)$
is a smooth solution to the Kolmogorov forward equation corresponding
to $L_{\alpha,\beta}$, i.e., 
\[
\partial_{t}p\left(x,y,t\right)=\partial_{y}^{2}\left(y^{\alpha}\left(1-y\right)^{\beta}p\left(x,y,t\right)\right).
\]

$p\left(x,y,t\right)$ is the fundamental solution to (\ref{eq: IVP 2sided boundary}).
Given $f\in C_{c}\left(\left(0,1\right)\right)$, 
\begin{equation}
u_{f}\left(x,t\right):=\int_{0}^{\infty}f\left(y\right)p\left(x,y,t\right)dy\text{ for }\left(x,t\right)\in\left(0,1\right)\times\left(0,\infty\right)\label{eq:def of u_f two-sided bdry case}
\end{equation}
is a smooth solution to (\ref{eq:IVP general equation}). Moreover,
for every $\left(x,t\right)\in\left(0,1\right)\times\left(0,\infty\right)$,
\begin{equation}
u_{f}\left(x,t\right)=\mathbb{E}\left[f\left(X\left(x,t\right)\right);t<\zeta_{0,1}^{X}\left(x\right)\right],\label{eq:prob interpretation of u_f two-sided bdry case}
\end{equation}
and hence for every Borel set $\Gamma\subseteq\left(0,1\right)$,
\begin{equation}
\int_{\Gamma}p\left(x,y,t\right)dy=\mathbb{P}\left(X\left(x,t\right)\in\Gamma,t<\zeta_{0,1}^{X}\left(x\right)\right)\label{eq:prob interpretation of p two-sided bdry case}
\end{equation}

Finally, $p\left(x,y,t\right)$ satisfies the Chapman-Kolmogorov equation,
i.e., for every $x,y\in\left(0,1\right)$ and $t,s>0$,
\begin{equation}
p\left(x,y,t+s\right)=\int_{0}^{\infty}p\left(x,\xi,t\right)p\left(\xi,y,s\right)d\xi.\label{eq:CK equation for p two-sided}
\end{equation}
\end{prop}

\begin{proof}
The only thing that requires proof is the smoothness of $p\left(x,y,t\right)$
on $\left(0,1\right)^{2}\times\left(0,\infty\right)$. By Theorem
\ref{thm:main theorem}, we know that $\left(y,t\right)\mapsto p\left(x,y,t\right)$
is smooth, and at the same time $\left(x,t\right)\mapsto p\left(x,y,t\right)$
solves the equation $\left(\partial_{t}-L_{\alpha,\beta}\right)p\left(x,y,t\right)=0$.
It is easy to see from here that $p\left(x,y,t\right)$ has all the
partial derivatives in $\left(x,y,t\right)$ of all orders. 
\end{proof}
The proposition above also leads to the wellposedness of the stochastic
differential equation associated with $L_{\alpha,\beta}$.
\begin{cor}
The stochastic differential equation (\ref{eq:SDE two-sided bdry case})
is well posed for every $x\in\left(0,1\right)$ up to the hitting
time at either $0$ or $1$ in the sense that if $\left\{ \tilde{X}\left(x,t\right):t\geq0\right\} $
is another solution to (\ref{eq:SDE two-sided bdry case}), then the
distribution of $X\left(x,t\right)$ conditioning on $t<\zeta_{0,1}^{X}\left(x\right)$
is identical with that of $\tilde{X}\left(x,t\right)$ given $t<\zeta_{0,1}^{\tilde{X}}\left(x\right)$.
\end{cor}

\begin{proof}
It is sufficient to observe that, for every $f\in C_{c}\left(\left(0,1\right)\right)$,
if $u_{f}\left(x,t\right)$ is defined as in (\ref{eq:def of u_f two-sided bdry case}),
then $\left\{ u_{f}\left(\tilde{X}\left(x,s\right),t-s\right):s\in\left[0,t\right]\right\} $
is a martingale, which, by (\ref{eq:prob interpretation of u_f two-sided bdry case}),
implies that 
\[
\mathbb{E}\left[f\left(X\left(x,t\right)\right);t<\zeta_{0,1}^{X}\left(x\right)\right]=u_{f}\left(x,t\right)=\mathbb{E}\left[f\left(\tilde{X}\left(x,t\right)\right);t<\zeta_{0,1}^{\tilde{X}}\left(x\right)\right].
\]
\end{proof}
Next we briefly discuss the other way of constructing $p\left(x,y,t\right)$,
which is to start with the localization of (\ref{eq: IVP 2sided boundary})
in a neighborhood of the right boundary $1$. It is easy to see that,
by exchanging $x$ and $1-x$, and at the same time exchanging $\alpha$
and $\beta$, we can follow the same steps as above to develop another
construction of the fundamental solution to (\ref{eq: IVP 2sided boundary}).
We will not repeat the details but only specify quantities and functions
that are necessary for the statement of the results. For example,
in this case the transformation is given by 
\[
z=\phi^{(R)}\left(x\right)=\frac{1}{4}\left(b^{(R)}\left(x\right)\right)^{2}\text{ where }b^{(R)}\left(x\right):=\int_{x}^{1}s^{-\alpha/2}\left(1-s\right)^{-\beta/2}ds;
\]
$q^{\left(R\right)}\left(z,w,t\right)$ is the fundamental solution
to the model equation with $\nu^{(R)}=\frac{1-\beta}{2-\beta}$, and
given $I\in\left(0,1\right)$, $q_{\phi^{(R)}\left(I\right)}^{(R)}\left(z,w,t\right)$
is the fundamental solution to the localization of the model equation
on $\left(0,\phi^{(R)}\left(I\right)\right)$; furthermore, we have
that
\[
\Theta\left(\phi^{(R)}\left(x\right)\right)=\frac{x^{\frac{\alpha}{4}}\left(1-x\right)^{\frac{\beta}{4}}}{\left(b^{(R)}\left(x\right)\right)^{\frac{\beta}{2\left(2-\beta\right)}}}\text{ with }\Theta_{I}^{(R)}=\left(\frac{2}{2-\beta}\right)^{\frac{\beta}{2\left(2-\beta\right)}}I^{-\frac{\alpha}{2\left(2-\beta\right)}},
\]
and for every $x\in\left(I,1\right)$,
\[
\left|V\left(\phi^{(R)}\left(x\right)\right)\right|\leq V_{I}^{(R)}\left(\phi^{(R)}\left(x\right)\right)^{-\frac{1-\beta}{2-\beta}}\text{ with }V_{I}^{(R)}=\frac{\alpha}{16}\left(4-\alpha+2\beta\right)I^{\frac{\alpha}{2-\beta}-2};
\]
we construct $q_{\phi^{(R)}\left(I\right)}^{V,(R)}\left(z,w,t\right)$
from $q_{\phi^{(R)}\left(I\right)}^{(R)}\left(z,w,t\right)$ via Duhamel's
method, and obtain the fundamental solution to (\ref{eq: IVP 2sided boundary})
localized on $\left(I,1\right)$ as
\[
p_{I}^{(R)}\left(x,y,t\right)=q_{\phi^{(R)}\left(I\right)}^{V,(R)}\left(\phi^{(R)}\left(x\right),\phi^{(R)}\left(y\right),t\right)\frac{x^{\frac{\alpha}{4}}\left(1-x\right)^{\frac{\beta}{4}}}{2y^{\frac{3\alpha}{4}}\left(1-y\right)^{\frac{3\beta}{4}}}\frac{\left(b^{(R)}\left(y\right)\right)^{\frac{4-\beta}{2\left(2-\beta\right)}}}{\left(b^{(R)}\left(x\right)\right)^{\frac{\beta}{2\left(2-\beta\right)}}};
\]
finally, if $\left\{ \tilde{\eta}_{n}\left(x\right):n\in\mathbb{N}\right\} $
is the sequence of hitting times that records the upward crossings
of $X\left(x,t\right)$ from $I$ to $H$. Then, (\ref{eq:prob estimate for hitting time of X two-sided bdry case})
and the strong Markov property of $\left\{ X\left(x,t\right):t\geq0\right\} $
are sufficient for us to obtain another version of the fundamental
solution, denoted by $\tilde{p}\left(x,y,t\right)$ temporarily, as
\begin{equation}
\begin{split}\tilde{p}\left(x,y,t\right) & =\lim_{G\searrow0}p_{G}^{(R)}\left(x,y,t\right)\\
 & =p_{I}^{(R)}\left(x,y,t\right)+\sum_{n=1}^{\infty}\mathbb{E}\left[p_{I}^{(R)}\left(H,y,t-\tilde{\eta}_{2n}\left(x\right)\right);\tilde{\eta}_{2n}\left(x\right)\leq t,\tilde{\eta}_{2n}\left(x\right)<\zeta_{0,1}^{X}\left(x\right)\right].
\end{split}
\label{eq:def of p two-sided bdry case right bdry}
\end{equation}
for $\left(x,y,t\right)\in\left(H,1\right)^{2}\times\left(0,\infty\right)$.
It is easy to see that $\tilde{p}\left(x,y,t\right)$ also satisfies
(\ref{eq:symmetry of p two-sided}), (\ref{eq:prob interpretation of p two-sided bdry case})
and (\ref{eq:CK equation for p two-sided}), which implies that $\tilde{p}\left(x,y,t\right)=p\left(x,y,t\right)$
almost everywhere on $\left(0,1\right)^{2}\times\left(0,\infty\right)$,
i.e., the two constructions of the fundamental solution to (\ref{eq: IVP 2sided boundary})
are consistent and $p\left(x,y,t\right)$ satisfies both (\ref{eq:def of p two-sided bdry case left bdry})
and (\ref{eq:def of p two-sided bdry case right bdry}). 

Depending on near which boundary we are conducting our analysis, we
can choose either (\ref{eq:def of p two-sided bdry case left bdry})
or (\ref{eq:def of p two-sided bdry case right bdry}) as the definition
of $p\left(x,y,t\right)$. For example, when both $x$ and $y$ are
close to one of the boundaries, we can develop approximations for
$p\left(x,y,t\right)$ similarly as in $\mathsection4.2$. 
\begin{cor}
\label{cor: approx for p(x,y,t) two-sided bdry case}For $\left(x,y,t\right)\in\left(0,1\right)^{2}\times\left(0,\infty\right)$,
set
\[
p^{(L)-approx.}\left(x,y,t\right):=q^{(L)}\left(\phi^{(L)}\left(x\right),\phi^{(L)}\left(y\right),t\right)\frac{x^{\frac{\alpha}{4}}\left(1-x\right)^{\frac{\beta}{4}}}{2y^{\frac{3\alpha}{4}}\left(1-y\right)^{\frac{3\beta}{4}}}\frac{\left(b^{(L)}\left(y\right)\right)^{\frac{4-\alpha}{2\left(2-\alpha\right)}}}{\left(b^{(L)}\left(x\right)\right)^{\frac{\alpha}{2\left(2-\alpha\right)}}}
\]
and 
\[
p^{(R)-approx.}\left(x,y,t\right)=q^{(R)}\left(\phi^{(R)}\left(x\right),\phi^{(R)}\left(y\right),t\right)\frac{x^{\frac{\alpha}{4}}\left(1-x\right)^{\frac{\beta}{4}}}{2y^{\frac{3\alpha}{4}}\left(1-y\right)^{\frac{3\beta}{4}}}\frac{\left(b^{(R)}\left(y\right)\right)^{\frac{4-\beta}{2\left(2-\beta\right)}}}{\left(b^{(R)}\left(x\right)\right)^{\frac{\beta}{2\left(2-\beta\right)}}}.
\]
Let $M^{(L)}\left(t\right)$ and, respectively, $M^{(R)}\left(t\right)$
be defined as in (\ref{eq:def of m_n(t)}) with $\nu=\nu^{(L)}$,
$V_{I}=V_{I}^{(L)}$ and, respectively, $\nu=\nu^{(R)}$, $V_{I}=V_{I}^{(R)}$.
Fix $0<G<I<H<1$, and set 
\[
t^{(L)}:=\left(\frac{4\left(2-\alpha\right)}{9\left(3-\alpha\right)}\phi^{(L)}\left(G\right)\right)\wedge\frac{\left(I-G\right)^{2}}{4M_{\alpha,\beta}}\text{ and }t^{(R)}:=\left(\frac{4\left(2-\beta\right)}{9\left(3-\beta\right)}\phi^{(R)}\left(H\right)\right)\wedge\frac{\left(H-I\right)^{2}}{4M_{\alpha,\beta}}.
\]
Then, for every $t\in\left(0,t^{(L)}\right)$ and $\left(x,y\right)\in\left(0,G\right)^{2}$
such that $\phi^{(L)}\left(x\right)\vee\phi^{(L)}\left(y\right)\leq\frac{1}{9}\phi^{(L)}\left(G\right)$,
\begin{align*}
\left|\frac{p\left(x,y,t\right)}{p^{(L)-approx.}\left(x,y,t\right)}-1\right| & \leq M^{(L)}\left(t\right)-1\\
 & \hspace{0.5cm}+\left[1+\left(\frac{2}{2-\alpha}\right)^{\frac{\alpha}{2-\alpha}}\frac{M^{(L)}\left(t\right)}{\left(1-G\right)^{\frac{2\beta}{2-\alpha}}}\left(\frac{G}{1-G}\wedge1\right)\right]\exp\left(-\frac{2\phi^{(L)}\left(G\right)}{9t}\right).
\end{align*}
Similarly, for every $t\in\left(0,t^{(R)}\right)$ and $\left(x,y\right)\in\left(H,1\right)^{2}$
such that $\phi^{(R)}\left(x\right)\vee\phi^{(R)}\left(y\right)\leq\frac{1}{9}\phi^{(R)}\left(H\right)$,
\begin{align*}
\left|\frac{p\left(x,y,t\right)}{p^{(R)-approx.}\left(x,y,t\right)}-1\right| & \leq M^{(R)}\left(t\right)-1\\
 & \hspace{0.5cm}+\left[1+\left(\frac{2}{2-\beta}\right)^{\frac{\beta}{2-\beta}}\frac{M^{(R)}\left(t\right)}{H^{\frac{2\alpha}{2-\beta}}}\left(\frac{1-H}{H}\wedge1\right)\right]\exp\left(-\frac{2\phi^{(R)}\left(H\right)}{9t}\right).
\end{align*}
\end{cor}

\begin{proof}
We only need to look at the statement involving $p^{(L)-approx.}\left(x,y,t\right)$.
There is not much to be done since a similar estimate (\ref{eq:estimate p/p^k-approx})
has been proven in Theorem \ref{thm:approximation of p(x,y,t) }.
We notice that $t^{(L)}$ is chosen such that the function $s\mapsto s^{\nu-2}\exp\left(-\frac{4\phi^{(L)}\left(G\right)}{9s}\right)$
is increasing on $\left(0,t^{(L)}\right)$, and $\left(I-G\right)^{2}\geq4M_{\alpha,\beta}t^{(L)}$.
Furthermore, in this case $S\left(x\right)=x$ for every $x\in\left(0,1\right)$,
and hence $\mathfrak{p}_{G}=G$. Combining the proof of Corollary
\ref{cor:approx for p(x,y,t) variation 1}, Corollary \ref{cor:approx for p(x,y,t) variation 2},
as well as (\ref{eq:Theta_I(L) two-sided bdry case}) and (\ref{eq:range of phi^(L)})
in the Appendix, we get that for every $\left(x,y,t\right)\in\left(0,G\right)^{2}\times\left(0,t^{(L)}\right)$
as described in the statement, 
\begin{align*}
\left|\frac{p\left(x,y,t\right)-p_{I}^{(L)}\left(x,y,t\right)}{p^{(L)-approx.}\left(x,y,t\right)}\right| & \leq M^{(L)}\left(t\right)\frac{\Theta\left(\phi^{(L)}\left(G\right)\right)}{\Theta\left(\phi^{(L)}\left(x\right)\right)}\left(\frac{\phi^{(L)}\left(G\right)}{\phi^{(L)}\left(x\right)}\right)^{\frac{1}{2-\alpha}}\frac{x}{G}\exp\left(-\frac{2\phi^{(L)}\left(G\right)}{9t}\right)\frac{G}{1-G}\\
 & \le\left(\frac{2}{2-\alpha}\right)^{\frac{\alpha}{2-\alpha}}\left(1-G\right)^{-\frac{2\beta}{2-\alpha}}M^{(L)}\left(t\right)\exp\left(-\frac{2\phi^{(L)}\left(G\right)}{9t}\right)\left(\frac{G}{1-G}\wedge1\right).
\end{align*}
\end{proof}

\section{Appendix}

This Appendix contains detailed derivations involving $\Theta\left(z\right)$
and $V\left(z\right)$ for $z\in\left(0,J\right)$. Assuming that
$J=\phi\left(I\right)$, it is sufficient for us to look at $\Theta\left(\phi\left(x\right)\right)$
and $V\left(\phi\left(x\right)\right)$ for $x\in\left(0,I\right)$,
where the notations become simpler. Recall that 
\[
a_{I}:=\max_{x\in\left[0,I\right]}\left\{ \frac{1}{a\left(x\right)},a\left(x\right)\right\} \text{ and }b_{I}:=\max_{x\in\left[0,I\right]}\left|b\left(x\right)\right|.
\]
We also introduce two more notations:
\[
a_{I}^{\prime}:=\max_{x\in\left[0,I\right]}\left|a^{\prime}\left(x\right)\right|\text{ and }b_{I}^{\prime}:=\max_{x\in\left[0,I\right]}\left|b^{\prime}\left(x\right)\right|.
\]

According to (\ref{eq:def of phi =000026 theta}) and (\ref{eq:def of Theta}),
we have that for every $x\in\left(0,I\right)$,
\[
\begin{split}\Theta\left(\phi\left(x\right)\right) & =\exp\left(-\int_{0}^{x}\frac{\theta(w)}{2\phi\left(w\right)}\phi^{\prime}\left(w\right)dw\right)\\
 & =\exp\left(-\int_{0}^{x}\left(\frac{\frac{1}{2}-\nu}{2\sqrt{\phi\left(w\right)w^{\alpha}a\left(w\right)}}-\frac{\left(w^{\alpha}a\left(w\right)\right)^{\prime}}{4w^{\alpha}a\left(w\right)}+\frac{b\left(w\right)}{2w^{\alpha}a\left(w\right)}\right)dw\right).
\end{split}
\]
Notice that 
\[
\frac{\frac{1}{2}-\nu}{2\sqrt{\phi\left(w\right)w^{\alpha}a\left(w\right)}}-\frac{\left(w^{\alpha}a\left(w\right)\right)^{\prime}}{4w^{\alpha}a\left(w\right)}=\left(\ln\frac{\left(2\sqrt{\phi\left(w\right)}\right)^{\frac{1}{2}-\nu}}{\left(w^{\alpha}a\left(w\right)\right)^{\frac{1}{4}}}\right)^{\prime},
\]
and further, if $\alpha=1$, then 
\[
\frac{b\left(w\right)}{2wa\left(w\right)}=\left(\ln\left(w^{\frac{b\left(0\right)}{2}}\right)\right)^{\prime}+\frac{1}{2w}\left(\frac{b\left(w\right)}{a\left(w\right)}-b\left(0\right)\right).
\]
Plugging these two expressions back into the right hand side of $\Theta\left(\phi\left(x\right)\right)$
leads to 
\[
\Theta\left(\phi\left(x\right)\right)=\begin{cases}
x^{\frac{\alpha}{4}}\left(4\phi\left(x\right)\right)^{-\frac{\alpha}{4\left(2-\alpha\right)}}\left(a\left(x\right)\right)^{\frac{1}{4}}\exp\left(-\int_{0}^{x}\frac{b\left(w\right)}{2w^{\alpha}a\left(w\right)}dw\right) & \text{ if }\alpha\neq1,\\
\left(\frac{x}{4\phi\left(x\right)}\right)^{\frac{1}{4}-\frac{b\left(0\right)}{2}}\left(a\left(x\right)\right)^{\frac{1}{4}}\exp\left(-\int_{0}^{x}\frac{1}{2w}\left(\frac{b\left(w\right)}{a\left(w\right)}-b\left(0\right)\right)dw\right) & \text{ if }\alpha=1,
\end{cases}
\]
which is exactly (\ref{eq:formula of Theta}). Given \textbf{(H1)
}and \textbf{(H2)}, the integral in the exponential function above
is well defined in both cases (when $\alpha=1$ and $\alpha\neq1$). 

With the notations introduced above, we have that when $\alpha\neq1$,
for every $x\in\left(0,I\right)$,
\[
\left(1-\frac{\alpha}{2}\right)^{\frac{\alpha}{2\left(2-\alpha\right)}}a_{I}^{-\frac{1}{2\left(2-\alpha\right)}}\leq\frac{x^{\frac{\alpha}{4}}\left(a\left(x\right)\right)^{\frac{1}{4}}}{\left(4\phi\left(x\right)\right)^{\frac{\alpha}{4\left(2-\alpha\right)}}}\leq\left(1-\frac{\alpha}{2}\right)^{\frac{\alpha}{2\left(2-\alpha\right)}}a_{I}^{\frac{1}{2\left(2-\alpha\right)}}
\]
and 
\[
\exp\left(\int_{0}^{x}\frac{\left|b\left(w\right)\right|}{2w^{\alpha}a\left(w\right)}dw\right)\leq\mathbb{I}_{\left(0,1\right)}\left(\alpha\right)\cdot e^{\frac{a_{I}b_{I}I^{1-\alpha}}{2\left(1-\alpha\right)}}+\mathbb{I}_{\left(1,2\right)}\left(\alpha\right)\cdot e^{\frac{a_{I}b_{I}^{\prime}I^{2-\alpha}}{2\left(2-\alpha\right)}};
\]
when $\alpha=1$, for every $x\in\left(0,I\right)$,
\[
2^{b\left(0\right)-\frac{1}{2}}a_{I}^{\frac{1}{2}b\left(0\right)-\frac{1}{2}}\leq\left(\frac{x}{4\phi\left(x\right)}\right)^{\frac{1}{4}-\frac{b\left(0\right)}{2}}\left(a\left(x\right)\right)^{\frac{1}{4}}\leq2^{b\left(0\right)-\frac{1}{2}}a_{I}^{\frac{1}{2}-\frac{1}{2}b\left(0\right)}
\]
and 
\[
\exp\left(\int_{0}^{x}\frac{1}{2w}\left|\frac{b\left(w\right)}{a\left(w\right)}-b\left(0\right)\right|dw\right)\leq e^{\frac{1}{2}a_{I}^{2}\left(a_{I}b_{I}^{\prime}+a_{I}^{\prime}b_{I}\right)I}.
\]
Hence, if we set
\[
A_{I}:=\mathbb{I}_{\left(0,1\right)}\left(\alpha\right)\cdot e^{\frac{a_{I}b_{I}I^{1-\alpha}}{2\left(1-\alpha\right)}}+\mathbb{I}_{\left\{ 1\right\} }\left(\alpha\right)\cdot e^{\frac{1}{2}a_{I}^{2}\left(a_{I}b_{I}^{\prime}+a_{I}^{\prime}b_{I}\right)I}+\mathbb{I}_{\left(1,2\right)}\left(\alpha\right)\cdot e^{\frac{a_{I}b_{I}^{\prime}I^{2-\alpha}}{2\left(2-\alpha\right)}}.
\]
then for every $x\in\left(0,I\right)$,
\[
\left(1-\frac{\alpha}{2}\right)^{\frac{1}{2}-\nu}a_{I}^{-\frac{1-\nu}{2}}A_{I}^{-1}\leq\Theta\left(\phi\left(x\right)\right)\leq\left(1-\frac{\alpha}{2}\right)^{\frac{1}{2}-\nu}a_{I}^{\frac{1-\nu}{2}}A_{I}.
\]
(\ref{eq:bound Theta_J}) follows from here by setting 
\begin{equation}
\Theta_{I}:=\left(\left(1-\frac{\alpha}{2}\right)^{\nu-\frac{1}{2}}\vee\sqrt{2}\right)a_{I}^{\frac{1-\nu}{2}}A_{I}.\label{eq:formula of Theta_J}
\end{equation}
For every $x,y\in\left(0,I\right)$, we can follow the arguments above
to get that
\[
a_{I}^{-\frac{1-\nu}{2}}A_{I}^{-1}\leq\frac{\Theta\left(\phi\left(x\right)\right)}{\Theta\left(\phi\left(y\right)\right)}\leq a_{I}^{\frac{1-\nu}{2}}A_{I}.
\]
Moreover, if $S\left(x\right)$ is as defined in (\ref{eq:def of S(x)}),
then
\[
S\left(x\right)=2^{2\nu-1}\int_{0}^{x}\exp\left(-\int_{0}^{u}\frac{b\left(w\right)}{w^{\alpha}a\left(w\right)}dw\right)du.
\]
In addition, from (\ref{eq:formula of Theta(phi)}), we can easily
derive that, for every $x,y\in\left(0,I\right)$,
\begin{equation}
\begin{split}\frac{\Theta\left(\phi\left(x\right)\right)}{\Theta\left(\phi\left(y\right)\right)}\phi^{\prime}\left(y\right) & =\left(\frac{\phi\left(y\right)}{\phi\left(x\right)}\right)^{\frac{1}{4}-\frac{\nu}{2}}\frac{\phi^{\frac{1}{2}}\left(y\right)x^{\frac{\alpha}{4}}a^{\frac{1}{4}}\left(x\right)}{y^{\frac{3\alpha}{4}}a^{\frac{3}{4}}\left(y\right)}\exp\left(-\int_{y}^{x}\frac{b\left(w\right)}{2w^{\alpha}a\left(w\right)}dw\right),\end{split}
\label{eq:transformation factor between q^V and p}
\end{equation}
and 
\begin{equation}
\begin{split}\frac{\left(\phi\left(x\right)\right)^{1-\nu}}{\phi^{\prime}\left(x\right)}\Theta^{2}\left(\phi\left(x\right)\right) & =\frac{\left(\phi\left(x\right)\right)^{\frac{1}{2}-\nu}x^{\alpha}a\left(x\right)}{\left(4\phi\left(x\right)\right)^{\frac{\alpha}{2\left(2-\alpha\right)}}}\exp\left(-\int_{0}^{x}\frac{b\left(w\right)}{w^{\alpha}a\left(w\right)}dw\right)\\
 & =2^{2\nu-1}x^{\alpha}a\left(x\right)\exp\left(-\int_{0}^{x}\frac{b\left(w\right)}{w^{\alpha}a\left(w\right)}dw\right).
\end{split}
\label{eq:symmetry factor for p}
\end{equation}

Now we move onto $V\left(z\right)$ and recall from (\ref{eq:formula of V(phi)})
that for every $x\in\left(0,I\right)$,
\[
V\left(\phi\left(x\right)\right)=\frac{\theta\left(x\right)}{4\phi\left(x\right)}\left(-\theta\left(x\right)+2-2\nu\right)-\frac{\theta^{\prime}\left(x\right)}{2\phi^{\prime}\left(x\right)}\text{ for every }x\in\left(0,I\right).
\]
By (\ref{eq:def of phi =000026 theta}), the choice of $\nu$ and
\textbf{(H1)} and \textbf{(H2)}, it is straightforward to verify that
when $x\in\left(0,I\right)$,
\[
\phi\left(x\right)=\frac{x^{2-\alpha}}{\left(2-\alpha\right)^{2}}\left(1+\mathcal{O}\left(x\right)\right)\text{ and }\theta\left(x\right)=\frac{b\left(0\right)x^{1-\alpha}}{2-\alpha}\mathbb{I}_{\left(0,1\right)}\left(\alpha\right)+\mathcal{O}\left(x^{2-\alpha}\right),
\]
which implies that
\[
\begin{split}\frac{\theta\left(x\right)}{\phi\left(x\right)} & =\frac{\left(2-\alpha\right)b\left(0\right)}{x}\mathbb{I}_{\left(0,1\right)}\left(\alpha\right)+\mathcal{O}\left(1\right).\end{split}
\]
In addition, we also have that
\[
\begin{split}\frac{\theta^{\prime}\left(x\right)}{\phi^{\prime}\left(x\right)} & =b^{\prime}\left(x\right)-\frac{\left(x^{\alpha}a\left(x\right)\right)^{\prime\prime}}{2}+\frac{2b\left(x\right)-\left(x^{\alpha}a\left(x\right)\right)^{\prime}}{2}\left(\frac{1}{2\sqrt{\phi\left(x\right)x^{\alpha}a\left(x\right)}}-\frac{\left(x^{\alpha}a\left(x\right)\right)^{\prime}}{2x^{\alpha}a\left(x\right)}\right)\\
 & =b^{\prime}\left(x\right)+\frac{b\left(x\right)}{2\sqrt{\phi\left(x\right)x^{\alpha}a\left(x\right)}}-\frac{b\left(x\right)\left(x^{\alpha}a\left(x\right)\right)^{\prime}}{2x^{\alpha}a\left(x\right)}-\frac{\left(x^{\alpha}a\left(x\right)\right)^{\prime\prime}}{2}-\frac{\left(x^{\alpha}a\left(x\right)\right)^{\prime}}{4\sqrt{\phi\left(x\right)x^{\alpha}a\left(x\right)}}+\frac{\left(\left(x^{\alpha}a\left(x\right)\right)^{\prime}\right)^{2}}{4x^{\alpha}a\left(x\right)}.
\end{split}
\]
We notice that 
\[
-\frac{\left(x^{\alpha}a\left(x\right)\right)^{\prime\prime}}{2}-\frac{\left(x^{\alpha}a\left(x\right)\right)^{\prime}}{4\sqrt{\phi\left(x\right)x^{\alpha}a\left(x\right)}}+\frac{\left(\left(x^{\alpha}a\left(x\right)\right)^{\prime}\right)^{2}}{4x^{\alpha}a\left(x\right)}=\mathcal{O}\left(x^{\alpha-1}\right),
\]
and 
\[
\frac{b\left(x\right)}{2\sqrt{\phi\left(x\right)x^{\alpha}a\left(x\right)}}-\frac{b\left(x\right)\left(x^{\alpha}a\left(x\right)\right)^{\prime}}{2x^{\alpha}a\left(x\right)}=\mathbb{I}_{\left(0,1\right)}\left(\alpha\right)\left(\frac{b\left(0\right)\left(1-\alpha\right)}{x}+\mathcal{O}\left(x^{\alpha-1}\right)\right)+\mathcal{O}\left(1\right).
\]
Putting all the above together yields that when $\alpha\in\left(0,1\right)$,
\[
V\left(\phi\left(x\right)\right)=\frac{\left(1-\nu\right)\left(2-\alpha\right)b\left(0\right)}{2x}-\frac{\left(1-\alpha\right)b\left(0\right)}{2x}+\mathcal{O}\left(x^{\alpha-1}\right)=\frac{\alpha b\left(0\right)}{2x}+\mathcal{O}\left(x^{\alpha-1}\right);
\]
when $\alpha\in[1,2)$, $V\left(\phi\left(x\right)\right)$ is bounded
for $x\in\left(0,I\right)$. Thus, we have proven all the claims in
Lemma \ref{lem:estimates on VJ}.

Next, we look at the case when $b\left(x\right)\equiv0$, where most
of the expressions above take simpler forms. For example, 
\[
\Theta\left(\phi\left(x\right)\right)=\frac{x^{\frac{\alpha}{4}}a^{\frac{1}{4}}\left(x\right)}{2^{\frac{\alpha}{2\left(2-\alpha\right)}}\left(\phi\left(x\right)\right)^{\frac{\alpha}{4\left(2-\alpha\right)}}},\;V\left(\phi\left(x\right)\right)=-\frac{\alpha\left(\alpha-4\right)}{16\left(2-\alpha\right)^{2}\phi\left(x\right)}+\frac{\left(x^{\alpha}a\left(x\right)\right)^{\prime\prime}}{4}-\frac{3\left(\left(x^{\alpha}a\left(x\right)\right)^{\prime}\right)^{2}}{16},
\]
\[
\frac{\Theta\left(\phi\left(x\right)\right)}{\Theta\left(\phi\left(y\right)\right)}\phi^{\prime}\left(y\right)=\left(\frac{\phi\left(y\right)}{\phi\left(x\right)}\right)^{\frac{\alpha}{4\left(2-\alpha\right)}}\frac{\phi^{\frac{1}{2}}\left(y\right)x^{\frac{\alpha}{4}}a^{\frac{1}{4}}\left(x\right)}{y^{\frac{3\alpha}{4}}a^{\frac{3}{4}}\left(y\right)}\text{ and }\frac{\left(\phi\left(x\right)\right)^{1-\nu}}{\phi^{\prime}\left(x\right)}\Theta^{2}\left(\phi\left(x\right)\right)=2^{-\frac{\alpha}{2-\alpha}}x^{\alpha}a\left(x\right).
\]
In particular, if $a\left(x\right)=\left(1-x\right)^{\beta}$ as in
$\mathsection5$, then
\begin{equation}
\Theta\left(\phi\left(x\right)\right)=\frac{x^{\frac{\alpha}{4}}\left(1-x\right)^{\frac{\beta}{4}}}{2^{\frac{\alpha}{2\left(2-\alpha\right)}}}\left(\int_{0}^{x}\frac{ds}{\sqrt{s^{\alpha}\left(1-s\right)^{\beta}}}\right)^{\frac{-\alpha}{4\left(2-\alpha\right)}}\text{ with }\Theta_{I}=\frac{\left(\frac{2}{2-\alpha}\right)^{\frac{\alpha}{2\left(2-\alpha\right)}}}{\left(1-I\right)^{\frac{\beta}{2\left(2-\alpha\right)}}}.\label{eq:Theta_I(L) two-sided bdry case}
\end{equation}
Furthermore,
\[
\begin{split}V\left(\phi\left(x\right)\right) & =\frac{\alpha\left(4-\alpha\right)}{4\left(2-\alpha\right)^{2}}\left(\int_{0}^{x}\frac{ds}{\sqrt{s^{\alpha}\left(1-s\right)^{\beta}}}\right)^{-2}+\frac{\alpha\left(\alpha-4\right)}{16}x^{\alpha-2}\left(1-x\right)^{\beta}\\
 & \hspace{1cm}\hspace{1cm}\hspace{1cm}-\frac{\alpha\beta}{8}x^{\alpha-1}\left(1-x\right)^{\beta-1}+\frac{\beta\left(\beta-4\right)}{16}x^{\alpha}\left(1-x\right)^{\beta-2}.
\end{split}
\]
Since 
\begin{equation}
\frac{x^{2-\alpha}}{\left(2-\alpha\right)^{2}}\leq\phi\left(x\right)\leq\left(1-x\right)^{-\beta}\frac{x^{2-\alpha}}{\left(2-\alpha\right)^{2}},\label{eq:range of phi^(L)}
\end{equation}
we see that
\[
V\left(\phi\left(x\right)\right)\geq x^{\alpha-2}\left(1-x\right)^{\beta-2}\left(-\frac{\alpha\beta x\left(1-x\right)}{8}+\frac{\beta\left(\beta-4\right)}{16}x^{2}\right)
\]
and
\begin{align*}
V\left(\phi\left(x\right)\right) & \leq\frac{\alpha\left(4-\alpha\right)}{16}x^{\alpha-2}\left[1-\left(1-x\right)^{\beta}\right]\\
 & \hspace{1cm}\hspace{1cm}+x^{\alpha-2}\left(1-x\right)^{\beta-2}\left(-\frac{\alpha\beta x\left(1-x\right)}{8}+\frac{\beta\left(\beta-4\right)}{16}x^{2}\right)\\
 & \leq x^{\alpha-1}\left(1-x\right)^{\beta-2}\left(\frac{\alpha\left(4-\alpha\right)\beta}{16}-\frac{\alpha\beta\left(1-x\right)}{8}+\frac{\beta\left(\beta-4\right)}{16}x^{2}\right).
\end{align*}
where in the last line we used the fact that for every $x\in\left(0,I\right)$,
\[
1-\left(1-x\right)^{\beta}\leq\beta x\left(1-x\right)^{\beta-2}.
\]
Combining the upper bound and the lower bound of $V\left(\phi\left(x\right)\right)$
leads to 
\[
\left|V\left(\phi\left(x\right)\right)\right|\leq x^{\alpha-1}\left(1-x\right)^{\beta-2}\frac{\beta}{16}\left(4-\beta+2\alpha\right)\text{ for every }x\in\left(0,I\right),
\]
which, by (\ref{eq:range of phi^(L)}), implies that when $\alpha\in\left(0,1\right)$,
\[
\left|V\left(\phi\left(x\right)\right)\right|\leq\frac{\beta}{16}\left(4-\beta+2\alpha\right)\left(2-\alpha\right)^{\frac{2\alpha-2}{2-\alpha}}\left(1-I\right)^{\frac{\beta}{2-\alpha}-2}\left(\phi\left(x\right)\right)^{-\frac{1-\alpha}{2-\alpha}}\text{ for every }x\in\left(0,I\right).
\]
Therefore, with this specific case of $a\left(x\right)=\left(1-x\right)^{\beta}$,
we see that the constant $V_{I}$ as introduced Lemma \ref{lem:estimates on VJ}
(identified with $V_{J}$ in for $J=\phi\left(I\right)$) can be taken
as 
\begin{equation}
V_{I}=\frac{\beta}{16}\left(4-\beta+2\alpha\right)\left(1-I\right)^{\frac{\beta}{2-\alpha}-2}.\label{eq:VI in two-sided bdry case}
\end{equation}

\bibliographystyle{plain}
\bibliography{mybib}

\end{document}